\documentclass[arxiv,reqno,twoside,a4paper,12pt]{amsart}

\usepackage{tikz}
\usepackage{amsmath, verbatim}
\usepackage{amssymb,amsfonts, mathrsfs, mathtools}
\usepackage[colorlinks=true,linkcolor=blue,citecolor=blue]{hyperref}
\usepackage[driver=pdftex,heightrounded=true,centering]{geometry}
\usepackage{longtable, slashed}
\usepackage{tabularx}
\usepackage{multirow}
\usepackage{caption}
\usepackage{subcaption}
\usepackage{latexsym,enumitem}
\usepackage[all]{xy}

\usepackage[latin2]{inputenc}
\usepackage[mathscr]{eucal}
\usepackage{t1enc, indentfirst}
\usepackage{graphicx, graphics}
\usepackage{pict2e, epic, amssymb }
\usepackage{tikz, pst-node}
\usepackage{tikz-cd, pgfplots}
\usetikzlibrary{arrows}
\usetikzlibrary{calc, patterns}
\usepackage{textcomp}
\pgfplotsset{compat=1.9}
\usepackage{epstopdf}
\usepackage{marvosym}
\usepackage{dsfont}

\usepackage{mathdg2015-GriTalVer}

\theoremstyle{plain}
\newtheorem{Th}{Theorem}[section]
\newtheorem{Lemma}[Th]{Lemma}
\newtheorem{Cor}[Th]{Corollary}
\newtheorem{Prop}[Th]{Proposition}
\newtheorem{theorem}[Th]{Theorem}
\newtheorem{assum}[Th]{Assumption}

 \theoremstyle{definition}
\newtheorem{defn}[Th]{Definition}

\newtheorem{Rem}[Th]{Remark}
\newtheorem{?}[Th]{Problem}

\newtheorem{Exs}[Th]{Examples}

\newcommand{\Hom}{{\rm{Hom}}}

\newcommand{\R}{{\mathbb{R}}}

\numberwithin{equation}{section}

\definecolor{qqwuqq}{rgb}{0,0,0}

\definecolor{darkgreen}{rgb}{0.01, 0.75, 0.24}

\newcommand{\bb}{{\mathrm{b}}}
\newcommand{\specb}{\textup{spec}_\bb}
\newcommand{\w}{\omega}
\newcommand{\W}{\Omega}
\newcommand{\U}{\mathscr{U}}
\newcommand{\phiT}{{}^\phi T}
\newcommand{\scT}{{}^{\mathrm{sc}}T}
\newcommand{\dM}{{\partial M}}

\newcommand{\Piperp}{\Pi^\perp}
\newcommand{\ssc}{{\mathrm{sc}}}
\newcommand{\zf}{{\mathrm{zf}}}

\newcommand{\Hsplit}{H_{\mathscr{H}}}
\renewcommand{\Mbar}{{\overline{M}}}
\renewcommand{\Psibar}{\overline{\Psi}}

\newcommand{\Umath}{{\mathscr{U}}}
\newcommand{\Umathbar}{{\overline{\mathscr{U}}}}
\newcommand{\Cmath}{{\mathscr{C}}}
\newcommand{\Hmath}{{\mathscr{H}}}
\newcommand{\calEwhat}{{\widehat{\calE}}}
\newcommand{\calFwhat}{{\widehat{\calF}}}
\newcommand{\Kwhat}{{\widehat{K}}}
\newcommand{\Ewhat}{{\widehat{E}}}
\newcommand{\Fwhat}{{\widehat{F}}}

\renewcommand{\lf}{{\mathrm{lb}}}
\renewcommand{\rf}{{\mathrm{rb}}}

\renewcommand{\ll}{{\mathrm{l}}}
\newcommand{\rr}{{\mathrm{r}}}

\newcommand{\ocup}{\,\overline{\cup}\,}

\renewcommand{\Atilde}{{\widetilde{A}}}

\renewcommand{\calH}{{\mathscr{H}}}

\newcommand{\codim}{\operatorname{codim}}

\newcommand{\lfz}{{\mathrm{lb}_0}}
\newcommand{\rfz}{{\mathrm{rb}_0}}
\newcommand{\bbfz}{{\bbf_0}}

\newcommand{\Diag}{\operatorname{Diag}}

\renewcommand{\interior}[1]{\mathrm{int}(#1)}

\date{\today}

\begin{document}

\title[Resolvent at low energy]{Spectral geometry on manifolds with fibred boundary metrics I: 
Low energy resolvent}

\author{Daniel Grieser}
	\address{Universit\"at Oldenburg, Germany}
	\email{daniel.grieser@uni-oldenburg.de}

\author{Mohammad Talebi}
	\address{Universit\"at Oldenburg, Germany}
	\email{mohammad.talebi@uni-oldenburg.de}

	\author{Boris Vertman}
	\address{Universit\"at Oldenburg, Germany}
	\email{boris.vertman@uni-oldenburg.de}
	

\maketitle

\begin{abstract}
We study the low energy resolvent of the Hodge Laplacian on a manifold  equipped with a fibred boundary metric. We determine the precise asymptotic behavior of the resolvent as a fibred boundary (aka $\phi$-) pseudodifferential operator when the resolvent parameter tends to zero.
This generalizes previous work by Guillarmou and Sher who considered asymptotically conic metrics, which correspond to the special case when the fibres are points. The new feature in the case of non-trivial fibres is that the resolvent has different asymptotic behavior on the subspace of forms that are fibrewise harmonic  and on its orthogonal complement. To deal with this, we introduce an appropriate 'split' pseudodifferential calculus, building on and extending work by Grieser and Hunsicker. Our work sets the basis for the discussion of spectral invariants on $\phi$-manifolds. 
\end{abstract}

\setcounter{tocdepth}{1}
\tableofcontents

\section{Introduction and statement of the main results}

\subsection{Setting of manifolds with fibred boundary metrics}

Consider a compact smooth manifold $\Mbar$
with interior $M$ and boundary $\partial \Mbar$
(also denoted here by $\partial M$),
which is the total space of a fibration
 $\phi: \partial M \to B$ over a closed manifold $B$ with fibres given by 
copies of a closed manifold $F$. We fix a boundary defining function 
$x:\Mbar\to\R^+=[0,\infty)$, i.e. $x^{-1}(0)=\dM$ and $dx\neq0$ at $\dM$, and a trivialization $\Umathbar\cong [0,\eps)_x\times\dM$ of an open
neighborhood $\Umathbar$ of $\partial M$. We consider a $\phi$-metric (also called fibred boundary metric) on $M$, i.e. a Riemannian metric $g_\phi$ which on $\Umath \cong (0,\varepsilon)_x\times\partial M$ has the form
\begin{equation}\label{g-phi-def}
g_{\phi} \restriction \Umath = g_0 + h\,,\quad g_0=\frac{dx^{2}}{x^4} + \frac{\phi^{*}g_{B}}{x^{2}} + g_{F}
\end{equation}
where $g_{B}$ is a Riemannian metric on the base $B$,  $g_{F}$ is a symmetric bilinear form on the 
total space $\partial M$, restricting to Riemannian metrics on the fibres $F$, and the perturbation $h$ is a two tensor on $\Umath$ satisfying the bound given in Assumption \ref{assum1} below.
If  $h\equiv 0$ then
$g_\phi$ is called an exact $\phi$-metric. 

\begin{Exs}
There are various examples where such metrics arise naturally. 
The natural metric on the moduli space of non-abelian magnetic monopoles
of charge $2$ is a $\phi$-metric, cf. \cite{gravi ins, kottke-singer, fritsch}. Other examples include
gravitational instantons, i.e. complete hyperk\"ahler $4$-manifolds, as well as products
of scattering with closed manifolds.
\end{Exs}
 
Recall that the Hodge Laplace operator associated to  $g_\phi$ is $\Delta_{\phi} = (d+d^*)^2$ 
where $d$ denotes the exterior derivative and
$d^*$ its adjoint. This is a self-adjoint and non-negative operator in $L^2(M,\Lambda T^*M; \dvol_\phi)$, 
where $\dvol_\phi$ is the volume form for $g_\phi$, so the resolvent
$$ (\Delta_{\phi}+k^2)^{-1}$$
is defined for all $k\neq0$. We always take $k>0$. The objective of this paper is 
the precise analysis of the Schwartz kernel of the resolvent, in particular of its asymptotic behavior as $k\to0$.
More generally, our results apply to the Hodge Laplacian acting on sections of a vector bundle over $\Mbar$ equipped with a flat connection. \medskip

If the fibration $\phi$ is trivial, $\partial M=B\times F$ has  the product metric, and
$h\equiv 0$ then $\Delta_{\phi}$ acting on functions is, over $\Umath$, given explicitly by
\begin{equation}\label{Delta}
\Delta_{\phi} \restriction \Umath = -x^{4}\partial_{x}^{2} + x^{2}\Delta_{B} + \Delta_{F}  - (2-b)x^{3}\partial_{x},
\end{equation}
where $b= \dim B$, $\Delta_B$ is the Laplace Beltrami
 operator of $(B,g_{B})$ and 
$\Delta_{F}$ is the Laplace Beltrami operator on $(F,g_{F})$. 
The full Hodge Laplacian admits a similar structure. 
In the non-product case we recover under additional assumptions a similar structure, 
where $\Delta_{F}$ is now replaced by a family of Hodge Laplacians $\Delta_{F_y}$, associated to 
the family of Riemannian metrics $g_F(y)$ on $F$, with parameter $y\in B$. We will be explicit below.
\medskip

The operator $\Delta_{\phi}+k^2$ is, for any parameter $k\geq0$, an elliptic element in a general class of differential operators having a structure similar to \eqref{Delta}, called $\phi$-differential operators. Mazzeo and Melrose in \cite{mazzeo1998pseudodifferential} developed a pseudo-differential calculus, denoted by $\Psi^*_\phi$, which contains parametrices for such operators $P$. Therein they gave a stronger condition, called full ellipticity, that 
is equivalent to $P$ being Fredholm between naturally associated Sobolev spaces, and is necessary for $P$ 
to have an inverse in the calculus. The operator $\Delta_{\phi}+k^2$ is fully elliptic for $k>0$, but not for $k=0$ unless $\Delta_{F_y}$ is invertible for all $y\in B$, which may happen in the case of the Hodge Laplacian twisted by a flat vector bundle. 
In that case, the behavior of the resolvent can be deduced directly from 
\cite{mazzeo1998pseudodifferential}. The contribution of our paper is that we allow the fibrewise Laplacians to have a non-trivial kernel.
\medskip

If the fibre $F$ is a point then fibred boundary metrics are called scattering or asymptotically conic metrics. In this case the low energy resolvent has been studied before, with application to the boundedness of Riesz transforms:
for the scalar Laplacian this was first done in the asymptotically Euclidean case, i.e.\ for $B= \mathbb{S}^b$, by
Carron, Coulhon and Hassell \cite{carron2006riesz}, then for general base and using a different construction of the resolved double space (see below) by  Guillarmou and Hassell 
\cite{guillarmou2008resolvent, guillarmou2009resolvent}, who also allow a potential.
Guillarmou and Sher \cite{guillarmou2014low} then used the calculus of \cite{guillarmou2008resolvent} to analyze the low energy resolvent of the Hodge Laplacian and used this to study the behavior of analytic torsion under conic degeneration.
\medskip

\subsection{Assumptions and the main result}

We study the Schwartz kernel of the resolvent $(\Delta_{\phi} + k^{2})^{-1}$ as $k\to 0$ under 
additional assumptions that we shall list here. We shall also explain why we make these assumptions.

\begin{assum}\label{assum1}
The higher order term $h$ satisfies $\vert h \vert_{g_{0}} = O(x^3)$ as $x \to 0$.
\end{assum}
 
\noindent This assumption 
guarantees that the terms resulting from $h$ do not alter the leading order behaviour of the Hodge Laplacian, more precisely of the terms $P_{ij}$ in \eqref{eqn:square_phi-prelim}.

\begin{assum}\label{assum2}
We assume that $\phi: (\partial M, g_F + \phi^*g_B) 
\to (B, g_B)$ is a Riemannian submersion.
\end{assum}

\noindent This assumption is used in order to obtain the expression 
\eqref{hodge de rham} for the Hodge de Rham operator.
 \medskip

The next assumption requires some preparation: 
The spaces $\mathscr{H}_y=\ker \Delta_{F_y}$ of harmonic forms on the fibres have 
finite dimension independent of the base point $y \in B$, since they are isomorphic 
to the cohomology of $F$. It is a standard fact that these spaces then form a vector bundle $\mathscr{H}$ 
over $B$. The metric $g$ induces a flat connection on the bundle $\mathscr{H}$, see 
\cite[Proposition 15]{gravi ins}. We denote the twisted Gauss-Bonnet operator on $B$ 
with values in this flat vector bundle by $D_B = \mathfrak{d} + \mathfrak{d}^*$. 
Its definition will be recalled explicitly in 
\eqref{DB}. Write $N_B$ for the number operator on differential forms $\W ^*(B,\mathscr{H})$, 
multiplying any $\w \in \W ^\ell(B,\mathscr{H})$ by $\ell$. We define the family of operators  
\begin{align}\label{Indicial family} 
I_\lambda(P_{00}) := -\lambda^2 + \left( \begin{array}{cc}
D^2_{B} + \Bigl(\frac{b-1}2-N_B\Bigr)^{2} &2\mathfrak{d} \\
2\mathfrak{d}^* & D^2_B + \Bigl(\frac{b+1}2-N_B\Bigr)^{2}
\end{array} \right)
\end{align}
acting on $\W ^*(B,\mathscr{H})\oplus \W ^{*-1}(B,\mathscr{H})$.
This is the indicial family of a $\bb$-operator $P_{00}$ that will be introduced 
in \eqref{eqn:square_phi-prelim}. Its set of indicial roots is defined as 
\begin{align}\label{Indicial set} 
\specb(P_{00}) :=  \{ \lambda \in \C \mid 
I_\lambda(P_{00}) \ \textup{is not invertible} \}.
\end{align}
Noting that $I_\lambda(P_{00})$ is symmetric, we find that 
$\specb(P_{00}) \subset \R$ is real. We can now formulate our next 
assumption.

\begin{assum}\label{assum3}
Assume that $\specb(P_{00}) \cap[-1,1]=\varnothing$.
Due to symmetry of $I_\lambda(P_{00})$ under the reflection 
$\lambda \mapsto - \lambda$, this is equivalent to $\specb(P_{00}) \cap[-1,0]=\varnothing$.
\end{assum}

Note that even if Assumption \ref{assum3} is violated for the (full) Hodge Laplacian but holds for 
the Laplacian acting on forms of a fixed degree $k$, then our results hold for that degree.
\footnote{Compare with \cite{kottke-rochon}, their method does not allow for this conclusion.} \medskip

Assumption \ref{assum3} can be reformulated in terms of spectral 
conditions on $D_B$ precisely as in \cite[(21)]{guillarmou2014low}
with $d_N$ and $\Delta_N$ replaced by $\mathfrak{d}$ and $D^2_B$. It is satisfied if $D_B$ has a sufficiently large spectral gap around zero.
As we will see in Corollary \ref{cor:kernel}, Assumption \ref{assum3} implies the \textit{non-resonance condition}\footnote{It is natural to ask whether the non-resonance condition and $0 \notin \specb(P_{00})$ already imply Assumption \ref{assum3}. We conjecture that this is not the case since \eqref{noresonance-consequence} is a global condition on $M$ while $\specb(P_{00})$ is determined locally at the boundary. However, we do not have an explicit counterexample.}
\begin{equation}\label{noresonance-consequence}
\ker_{x^{-1}L^{2}(M,g_\phi) }\Delta_{\phi} = \ker_{L^{2}(M,g_\phi)}\Delta_{\phi}.
\end{equation}
imposed also by Guillarmou and Sher \cite{guillarmou2014low}, in addition to asking for 
$0 \notin \specb(P_{00})$. These conditions are needed in order to construct the Fredholm 
inverse for the Hodge Laplacian. Specifically, the assumption $0 \notin \specb(P_{00})$
is used to obtain a Fredholm inverse in \eqref{cphi-Fredholm}.
The no zero-resonances assumption is used in Lemma \ref{lem:ker piphi}, that jointly with the functional
analytic observation of Lemma \ref{funkana} yields the zf parametrix in \eqref{Gphi}
and \eqref{Gkphi}. Now, the full Assumption \ref{assum3} is used to ensure that 
such constructed parametrices actually act boundedly on $L^2(M,g_\phi)$.

\begin{assum}\label{assum4}
The twisted Gauss Bonnet operator $D_B$, defined in \eqref{DB}, commutes with the  orthogonal
projection $\Pi$ in \eqref{projections-def} onto fibre-harmonic forms. 
\end{assum}

This assumption implies that the off-diagonal terms in the decomposition of $\Delta_{\phi}$ with respect to the bundles $\mathscr{H}$ and its orthogonal complement in $\Cinf(F)$, near $\dM$, vanish quadratically at the boundary as $\phi$-operators. This assumptions is needed to analyze the structure of the Fredholm inverse of $\Delta_{\phi}$, as 
used by Grieser and Hunsicker in \cite[Theorem 2]{grieser-para}, cf. Remark \ref{Dpi-condition}.

\begin{assum}\label{assum5}
The base $B$ of the fibration $\phi: \partial M \to B$ is of dimension $\dim B \geq 2$.
\end{assum}

This assumption has also been imposed by Guillarmou and Hassel 
\cite{guillarmou2008resolvent} and \cite{guillarmou2009resolvent}, where 
in case of trivial fibres $F$ it means $\dim M \geq 3$. If $\dim B<2$ then the resolvent  has a different behavior as $k\to0$. We use the assumption  explicitly in \eqref{eqn:inverse Delta}. Our main result is now as follows.

\begin{theorem}\label{main-theorem-intro} 
Under the Assumptions \ref{assum1}, \ref{assum2}, \ref{assum3}, \ref{assum4} and  \ref{assum5},
the Schwartz kernel of the resolvent
$(\Delta_{\phi} + k^{2})^{-1}, k>0,$ lifts to a polyhomogeneous conormal
distribution on an appropriate manifold with corners, 
with a conormal singularity along the  diagonal. 
\end{theorem}
We also determine the exponents and the leading terms in the asymptotics of the resolvent kernel, and make the different asympotic behavior with respect to the decomposition into fibrewise harmonic forms and their orthogonal complement explicit. The precise statement is given in Theorem \ref{thm:main thm detail}.
\medskip

\begin{Rem}
Our result admits some obvious extensions:
\begin{enumerate}
\item Similar to Assumption \ref{assum1}, replacing $\Delta_\phi$ by $\Delta_\phi + V$, where
$V = x^3 W$ with a positive\footnote{cf. \S \ref{Phi-vector-def} for the definition of ${}^{\phi}T^*M$}  $W \in C^\infty(\Mbar, \End(\Lambda{}^{\phi}T^*M))$, 
does not affect the argument and Theorem \ref{main-theorem-intro} still holds. Without the positivity condition on $V$, the problem is much more 
intricate since the potential might lead to non-positive spectrum and affect zero resonances, cf. \cite{guillarmou2009resolvent}.
\item As already noted in \cite[Theorem 6]{Sher}, the proof of Theorem \ref{main-theorem-intro} carries over verbatim to 
the case where the resolvent parameter $k^2$ is allowed to vary in the right half plane $\Re k>0$. Then 
the resolvent asymptotics is
uniform as $k\to0$ in angular sectors $|\arg k| \leq \frac\pi2-\eps$ for any $\eps>0$.
\end{enumerate}
\end{Rem}

\subsection{Key points of the proof}
As in the papers by Guillarmou, Hassell and Sher, our strategy to prove polyhomogeneity of the resolvent kernel is as follows: first, we construct an appropriate manifold with corners, which we call $M^2_{k,\phi}$, on which we expect the resolvent to be polyhomogeneous. The Hodge Laplacian behaves like a scattering Hodge Laplacian on the space of fibrewise harmonic forms and like a fully elliptic $\phi$-operator on the orthogonal complement of this space. The low energy resolvents for these two types of operators are described by two different blowup spaces, denoted $M^2_{k,\ssc,\phi}$ and $M^2_\phi\times\R_+$. Therefore, $M^2_{k,\phi}$ is chosen as a common blowup of these two spaces.
\medskip

The boundary hypersurfaces of $M^2_{k,\phi}$ then correspond to limiting regimes, so that the leading asymptotic term(s) of the resolvent kernel at each of them can be determined by solving a simpler model problem. The model problem at the $k=0$ face called $\zf$ involves the Fredholm inverse of $\Delta_\phi$, so we need to analyze this first.
Combining these solutions and an interior parametrix, obtained by inverting the principal symbol (which corresponds to the 'freezing coefficients' model problem at each interior point), we obtain an initial parametrix for the resolvent. We then improve this parametrix using a Neumann series argument and finally show by a standard argument that the exact resolvent has actually the same structure as the improved parametrix. \medskip

The contribution of our paper is that we analyze the low energy resolvent of $\Delta_{\phi}$ in the non-fully elliptic case, i.e. allowing the fibrewise Laplacians to have non-trivial kernel. A parametrix construction for a closely related operator, the Hodge Laplacian for the cusp metric $x^2g_\phi$, was given by Grieser and Hunsicker in \cite{grieser-para}, and a similar construction for the spin Dirac operator (assuming constant dimension for the kernel of the operators induced on the fibres) was given by Vaillant in \cite{vaillant}. Vaillant also analyzed the low energy behavior of the resolvent for this operator.
\medskip

For the construction of the Fredholm inverse of $\Delta_\phi$ we introduce and analyze a pseudodifferential calculus, the 'split' calculus (see Definition \ref{def:split phi calc}), which contains this Fredholm inverse, and reflects the fact that this operator exhibits different asymptotic behavior on the subspaces of fibre-harmonic forms and on its orthogonal complement. This construction is close to the construction by Grieser and Hunsicker in \cite{grieser-para}, where the notion of split operator was first introduced. 
For the proof of our main theorem we then construct a split resolvent calculus that additionally encodes the asymptotic behavior of the resolvent as $k\to0$. This split resolvent calculus combines the $\phi$-calculus of Mazzeo and Melrose, the resolvent calculus of Guillarmou and Hassell and the 
split calculus.
\medskip

\subsection{Structure of the paper}
Our paper is structured as follows. In Section \ref{sec:fund} we review fundamental 
aspects of geometric microlocal analysis that are important in the
present work. Then we review the pseudodifferential b-calculus in \cite{melrose1993atiyah}
and the pseudodifferential $\phi$-calculus in \cite{mazzeo1998pseudodifferential} in Sections \ref{sec:b-calc} and \ref{sec:phi-calc}. 
We state Fredholm property results and the asymptotic description of the parametrices 
for elliptic elements in both calculi. 
\medskip

In Section \ref{sec:split} we analyze the structure of the Hodge Laplacian and display its split structure with respect to the bundle of fibrewise harmonic forms and its orthogonal complement. We introduce the split pseudodifferential calculus, which is a variant of the $\phi$-calculus that reflects this splitting. We give an improved version of the parametrix construction for the Hodge Laplacian in that calculus, building on the construction by
Grieser and Hunsicker \cite{grieser-para}.
\medskip

In Section \ref{section 3} we review the low energy resolvent construction of Guillarmou, Hassell and Sher \cite{guillarmou2008resolvent}, \cite{guillarmou2014low} on scattering manifolds. In particular, we present the blowup space $M^2_{k,\ssc}$
due to Melrose and S{\'a} Barreto  \cite{melrosesabarreto}, which is used in those works, and its slightly more general cousin $M^2_{k,\ssc,\phi}$ that we need.
\medskip

The proof of our main theorem is given in Section \ref{sec:proof main thm}. 
We first construct the blowup space $M^2_{k,\phi}$ and then construct the initial parametrix. The final step towards the resolvent kernel requires a composition theorem, that we prove in Section \ref{section 6}.
\medskip

This work will have applications in the construction of renormalized zeta functions 
and analytic torsion 
in the setting of manifolds with $\phi$-metrics. These will the subject of a separate paper.

\begin{Rem}
The authors wish to acknowledge that in a parallel work by Kottke and Rochon \cite{kottke-rochon}, 
a result similar to Theorem \ref{main-theorem-intro} was obtained independently and simultaneously 
using partly different methods, in particular working with the Dirac operator, and not
relying on a split-pseudodifferential calculus. \end{Rem}

\medskip

\emph{Acknowledgements:} 
The authors are grateful to the anonymous referees for valuable comments.
The second author would like to thank Colin Guillarmou
for helpful discussions. He is grateful to the Mathematical Institute of University 
of Oldenburg for their hospitality and financial support. 
DG and BV were supported by the DFG Priority Research Program 2026 "Geometry at Infinity".

\section{Fundamentals of geometric microlocal analysis}\label{sec:fund}

We briefly recall here the main concepts and tools of geometric microlocal analysis that will later be used 
for the construction of the resolvent kernel for $\Delta_{\phi}$. The main reference 
is \cite{melrose1993atiyah}; see \cite{Gri:BBC} for an introduction.

\subsection{Manifolds with corners}
A compact manifold with corners $X$, of dimension $N$, is by definition modelled near each point $p\in X$ diffeomorphically by
$(\mathbb{R}^+)^{k}\times\mathbb{R}^{N-k}$ for some $k \in \mathbb{N}_0$, where $\R^+=[0,\infty)$.
If $p$ corresponds to $0$ then $k$ is called the codimension of $p$. A face of $X$, of codimension $k$, is the closure of a connected component of the set of points of codimension $k$. A boundary hypersurface is a face of codimension one, a corner is a face of codimension at least two. 
We assume that each boundary hypersurface $H$ is embedded, i.e. it has a defining function $\rho_H$, that is, a smooth function $X\to\R^+$ with
$H = \{\rho_H = 0\}$ and $d\rho_H$ nowhere vanishing on $H$. Then $\rho_H$ can be taken as coordinate of 
the first component of a tubular neighborhood $X\supset U\cong [0,\eps)\times H$. 
The set of boundary hypersurfaces of $X$ is denoted $\calM_1(X)$.
 In this section, we 
always work in the category of manifolds with corners. 

\subsection{Blowup of p-submanifolds}
Assume $P \subset X$ is a p-submanifold of a manifold 
with corners $X$, that is, near any $p\in P$ there is a local model for $X$ in which $P$ is locally a coordinate subspace. The blowup space $[X;P]$ is constructed by gluing
$X\backslash P$ with the inward spherical normal bundle of $P \subset X$. The latter is called the front face of the blowup.
The resulting space is equipped with a natural topology. It has a unique minimal
differential structure with respect to which smooth functions with compact support in the 
interior of $X\backslash P$ and polar coordinates around $P$ in $X$ are smooth, cf. \cite[\S 4.1]{melrose1993atiyah}.
\medskip

\noindent The canonical blowdown map 
\begin{align*}
\beta: [ X; P] \longrightarrow X,
\end{align*} 
is defined as the identity on  $X\backslash P$ and as
the bundle projection on the inward spherical normal bundle of $P \subset X$.
Finally, given a p-submanifold $Z\subset X$, we  define its 
lift under $\beta$ to a submanifold of $[ X; P]$ as follows.
\begin{enumerate}
\item if $Z \subseteq P$ then $\beta^{*}(Z) := \beta^{-1}(Z),$
\item if $Z \nsubseteq P$ then $\beta^{*}(Z) := \text{closure of \ }
\beta^{-1}(Z \backslash P).$
\end{enumerate}

\subsection{b-vector fields, polyhomogeneous functions and conormal distributions}
Let $X$ be a manifold with corners.
\begin{defn}(b-vector fields)
A b-vector field on $X$ is a smooth vector field which is tangential to all boundary hypersurfaces of $X$. 
The space of b-vector fields on $X$ is denoted  $\mathcal{V}_{\bb}(X)$.
\end{defn}

\begin{defn}(Index sets)
\begin{enumerate}
\item[(1)] A subset $E =\{(\gamma,p)\} \subset \mathbb{C}\times \mathbb{N}_{0}$ 
is called an index set if 
\begin{enumerate}
\item[(a)] the real parts $\text{Re}(\gamma)$ accumulate only at $+\infty.$
\item[(b)] For each $\gamma$ there exists $P_{\gamma}$ such that
$(\gamma, p) \in E$ implies $ p \leq P_{\gamma}$.
\item[(c)] If $(\gamma, p) \in E$ then $(\gamma + j, p') \in E$
for all $ j \in \mathbb{N}_{0}$ and $0 \leq p' \leq p.$ 
\end{enumerate} 
If $a\in\R$ then $a$ also denotes the index set $(a+\N_0)\times\{0\}$. 
Addition of index sets is addition in $\C\times\N_0$. 
For example, $a+E=\{(\gamma+a,p)\mid (\gamma,p)\in E\}$.
The \emph{extended union} of two index sets
$E$ and $F$ is defined as
\begin{align}
\label{extended}
E \, \overline{\cup} \, F = E \cup F \cup \{((\gamma, p + q + 1): \exists \, (\gamma,p) \in E,\ 
\mbox{and}\  (\gamma,q) \in F \}.
\end{align}
If $E$ is an index set and $a\in\R$ then we write 
\begin{align*}
E>a &:\iff (\gamma,k)\in E \text{ implies } \Re \gamma>a \\
E\geq a &:\iff (\gamma,k)\in E \text{ implies } \Re \gamma\geq a, \text{ and $k=0$ if $\Re \gamma=a$}.
\end{align*}
\item[(2)] An index family $\calE = (\calE_H)_{H\in\calM_1(X)}$ for $X$ is an assignment of an index set $\calE_H$ to each boundary hypersurface $H$.
Moreover, $\calE^H$ denotes the index family for $H$ which to any $H\cap H'\in\calM_1(H)$, where $H'\in\calM_1(X)$ has non-trivial intersection with $H$, 
assigns the index set $\calE_{H'}$.
\end{enumerate}
\end{defn}

\begin{defn}(Polyhomogeneous functions) A smooth function $\w $ on the interior of $X$ is called polyhomogeneous
on $X$ with index family $\calE$, we write $ \w  \in \mathcal{A}^{\calE}_{\phg}(X)$, if the following condition is satisfied:
$\w $ has an asymptotic 
expansion near each $H\in\calM_1(X)$  of the form 
$$
\w  \sim \sum_{(\gamma,p) \in \calE_{H}} a_{\gamma,p} \cdot \rho_{H}^{\gamma}
(\log\rho_H)^{p},  \quad \rho_{H} \longrightarrow 0,
$$
for some tubular neighborhood of $H$ with defining function $\rho_H$, where the coefficients $a_{\gamma,p}$ are
polyhomogeneous functions on $H$ with index family $\calE^H$. 
Since $\dim H < \dim X$, this defines polyhomogeneity by induction. \medskip

The asymptotic expansion above is assumed to be preserved under iterated application of b-vector fields. Its precise meaning is that there exists $\Gamma_0\in\R$ so that for all $\Gamma\in\R$ and any finite collection of b-vector fields $V_1,\cdots, V_N$ we have
$$
\bigl( V_1 \circ \cdots \circ V_N \bigr) \bigl(
\w  - \sum_{\stackrel{(\gamma,p) \in \calE_{H}}{\Re\gamma\leq\Gamma}} a_{\gamma,p} \cdot \rho_{H}^{\gamma}
(\log\rho_H)^{p} \bigr)
= O (\rho_H^\Gamma \prod_{H'\neq H} \rho_{H'}^{-\Gamma_0})
$$
It is a non-trivial fact that this condition is independent of the choice of tubular neighborhood, see \cite{Mel:CCDMWC}.
\end{defn}

\begin{defn}
 [Conormal distributions] \label{def:conormal distr}
 Let $P\subset X$ be a p-submanifold which is interior, i.e. not contained in $\partial X$. 
A distribution $u$ on $X$ is \textbf{conormal} of order $\mu\in\R$ with respect to  $P$ if it is smooth on $X\setminus P$ and near any point of $P$, with $X$ locally modelled by $(\R^+)^k_x \times \R^{N-k}_{y',y''}$ and $P=\{y''=0\}$ locally,
\begin{equation}
 \label{eqn:def conormal}
u(x,y',y'') = \int e^{iy''\eta''}a(x,y';\eta'')\,d\eta''
\end{equation}
for a symbol $a$ of order $m=\mu+\frac14\dim X - \frac12\codim P$.
\end{defn}
If $X=M\times M$ for a closed manifold $M$ and $P\subset X$ is the diagonal,
then $\dim X=2\dim P$, hence $m=\mu$, and conormal distributions are precisely the Schwartz kernels of pseudodifferential operators on $M$, 
with $m$ equal to the order of the operator. In our setting we will need to make the resolvent parameter $k$ into a variable, so $X=M\times M\times\R^+$ and $P=\Diag_M\times\R^+$, hence $m=\mu+\frac14$. The order of the operator still equals the order $m$ of the symbol, so its kernel is conormal of order $m-\frac14$.
\medskip

Polyhomogeneous sections of and conormal distributions valued in vector bundles over $X$ are defined analogously.

\subsection{b-maps and b-fibrations}
The contents of this subsection are due to Melrose \cite{Mel:CCDMWC}, \cite{melrose1993atiyah}, see also 
 \cite[\S 2.A]{mazzeorafe1991elliptic}.
\medskip

A smooth map between manifolds with corners is
one which locally is the restriction of a smooth map on a
domain of $\mathbb{R}^{N}$. We single out two classes of smooth maps,
such that polyhomogeneous functions behave nicely under the pullback and the push-forward
by these maps. We begin with the definition of a b-map.
  
\begin{defn} Consider two manifolds with corners $X$ and $X'$. 
Let $\rho_H$, $H\in\calM_1(X)$ and $\rho_{H'}$, $H'\in\calM_1(X')$ be defining functions. A smooth map $f : X' \to X$ is called a  b-map
if for every $H\in\calM_1(X)$, $H'\in\calM_1(X')$ there exists $e(H,H') \in \mathbb{N}_0$
and a smooth non vanishing function $h_H$ such that
\begin{equation} \label{hypb}
 f^{*}(\rho_{H}) = h_H \prod_{H'\in\calM_1(X')} \rho_{H'}^{e(H,H')}.
\end{equation}
\end{defn}

\noindent The crucial property of a b-map $f$ is that the pullback of polyhomogeneous 
functions under $f$ is again polyhomogeneous, with an explicit
control on the transformation of the index sets. 

\begin{Prop}
Let $f : X' \longrightarrow X $ be a b-map and 
$ u \in \mathcal{A}_{\text{phg}}^{\mathcal{F}}(X)$.
Then $f^{*}(u) \in \mathcal{A}_{\text{phg}}^{\mathcal{E}}(X')$
with index set $\mathcal{E} = f^{\bb}(\mathcal{F})$,
where $f^{\bb}(\mathcal{F})$ defined as in
\cite[A12]{mazzeorafe1991elliptic}.
\end{Prop}

In order to obtain a polyhomogeneous function
under pushforward by $f$, one needs additional conditions on $f$.
On any manifold with corners $X$, we associate to the space of b-vector
fields $\mathcal{V}_{\bb}(X)$ the b-tangent bundle ${}^\bb TX$, such that $\mathcal{V}_{\bb}(X)$
forms the space of its smooth sections. There is a natural bundle map ${}^\bb TX\to TX$ (see Section \ref{subsec:bvector fields diff ops} for details in case $X$ is a manifold with boundary).
 The differential
$d_{x}f: T_{x}X' \longrightarrow T_{f(x)}X$ of a b-map $f$ lifts under this map to the $\bb$-differential  $d^{\bb}_{x}f : {}^{\bb}T_{x} X' \longrightarrow {}^\bb T_{f(x)} X$ for each $x \in X'$. We can now proceed with the following definition.
 
\begin{defn}\label{b-fib}  \ \\[-3mm]
\begin{itemize}
\item A b-map $f: X' \to X$ is a b-submersion if $d_{x}^{\bb}f$ is surjective
for all $x \in X'$.
\item $f$ is called b-fibration if $f$ is a b-submersion and, in addition,
does not map boundary hypersurfaces of $X'$ to
corners of $X$, i.e for each $H$ there exists
at most one $H'$ such that $e(H,H') \neq 0$ in \eqref{hypb}.
\end{itemize}
\end{defn}
 
\noindent We now formulate the Pushforward theorem due to Melrose \cite{melrose1993atiyah}.
The pushforward map acts on densities instead of functions, and hence we consider the
density bundle $\W (X)$ of $X$, and the corresponding $\bb$-density
bundle
\begin{equation}
 \label{eqn:def omega b}
\W _{\bb}(X) := \left( \prod\limits_{H\in\calM_1(X)} \rho_H^{-1}\right) \W (X).
\end{equation}
Then we write $\mathcal{A}_{\text{phg}}^{\mathcal{E}}(X, \W _{\bb}(X))$ for polyhomogeneous
sections of the b-density bundle $\W _{\bb}(X)$ over $X$, with index set $\mathcal{E}$. 
The precise result is now as follows. 

\begin{Prop} \label{pushforward}
Let $f : X' \longrightarrow X$ be  a b-fibration.
Then for any index family $\mathcal{E}'$
for $X'$, such that for each
$H'$ with $e(H,H') = 0$ for all $H$ we have\footnote{This condition means that $\calE'_{H'}> 0$ for  
any $H'$ which maps into interior of $X$. This condition simply ensures
integrability, so that pushforward is well-defined.} $\calE'_{H'} > 0$,
the pushforward map is well-defined and acts as $$
f_{*}: \mathcal{A}_{\phg}^{\mathcal{E}'}(X',\W _{\bb}(X'))
\longrightarrow \mathcal{A}_{\phg}^{f_\bb
(\mathcal{E}')}(X,\W _{\bb}(X)).$$
Here, $f_{\bb}(\mathcal{E}')$ is defined as in \cite[A.15]{mazzeorafe1991elliptic}.
\end{Prop}

\subsection{Operators acting on half-densities} \label{subsec:half densities}

We will always identify an operator with its Schwartz kernel via integration, so
it is natural to consider densities. The most symmetric way to do this is 
using half-densities: if the Schwartz kernel is a half-density then the operator 
it defines  maps half-densities to half-densities naturally. However, differential operators are 
not typically given as acting on half-densities. 

The connection is made by fixing a positive 
real density $\nu$ (in this paper, typically the volume form associated to $g_\phi$ or a related density) on $X$. This defines an isometry
$$L^2(X,\nu)\to L^2(X,\W ^{\frac12})\,\quad
 u \mapsto u\,\nu^\frac12 
$$
where $L^2(X,\nu):=\{u:X\to \C\mid \int_X |u|^2 \,\nu<\infty\}$ and  $L^2(X,\W ^\frac12)$ is the space of square-integrable half-densities on $X$.\footnote{Note that $L^2(X,\W ^\frac12)$ is naturally identified with $L^2(X,\W _{\bb}^\frac12)$ (if $X$ has corners) -- one could also write any other rescaled density bundle here -- since square integrability is intrinsic for half-densities.} \medskip

Then if $A$ is an operator acting in $L^2(X,\nu)$ (i.e. on functions), the operator on half-densities induced by this identification is given by
 $$ \Atilde (u \nu^\frac12) := (Au) \nu^\frac12.$$
Note that
 $$ A\text{ symmetric in }L^2(X,\nu) \iff \Atilde \text{ symmetric in } L^2(X,\W ^\frac12)$$
 since by definition of $\Atilde$ we have
$
\int_X Au_1\cdot \overline{u_2}\, \nu = \int_X \Atilde (u_1\nu^\frac12) \cdot\overline{u_2\nu^\frac12}.
$ 
Also, if $A$ is given by an integral kernel $K$ with respect to $\nu$, i.e.
$(Au)(p)=\int K(p,p')u(p')\,\nu(p')$ then 
 $\widetilde{A}$ is given by the integral kernel $\widetilde{K}$ where $\widetilde{K}$ is the half-density
\begin{equation}
 \label{eqn:operators half-densities}
\widetilde{K}(p,p') = K(p,p') \,\nu(p)^\frac12\,\nu(p')^\frac12.
\end{equation}

\noindent In practice we often write $A$ instead of $\Atilde$. 

\section{Review of the pseudo-differential b-calculus}\label{sec:b-calc}

In this section we review elements of the b-calculus \cite{{melrose1993atiyah}}.
In this section  $\Mbar$ is a compact manifold with boundary $\partial M$, of dimension $n$. In contrast to the rest of the paper,
$\partial M$ need not be fibred.
 
\subsection{b-vector fields and b-differential operators} \label{subsec:bvector fields diff ops}

Recall that the space of b-vector fields $\mathcal{V}_{\bb} = \mathcal{V}_{\bb}(\Mbar)$ is defined as the 
space of smooth vector fields on $\Mbar$ which are tangential to $\partial M$.
Fix local coordinates $(x,\theta)$ near a boundary point, where $x$ defines the boundary, so that $\theta = \{\theta_i\}_i$ define local coordinates on $\partial M$.
Then, $\mathcal{V}_{\bb}$ is spanned, locally freely over $\Cinf(\Mbar)$, by 
$$
(x\partial_{x}, \partial_{\theta_{i}}).
$$
The b-tangent bundle $^\bb TM$ over $\Mbar$ is defined by requiring its space of smooth sections to be 
$\mathcal{V}_{\bb}$. Interpreting an element of $\calV_\bb$ as a section of $TM$ rather than of $^\bb TM$ defines a vector bundle map $^\bb TM\to TM$ which is an isomorphism over the interior of $M$ but has kernel $\Span\{x\partial_x\}$ over $\dM$.  The dual bundle of $^\bb TM$, the b-cotangent bundle, is denoted by $^\bb T^*M$. It has local basis $(\frac{dx}x, d\theta_i)$.
Let us also consider some Hermitian vector bundle $E$ over $\Mbar$.\medskip

The space of b-differential operators $\textup{Diff}^{m}_\bb(M; E)$ of order $m \in \N_0$ with values in $E$,
consists of differential operators of $m$-th order on $M$, given locally near the boundary $\partial M$ by 
the following differential expression (we use the convention $D_x = \frac{1}{i}\partial_x$ etc.)
\begin{equation}\label{b operator}
P = \sum_{q+ \vert \alpha\vert \leq m} P_{\alpha, q}(x,\theta)(xD_{x})^{q}(D_{\theta})^{\alpha},
\end{equation}
where the coefficients $P_{\alpha, q} \in C^\infty (\overline{\U},\End(E))$ are smooth sections of $\End(E)$.
Its b-symbol is given at a base point $(x,\theta) \in \overline{\U}$
by the homogeneous polynomial in $(\xi,\zeta)\in\R\times\R^{n-1}$
\begin{equation*}
\sigma_\bb(P)(x,\theta; \xi,\zeta) = \sum_{j+ \vert \alpha \vert = m} P_{\alpha, q}(x,\theta) \xi^{q}\zeta^{\alpha}
\end{equation*}
Invariantly this is a function on $^\bb T^*M$ (valued in $\End(E)$) if we identify $(\xi,\zeta)$ with $\xi\frac{dx}x+\zeta\cdot d\theta \in {}^\bb T^*_{(x,\theta)}M$. An operator $P \in \textup{Diff}^{m}_\bb(M; E)$ is said to be b-elliptic if 
$\sigma_\bb(P)$ is invertible on $^\bb T^{*}M \backslash \{0\} $. Writing $P^{m}(^\bb T^{*}M; E)$ for the space of 
homogeneous polynomials of degree $m$ on the fibres of $^\bb T^{*}M$ with values in $\End(E)$, the 
b-symbol map defines a short exact sequence
\begin{equation}
0 \longrightarrow \text{Diff}_\bb^{m-1}(M; E) \hookrightarrow \text{Diff}_\bb^{m}(M; E) 
\overset{\sigma_\bb \ }{\longrightarrow} P^{m}(^\bb T^{*}M; E) \longrightarrow 0.
\end{equation}

\subsection{b-Pseudodifferential operators} 

Parametrices to b-elliptic b-differential operators are polyhomogeneous conormal 
distributions on the b-double space that we now define. Consider the double space $\Mbar \times \Mbar$ and 
blow up the codimension two corner $\partial M \times \partial M$.  This defines the 
b-double space
\begin{equation*}
 M^{2}_{\text{b}} =
  [ \Mbar \times \Mbar ; \partial M \times \partial M].
\end{equation*}
\noindent We may illustrate this blowup as in Figure \ref{figure 1}, where $\theta,\theta'$ are omitted.
As usual, this blowup can be described in projective local coordinates.
If $(x,\theta),(x',\theta')$ are local coordinates on the two copies of $\Mbar$ near the boundary 
then local coordinates near the upper corner of the resulting front face bf are given by 
\begin{equation}\label{sx}
s = \frac{x}{x'}, \ x',\ \theta, \ \theta',
\end{equation}
where $s$ defines lb and $x'$ defines bf locally. Interchanging the roles of $x$ and $x'$, we get projective local coordinates
near the lower corner of bf. Pullback by the blowdown map $\beta_\bb$ is simply a change of 
coordinates from standard to projective coordinates. 
We will always fix a boundary defining function $x$ for $\dM$ and choose $x'=x$ as functions on $\Mbar$. Then $s$ is defined on a full neighborhood of $\bbf\setminus\rf$, and if in addition $\theta'=\theta$ as (local) functions on $\Mbar$ then the b-diagonal,
$$ \Diag_{\bb} :=\beta_\bb^*\Diag_\Mbar,\quad \Diag_\Mbar=\{(p,p)\mid p\in \Mbar\}\subset\Mbar\times\Mbar\,,$$ 
is locally $s=1,\theta=\theta'$. It is a p-submanifold of $M^2_{\bb}$.

\medskip

 \begin{figure}[h]    
\centering

\begin{tikzpicture}[scale = 1]   
    \draw[->] (0,0) -- (0,2);
    \draw[->] (0,0) -- (2,0);
    \node at (0,2.4) {$x'$};
    \node at (2.4,0) {$x$};
       \node at (4.7,1.5) {$\lf$};
           \node at (6.7,-0.3) {$\rf$};
           \node at (6.1,0.9) {$\bbf$};
      
    \draw[<-] (2.7,1) -- node[above] {$\beta_\bb$} (4.3,1);

   \begin{scope}[shift={(5,0)}]
        \draw[-](1,0) --(2,0);
        \draw[-](0,1)--(0,2);
         \draw  (1,0) arc (0:90:1);
          \node at (0,2.4) {$x'$};
           \node at (2.4,0) {$x$};
           
    \end{scope}
\end{tikzpicture}
\caption{b-double space $M^{2}_\bb$ and $\beta_\bb: M^{2}_\bb \to \Mbar^2$.} \label{figure 1}
\end{figure}

We defined b-densities in \eqref{eqn:def omega b}. 
The b-density bundle $\W _{\bb}(\Mbar^2)$ on   
the double space $\Mbar^2$ has local basis
$$ 
\frac{dx}x \frac{dx'}{x'} d\theta d\theta'.
$$
The b-density bundle on $M^2_{\bb}$ is, in coordinates \eqref{sx}, spanned by 
$$
 \frac{ds}{s} \frac{dx'}{x'} d\theta d\theta'.
$$
Note that
$\W _{\bb}(M^2_{\bb}) = \beta_\bb^* \W _{\bb}(\Mbar^2)$.
The corresponding half b-density bundle is denoted by $\W ^{1/2}_{\bb}(M^2_{\bb})$.
\medskip

We can now define the \emph{small and full calculus} of pseudo-differential 
b-operators, following \cite{{melrose1993atiyah}}. Note that we identify the 
operators with the lifts of their Schwartz kernels to $M^2_{\bb}$.
Recall from Section \ref{subsec:half densities} that operators act on half-densities, so their Schwartz kernels are half-densities.
\begin{defn}\label{b-calculus-basic} 
Let $\Mbar$ be a compact manifold with bundary and $E\to\Mbar$ a vector bundle.
\begin{itemize}
\item[1.] The small calculus $\Psi_{\bb}^m(M; E)$ of b-pseudodifferential operators is the space of distributions on $M^2_{\bb}$ taking values in $\W ^{1/2}_{\bb}(M^2_{\bb}) \otimes \textup{End}(E)$
 which are conormal with 
respect to the b-diagonal, smoothly down to $\bbf$, and vanish to infinite order at $\lf$, $\rf$.
\item[2.] The full calculus $\Psi_{\bb}^{m,\calE}(M; E)$ of b-pseudodifferential operators is defined as
\begin{align*}
&\Psi_{\bb}^{m,\calE}(M; E) := \Psi_{\bb}^{m}(M; E) + \calA_\bb^\calE(M;E) \\
&\textup{where} \, \calA_\bb^\calE(M;E) := \calA_{\text{phg}}^\calE(M^2_{\bb}, 
\W ^{1/2}_{\bb}(M^2_{\bb}) \otimes \textup{End}(E)),
\end{align*}
if $\calE$ is an index family for $M^2_{\bb}$ with $\calE_\bbf\geq 0$.
We sometimes leave out the bundle $E$ from the notation if it is clear from the context.
\footnote{This is a coarse version of the Definition 5.51 given in \cite{melrose1993atiyah}: there, the index sets are given for $\lf$ and $\rf$ only, and then the $\calA_{\text{phg}}$ term is replaced by $\calA_{\text{phg}}^{\calE_\lf,0,\calE_\rf}(M^2_{\bb}) + \calA_{\text{phg}}^{\calE_\lf,\calE_\rf}(\Mbar^2)$. If $\calE_\lf+\calE_\rf>0$, which is true for all index sets of operators appearing in our paper, this is contained in the given definition by the pull-back theorem, with $\calE_\bbf=0 \cup (\calE_\lf+\calE_\rf)$.
Also, this notation, and similar notation used below for other calculi, is not the same as that, e.g., in \cite[Definition 26]{grieser2009pseudodifferential}. Here we assume the conormal singularity to have smooth coefficients up to $\bbf$, while there they have index set $\calE_\bbf$. This is especially relevant when the index set at the front face is allowed to contain negative exponents, as in the definition of the $(k,\ssc)$ calculus.
}
\end{itemize}
\end{defn}
Here $\End(E)$ is the vector bundle over $M^2_{\bb}$ which is the pullback of the bundle over $\Mbar\times\Mbar$ that has fibre $\Hom(E_{p'},E_p)$ over $(p,p')\in 
\Mbar\times\Mbar$.
Note that $\Psi^{-\infty,\calE}_{\bb}(M;E)  = \calA_{\bb}^\calE(M;E)$ if $0\subset\calE_\bbf$.

\subsection{Fredholm properties of b-operators}\label{subsection-I} 
For any $P \in \textup{Diff}^{m}_\bb(M; E)$ of the form \eqref{b operator} locally near the 
boundary, we define the corresponding indicial operator $I(P)$ and indicial family $I_\lambda(P)$ by 
\begin{align*}
&I(P) = \sum_{q+ \vert \alpha\vert \leq m} P_{\alpha, q}(0,\theta) (xD_x)^q (D_{\theta})^{\alpha}, \\
&I_\lambda (P) = \sum_{q+ \vert \alpha\vert \leq m} P_{\alpha, q}(0,\theta)(\tfrac1i\lambda)^{q}(D_{\theta})^{\alpha},
\end{align*}
where the latter is a family of differential operators on $\dM$, acting in $L^2(\dM;E\restriction \dM)$. The set of indicial roots 
$\specb(P)$ is defined as
$$
\specb(P):= \{\lambda \in \C \mid I_{\lambda}(P) \ \textup{is not invertible}\}.
$$
Before we can proceed with stating the Fredholm theory results for b-operators, 
let us define weighted b-Sobolev spaces for $m \in \mathbb{R}$ and $\ell \in \mathbb{R}$ 
\begin{align*}
&x^{\ell}H^{m}_\bb(M;E) := \{ u = x^{\ell} \cdot v 
 \,\vert\, \forall P \in 
\Psi^{m}_\bb(M;E): P v \in  L^{2}(M;E)\}.
\end{align*}
Note that we define $L^2(M;E) \equiv L^2(M; E; \dvol_{\bb})$ with respect to the b-density $\dvol_{\bb}$,
which is a non-vanishing section of the b-density bundle $\W _{\bb}(\Mbar)$.

\begin{theorem}[Parametrix in the b-calculus] \label{thm: b-parametrix}
Let $P \in \textup{Diff}^m_{\bb}(M; E)$ be b-elliptic. Then for each $\alpha \notin \Re (\specb(P))$ 
there is an index family $\calE(\alpha)$ for $M^2_{\bb}$ satisfying
$$ \calE(\alpha)_{\lf} > \alpha,\quad \calE(\alpha)_{\rf} > -\alpha,\quad \calE(\alpha)_{\bbf} \geq 0,$$
and a parametrix
$Q_{\alpha} \in \Psi_{\bb}^{-m, \calE}(M; E)$,
inverting $P$ up to remainders
$$
P \circ Q_{\alpha} = \Id - R_{r,\alpha},\quad Q_{\alpha} \circ P= \Id - R_{l,\alpha},$$
where the remainders satisfy
$$ R_{r,\alpha} \in x^\infty\, \Psi_{\bb}^{-\infty,\calE(\alpha)} (M;E)
,\quad
R_{l,\alpha} \in \Psi_{\bb}^{-\infty,\calE(\alpha)}(M;E)\, x^\infty.
$$
The restriction of the Schwartz kernel of $Q_{\alpha}$ to $\bbf$ is given by the inverse of the
 indicial operator $I(P)$ in $x^\alpha L^2(\R^+ \times \dM, E)$, with weight $\alpha$, i.e. having asymptotics as dictated by $\calE(\alpha)$ at $\lf$ and $\rf$. 
 
The index family $\calE(\alpha)$ is determined by $\specb(P)$
and satisfies
\begin{equation}
\label{eqn:index sets}
\begin{split}
\pi\calE(\alpha)_\lf &= \{z+r \mid z\in \specb(P), \Re z >\alpha,r\in\N_0\},\\
\pi\calE(\alpha)_\rf &= \{-z+r \mid z\in \specb(P), \Re z <\alpha,r\in\N_0\},  
\end{split}
\end{equation}
where $\pi:\C\times\N_0\to\C$ is the projection onto the first factor, i.e. we neglect logarithms. 
\end{theorem}

\noindent Note that $x^\infty\, \Psi_{\bb}^{-\infty,\calE(\alpha)}(M;E) = \calA_{\bb}^{\calE(\alpha)_{|\rf}}(M;E)$ where $\calE(\alpha)_{|\rf}$ is the index family with index sets equal to $\calE(\alpha)$ at $\rf$ and empty otherwise. Similarly, we have $\Psi_{\bb}^{-\infty,\calE(\alpha)} (M;E)\,x^\infty =  \calA_{\bb}^{\calE(\alpha)_{|\lf}}(M;E)$.
\medskip

By standard boundedness results this implies (cf. \cite[Theorem 5.60 and Prop. 5.61]{melrose1993atiyah})
the following Fredholm and regularity result.

\begin{theorem}[Fredholmness and regularity of elliptic b-operators]
\label{thm:b Fredholm regularity}\mbox{}\\
Let $P\in \Diff^m_{\bb}(M; E)$ be  b-elliptic. Then $P$ is Fredholm as a map 
$$
P:x^\alpha H^{s+m}_{\bb}(M; E) \rightarrow x^{\alpha}H^{s}_{\bb}(M; E),
$$ 
for any $\alpha \notin \Re (\specb (P))$ and any $s\in\R$. The Fredholm inverse of $P$ is in 
the full b-calculus $\Psi_{\bb}^{-m, \calE(\alpha)}(M; E)$ with $\calE(\alpha)$ as in  Theorem 
\ref{thm: b-parametrix}. Moreover, if $u \in x^\alpha H^s_{\bb}(M; E)$ for some $\alpha,s\in\R$ and 
$Pu \in \mathcal{A}^I_{\rm{phg}}(M; E)$ for some index set $I$ then $u\in \mathcal{A}^J_{\rm{phg}}(M; E)$,
where $J = I \overline\cup K$ for some index set $K>\alpha$, determined by $\specb(P)$.
\end{theorem}

In particular, if $u$ has only Sobolev regularity, but is mapped by a differential b-operator
to a section with an asymptotic expansion at $\partial M$, for instance if $u$ is in
the kernel of $P$, then $u$ must also have a full asymptotic expansion at $\partial M$.
\medskip

Recall that the Fredholm inverse is defined as follows: if $K= \ker P$ and $R=\Ran P$, then the Fredholm inverse of $P$ is zero on $R^\perp$ and equals $(P_{|K^\perp\to R})^{-1}$ on $R$.

\section{Review of the pseudo-differential $\phi$-calculus}\label{sec:phi-calc}

In this section we review elements of the $\phi$-calculus, 
following \cite{mazzeo1998pseudodifferential}. 
We are now in the setting  
of a compact manifold $\Mbar$ with a fibration $\phi:\partial M\to B$ of the boundary. We also fix a boundary defining function $x \geq 0$ and collar neighborhood $\Umathbar\cong [0,\eps)_x\times\dM$  of $\dM$. 

\subsection{$\phi$-vector fields and $\phi$-differential operators}

 \begin{defn}\label{Phi-vector-def}
 A b-vector field $V$ on $\Mbar$ is called a $\phi$-vector field,  $V\in\mathcal{V}_{\phi}\equiv \mathcal{V}_{\phi}(\Mbar)$, if at the boundary it is tangent to the fibres of the fibration 
 $\phi: \partial M \to B$ and if it satisfies $Vx \in x^{2}C^{\infty}(\Mbar)$ for the chosen boundary defining function $x$. Near a boundary point we use coordinates
$\{x,y_i,z_j\}$ with $y = \{y_i\}_i$
being local coordinates on the base $B$, lifted to $\partial M$ and extended to 
 $[0,\varepsilon) \times \partial M$, and $z=\{z_j\}_j$ restricting to local coordinates on the fibres $F$. Then $\calV_\phi$ is spanned, locally freely over $\Cinf(\Mbar)$, by the vector fields
 $$
 x^{2}\frac{\partial}{\partial x}, x\frac{\partial}{\partial y_{i}}, \frac{\partial}{\partial z_{j}}.
 $$
We introduce the so called $\phi$-tangent space by requiring $\mathcal{V}_{\phi}(M)$
to be its smooth sections
\begin{equation*}
C^\infty(\Mbar, {}^\phi TM) = \mathcal{V}_{\phi}
= C^\infty(\Mbar)\text{-span}\left\langle x^{2}\frac{\partial}{\partial x}, 
x\frac{\partial}{\partial y_{i}}, \frac{\partial}{\partial z_{j}}\right\rangle,
\end{equation*}
where the second equality obviously holds only locally near $\partial M$.
Note that the metric $g_\phi$ extends to a smooth positive definite quadratic form on $\phiT M$ over all of $\Mbar$.
The dual bundle $\phiT^{*} M$, the so-called $\phi$-cotangent space, satisfies
\begin{equation*}
C^\infty(\Mbar, {}^\phi T^*M) = 
C^\infty(\Mbar)\text{-span}\left\langle \frac{dx}{x^{2}}, \frac{dy_{i}}{x},dz_{j}\right\rangle.
\end{equation*}
 \end{defn}
The space of $\phi$-vector fields $\mathcal{V}_{\phi}$ is closed under brackets, hence is a Lie algebra,
and is a $C^{\infty}(\Mbar)$-module. Hence it leads to the definition of
$\phi$-differential operators $\text{Diff}_{\phi}^{*}(M; E)$, where $E$ is some fixed Hermitian vector bundle. 
Explicitly, $P \in \text{Diff}_{\phi}^{m}(M; E)$ if it is an $m$-th order differential operator 
in the open interior $M$, and has the following structure locally near
the boundary $\partial M$
\begin{equation}\label{diff-phi}
 P = \sum_{\vert \alpha \vert + \vert \beta \vert + q \leq m} 
 P_{\alpha,\beta,q}(x,y,z)(x^{2}D_{x})^{q}(xD_{y})^\beta D_{z}^{\alpha},
 \end{equation}
with coefficients $P_{\alpha,\beta,q}\in C^\infty(\overline{\U},\End(E))$ smooth up to the boundary. 
The $\phi$-symbol $\sigma_{\phi}(P)$ is then locally given
over the base point $(x,y,z) \in \overline{\U}$ by 
the homogeneous polynomial in $(\xi, \eta, \zeta) \in \R\times\R^{\dim B}\times\R^{\dim F}$
\begin{equation}\label{phi-diff}
 \sigma_{\phi}(P) (x,y,z; \xi, \eta, \zeta)= \sum_{\vert \alpha \vert + \vert \beta \vert + q = m} 
 P_{\alpha,\beta,q}(x,y,z)\xi^{q}\eta^\beta \zeta^{\alpha}.
 \end{equation}
 Invariantly this is a function (valued in $\textup{End}(E)$) if we identify $(\xi,\eta,\zeta)$ with $\xi\frac{dx}{x^2}+\eta\cdot\frac{dy}x+\zeta\cdot dz\in \phiT^{*}_{(x,y,z)}M$.
We say that $P$ is $\phi$-elliptic if $\sigma_{\phi}(P)$ is
invertible off the zero-section of $\phiT^{*}M$.
Writing $P^{m}(\phiT^{*}M; E)$ for the space of 
homogeneous polynomials of degree $k$ on the fibres of $\phiT^{*}M$ valued in $\End(E)$, the 
$\phi$-symbol map defines a short exact sequence
\begin{equation}
0 \longrightarrow \text{Diff}_{\phi}^{m-1}(M; E) \hookrightarrow \text{Diff}_{\phi}^{m}(M; E) 
\overset{\sigma_{\phi} \ }{\longrightarrow} P^{m}(\phiT^{*}M; E) \longrightarrow 0.
\end{equation}

\subsection{$\phi$-Pseudodifferential operators} \label{subsec:phi psido}
 
We now recall the notion of $\phi$-pseudo\-differential operators $\Psi^{*}_{\phi}(M; E)$ from 
Mazzeo and Melrose \cite{mazzeo1998pseudodifferential}. These will be  operators whose 
Schwartz kernels lift to polyhomogeneous  distributions with conormal singularity along 
the lifted diagonal on the $\phi$-double space $M^2_\phi$. The space $M^2_\phi$ is obtained from 
the b-double space $M^{2}_{\bb}$ by an additional blowup:
recall the definition of the (interior) fibre diagonal
\begin{align*}
&\textup{diag}_{\phi,\mathrm{int}} = \{ (p, p') \in \overline{\U} \times 
\overline{\U}: \phi(p) = \phi(p')\},
 \end{align*}
In coordinates this is the submanifold $\{x=x', y=y'\}$, so the boundary of its lift is contained in the b-face of $M^2_{\bb}$ and locally given by
\begin{align*}
&\textup{diag}_\phi := 
\partial(\beta_\bb^*(\textup{diag}_{\phi,\mathrm{int}})) = \{s=1, y = y', x'=0\}.
\end{align*}
The $\phi$-double space is now defined by
\begin{equation}\label{phi-double}
M_{\phi}^{2} := [M^{2}_\bb; \textup{diag}_\phi],\quad \beta_{\phi-\bb}:M_\phi^2\to M_{\bb}^2.
\end{equation}

\begin{figure}[h]
\centering

\begin{tikzpicture}[scale = 1]
     \draw[-](1,0) --(2,0);
     \draw[-](0,1)--(0,2);
      \draw[-](-1,-1)--(3,1);
      \draw[-](-2,0)--(2,2);
      \draw[magenta]  (1,0) arc (0:90:1);
      \draw  (-1,-1) arc (0:90:1);
       \node at (0,2.4) {$x'$};
         \node at (2.2,0) {$x$};
  \node at (0,-1) {$y-y'$};
    \node[magenta] at (1.5,1) {$y=y'$};
    
    \draw[<-] (4,1)-- node[above] {$\beta_{\phi-\bb}$} (6,1);

   \begin{scope}[shift={(8,0)}]
        \draw[-](1,0) --(2,0);
        \draw[-](0,1)--(0,2);
        \draw[-](-1,-1)--(3,1);
         \draw[-](-2,0)--(2,2);
           \draw  (-1,-1) arc (0:90:1);
         \draw (0,1).. controls (0.3,1) and (0.55,0.85) .. (0.6,0.8);
             \draw (1,0).. controls (1,0.3) and (0.85,0.55) .. (0.8,0.6);
            \draw[magenta] (0.6,0.8).. controls (0.8,1) and (1,0.8)..(0.8,0.6);
             \draw[magenta] (0.6,0.8).. controls (0.6,0.7) and (0.7,0.6)..(0.8,0.6);
           \node at (2.2,0) {$x$};
            \node at (0,2.4) {$x'$};
           \node at (2.2,0) {$x$};
           \node at (-0.3,1.5) {$\lf$};
           \node at (1.5,-0.3) {$\rf$};
           \node at (-1,0) {$\bbf$};
           \node[magenta] at (1.1,1.1) {$\phif$};

    \end{scope}
\end{tikzpicture}
\caption{$\phi$-double space $M_{\phi}^{2}$} \label{figure 2}
\end{figure}

\noindent This blowup is illustrated in Figure \ref{figure 2}, with the $y,z,z'$ coordinates omitted. 
Projective coordinates near the interior of $\phif$ can be given using \eqref{sx} by
\begin{equation}\label{SU}
T = \frac{s-1}{x'}, \ Y = \frac{y-y'}{x'}, z, \ x', \ y', \ z',
\end{equation}
where $x'$ defines $\phif$ locally 
and bf lies in the limit $|(T,Y)| \to \infty$. Here, the roles of $x$ and $x'$ can be interchanged
freely. Pullback by the blowdown map is again simply 
a change of coordinates from standard to e.g. the projective coordinates above. The total blowdown map is given by
\begin{align*}
\beta_{\phi} = \beta_\bb\circ \beta_{\phi-\bb} 
: M^{2}_{\phi} \longrightarrow \Mbar^{2}.
\end{align*}
We now define the \emph{small calculus} and the \emph{full calculus} of pseudo-differential $\phi$-operators, following
\cite{mazzeo1998pseudodifferential} and \cite{grieser2009pseudodifferential}.
 As always we identify operators with the lifts of their Schwartz kernels to $M^2_\phi$, and let operators act on half-densities.
It is convenient\footnote{This means that at $\phif$ we normalize to $\phi$-densities -- for the fibration of $\phif$ given locally by projection to the $y'$ coordinate in \eqref{SU} --  and at all other faces to b-densities. In this way kernels of $\phi$-differential operators,  and more generally operators in the small $\phi$-calculus, are conormal to the diagonal 
uniformly up to $\phif$ (when considered as sections of $\W _{\bb\phi}^\frac12(M^2_\phi)$), and operators in the small b-calculus are smooth up to $\bbf$, away from $\phif$.} 
 to normalize to the $b\phi$-density bundle
\begin{equation}
  \label{eqn:densities M2phi}
\W _{\bb\phi}(M^2_\phi) := \rho_\phif^{-(b+1)}\W _{\bb}(M^2_\phi) = 
\rho_\phif^{-2(b+1)}\beta_\phi^* \W _{\bb}(\Mbar^2).
\end{equation}
The corresponding half $b\phi$-density bundle is denoted by $\W ^{1/2}_{\bb\phi}(M^2_\phi)$.

\begin{defn}\label{phi-calculus-basic} We define small and full calculi of $\phi$-operators. 
\begin{itemize}
\item[1.] The small calculus $\Psi_\phi^m(M; E)$ of $\phi$-pseudodifferential operators is the 
space of $\W ^{1/2}_{\bb\phi}(M^2_\phi) \otimes \textup{End}(E)$-valued  
distributions on $M^2_\phi$ which are conormal with 
respect to the lifted diagonal, smoothly down to $\phif$, and vanish to infinite order at $\lf$, $\rf$ and $\bbf$.
\item[2.] The full calculus $\Psi_\phi^{m,\calE}(M; E)$ of $\phi$-pseudodifferential operators is defined as
\begin{align*}
&\Psi_\phi^{m,\calE}(M; E) := \Psi_\phi^{m}(M; E) + \calA_{\phi}^\calE(M;E), \\
&\textup{where} \ \calA_{\phi}^\calE(M;E):= \calA_{\text{phg}}^\calE(M^2_\phi, 
\W ^{1/2}_{\bb\phi}(M^2_\phi) \otimes \textup{End}(E)),
\end{align*}
if $\calE$ is an index family for $M^2_\phi$ with $\calE_{\phif}\geq 0$.
Sometimes we leave out the bundle $E$ from notation if it is clear from the context.
\end{itemize}
\end{defn}

\subsection{Fredholm theory of $\phi$-operators}\label{normal op}

The normal operator of a $\phi$-differential operator 
$P \in \text{Diff}_{\phi}^{m}(M)$ is defined as follows. Write $P$ in coordinates near the boundary as
\begin{align*}
P &= \sum_{\vert \alpha \vert + \vert \beta \vert + q \leq m} 
 P_{\alpha,\beta,q}(x,y,z)(x^{2}D_{x})^{q}(xD_{y})^\beta D_{z}^{\alpha}.
\end{align*}
Then we define
\begin{equation}
\label{eqn:normal family def}
N_{\phi}(P)_{y'} :=  
\sum_{\vert \alpha \vert + \vert \beta \vert + q \leq m} 
 P_{\alpha,\beta,q}(0,y',z) D^q_T D^\beta_Y D_{z}^{\alpha}.
\end{equation}
The normal operator $N_{\phi}(P)_{y'}$ is a family of differential operators acting on
 $\mathbb{R}\times\mathbb{R}^{b}\times F$, parametrized by $y'\in B$. The Schwartz kernel of $N_\phi(P)$ can be identified with the restriction of the Schwartz kernel of $P$ to $\phif$, using coordinates \eqref{SU}. \medskip

Note that  $N_{\phi}(P)_{y'}$ is translation invariant (constant coefficient) in $(T,Y)$.
Performing Fourier transform in $(T,Y) \in \R \times \R^b$ we define the normal family
\begin{equation}\label{normal-family}
\widehat{N}_{\phi}(P)_{y'} :=
\sum_{\vert \alpha \vert + \vert \beta \vert + q \leq m} 
 P_{\alpha,\beta,q}(0,y',z) \tau^q \xi^\beta D_{z}^{\alpha}.
\end{equation}
This is a family of differential operators on $F$, parametrized by 
$(\tau, \xi) \in \R \times \R^b$ and $y'\in B$. 
\begin{defn}[Full ellipticity] \label{full ellipticity}
An elliptic differential $\phi$-operator $P \in \text{Diff}_{\phi}^{m}(M)$ is said to be fully elliptic
if additionally the operator family $\widehat{N}_{\phi}(P)_{y'}(\tau, \xi)$
is invertible for all $(\tau, \xi; y')$, {including $(\tau, \xi) = 0$}.
\end{defn}

We can now state Fredholm results for fully elliptic $\phi$-differential 
operators, due to \cite{mazzeo1998pseudodifferential}. 
 
\begin{Th}[Invertibility up to smooth kernel]\cite{mazzeo1998pseudodifferential}\,(Proposition 8).
\label{invphi} 
If $ P \in\textup{Diff}^{m}_{\phi}(M; E)$ is fully elliptic
in the sense of Definition \ref{full ellipticity}, then there exists a small calculus parametrix $Q \in \Psi_{\phi}^{-m}(M;E)$
satisfying
\begin{align*}
P  Q - \textup{Id},\ Q  P - \textup{Id} \in  x^\infty\,\Psi_\phi^{-\infty}(M;E) = \calA_\phi^\varnothing\,.
\end{align*}
\end{Th}

In order to state continuity and Fredholm results we introduce weighted $\phi$-Sobolev spaces. We write 
for any $m, \ell \in \R$
\begin{align*}
&x^{\ell}H^{m}_{\phi}(M;E) := \{ u =  x^{\ell} \cdot v \mid \forall \, P \in \Psi_{\phi}^{m}(M;E): \ 
Pv \in L^{2}(M;E) \}.
\end{align*}
Here as before, we define $L^2(M;E) \equiv L^2(M; E; \dvol_{\bb})$ with respect to the b-density $\dvol_{\bb}$. However, as derivatives in the definition of $H^m_\phi(M;E)$, e.g. for $m\in\N$, we use $\phi$-derivatives.

\begin{Th} \cite[Proposition 9, 10]{mazzeo1998pseudodifferential}.
Let $P\in \Diff^m_\phi(M; E)$ be  fully elliptic. Then for any $\alpha, s\in\R$, $P$ is Fredholm as a map 
$$
P: x^\alpha H^{s+m}_\phi(M; E) \rightarrow x^{\alpha}H^{s}_\phi(M; E).
$$ 
The Fredholm inverse lies in $\Psi_{\phi}^{-m}(M;E)$.
 \end{Th} 

We point out that this theorem is simpler than the corresponding  Theorem \ref{thm:b Fredholm regularity}  for b-operators: no weights $\alpha$ are excluded, and the Fredholm inverse lies in the same calculus (not a 'large' calculus).
This is due to the fact that \emph{full ellipticity} is a very strong condition: e.g.
for the Laplace-Beltrami operator on a $\phi$-manifold, twisted with some flat vector bundle $E$, full ellipticity would require
twisted zero-cohomology $H^0(F,E\restriction F)$ on fibers to vanish.

\section{Hodge Laplacian for $\phi$-metrics, split parametrix construction}\label{sec:split}

In this section we analyze the structure of the Hodge Laplacian for a $\phi$-metric and exhibit its 'split' structure with respect to fibre harmonic and perpendicular forms.
We define a pseudodifferential calculus which reflects this structure, carry out the parametrix construction (building on and extending \cite{grieser-para}) and show that the Fredholm inverse of the Hodge Laplacian lies in this calculus. We continue in the setting 
of a compact manifold $\Mbar$ with fibred boundary $\partial M$, with a choice of boundary defining function $x \geq 0$ and collar neighborhood $\overline{\Umath}=[0,\eps)_x\times\dM$. In this section we always consider 
$E= \Lambda_\phi M:=\Lambda {}^\phi T^*M$ and often omit the vector bundle from the notation.

\subsection{Structure of the Hodge Laplacian $\Delta_{\phi}$}\label{subsec:str hodge}
Recall the definition of a Riemannian submersion $\phi: (\partial M, g_\dM) 
\to (B, g_B)$:  we split the tangent bundle $T\partial M$ 
into subbundles $\calV\oplus T^H \partial M$, where at any $p\in \partial M$ 
the vertical subspace $\calV_p$ is the tangent 
space to the fibre of $\phi$ through $p$, and the horizontal subspace $T^H_p \partial M$ is its orthogonal complement with respect to $g_\dM$. 
Then $\phi$ is a Riemannian submersion if the restriction of the differential $d\phi:T^H\dM \to TB$ is an isometry. In this case, one can write $g_\dM=g_F + \phi^*g_B$ where $g_{F}$ equals $g_{\dM}$ on $\calV$ and  vanishes on  $T^H \partial M$. 
We also write $\phi^*TB$ for  $T^H \partial M$. \medskip

With respect to the fibred boundary metric $\phi$ we obtain on the collar neighborhood $\Umath$
\begin{equation}\label{splitting1}
 {}^\phi T\U := \phiT  M\restriction_{\overline{\Umath}} \ = 
 \textup{span} \left\{x^2\partial_x\right\} \oplus x \phi^{*}T B \oplus \calV.
\end{equation}
This splitting induces an orthogonal splitting of the
$\phi-$cotangent bundle ${}^\phi T^*M$:
\begin{equation}\label{splitting2}
{}^\phi T^*\U := 
{}^\phi T^*M \restriction_{\overline{\Umath}} \ = \textup{span} \left\{ \frac{dx}{x^2}\right\} \oplus  x^{-1} \phi^*T^*B \oplus \calV^*,
\end{equation}
where $\calV^*$ is the dual of $\calV$, see \eqref{eqn:V T*F identif}.
Recall that  the $\phi$-tangent bundle ${}^\phi TM$ is spanned locally over $\overline{\Umath}$ 
by $x^2\partial_x$, $x\partial_{y_i}$, $\partial_{z_j}$. In terms of this basis the metric $g_\phi$ takes the form
\begin{equation}
 \label{eqn:phi metric}
 g_\phi = 
\begin{pmatrix}
 1 & 0 & 0 \\
 0 & A_{00}(y) & x A_{01}(y,z) \\
 0 & x A_{10}(y,z) & A_{11}(y,z) 
\end{pmatrix}
+ O(x^3)
\end{equation}
with $A_{ij}$ smooth and $A_{00}$ not depending on $z$ because of the Riemannian submersion condition, Assumption \ref{assum2} on $g_0$. 
The vanishing order $O(x^3)$ of the higher order terms is due to Assumption \ref{assum1}.
Note that in local coordinates $\calV^*$ is usually not spanned by the $dz_j$ except at $x=0$, unless the off-diagonal terms in \eqref{eqn:phi metric} vanish. \medskip

With respect to the corresponding decomposition of $\Lambda {}^\phi T^*M$ over $\Umath$ 
and assuming that the higher order term $h\equiv 0$ is trivially zero for the moment, we compute as in Hausel, Hunsicker and Mazzeo \cite[\S 5.3.2]{gravi ins} for the Hodge Dirac operator $D_\phi$ over the collar neighborhood $\Umath$
\begin{equation}\label{hodge de rham}
D_\phi = x^{2}D_{x} + x \mathbb{A} + D_{F} + xD_{B} 
-x^{2}\mathscr{R},
\end{equation}
where we set $\mathbb{A} = A + A^{*}$, $D_{F} = d_{F} + d_{F}^{*}$, $\mathscr{R} = R + R^{*}$ and
\begin{equation}\label{DB}
D_{B} =  \left( d_{B} - \mathbb{I} \right) +  \left( d_{B} - \mathbb{I} \right)^*. 
\end{equation} 
Here, $\mathbb{I}$ and $R$ are the second fundamental form and the curvature of 
the Riemannian submersion $\phi$, respectively; $d_B$ is the sum of the lift of the exterior derivative on $B$ to $\partial M$
plus the action of the derivative in the $B$-direction on 
the $\calV^*-$components of the form. The term $x^{2}D_{x}$
acts for any $\w  \in \Lambda^\ell \left( x^{-1} \phi^*T^*B\right) \oplus \Lambda \calV^*$ as follows
$$
\left( x^{2}D_{x}\right)  \w  = \frac{dx}{x^2} \wedge \left( x^{2}\partial_{x} \right) \w , \quad 
x^{2}D_{x} \left( \frac{dx}{x^2} \wedge \w  \right) = -\left( x^{2}\partial_{x} \right) \w .
$$
To be precise, $x^2\partial_x$ is the lift of the corresponding differential operator
on $(0,\varepsilon)$ under the projection $\pi: (0,\varepsilon) \times 
\partial M \to (0,\varepsilon)$. Finally, $A$ is a $0-$th order differential operator, acting for any
$\w  \in \Lambda^\ell \left( x^{-1} \phi^*T^*B\right) \oplus \Lambda \calV^*$ by
$$
A \w  = -\ell \cdot  \frac{dx}{x^2} \wedge \w , \quad 
A \left( \frac{dx}{x^2} \wedge \w  \right) = (b-\ell) \cdot \w .
$$
We can now take the square of $D_\phi$ to compute the Hodge Laplacian
\begin{equation}\label{hodge laplace}\begin{split}
\Delta_{\phi} = D^2_\phi = (x^{2}D_{x})^2 + D^2_{F} + x^2D^2_{B}
+Q, \ Q \in x\cdot \text{Diff}^{^2}_{\phi}(M).
\end{split}\end{equation}
Since $D^2_B$ equals $\Delta_B$ plus terms in $x\cdot \text{Diff}^{^2}_{\phi}(M)$, we conclude
for the normal operator and normal family, defined in \eqref{eqn:normal family def}, \eqref{normal-family}
 \begin{align}
 \label{eqn:normal op Delta}
N_\phi(\Delta_{\phi})_{y} &= \Delta_{T,Y} + \Delta_{F_y}
 \\
 \label{NP}
 \widehat{N}_{\phi}(\Delta_{\phi})_{y} &= \tau^2 + \vert \xi \vert^{2}
  + \Delta_{F_y} \in \text{Diff}^{2}(F_y).
   \end{align}
Here $\Delta_{T,Y}$ is the Euclidean Laplacian on $\R^{b+1}\cong\R\times T_yB$, where the scalar product on $T_yB$ as well as the norm $ \vert \xi\vert$  for $\xi\in T^*_{y}B$ are defined by $g_B$. Also, $\Delta_{F_y} \equiv D^2_{F_y}$.
Thus the normal family of the Hodge Laplacian is invertible for each $(\tau,\xi)$ (including $(\tau,\xi)=0$), 
and hence $\Delta_{\phi}$ is fully elliptic, only if $\text{ker}(\Delta_{F_y})  = \{0\}$. Note that we do not
impose this restriction in this paper and hence require additional methods for the low energy resolvent construction.

\begin{Rem}
In case the higher order term $h$ in \eqref{g-phi-def}
is non-trivial, additional terms of the form $x\cdot \text{Diff}^{^2}_{\phi}(M)$
appear, which do not contribute to the normal operator. They do, however, 
contribute non-trivially to the split structure of $\Delta_{\phi}$ in the next section,
unless Assumption \ref{assum1} is imposed.
\end{Rem}

\subsection{Splitting into fibrewise harmonic and perpendicular forms}

The explicit form of $\widehat{N}_\phi(\Delta_{\phi})$ in \eqref{NP} shows that the 
Hodge Laplacian is not fully elliptic, unless $\Delta_{F_y}$ is invertible for one and hence each $y\in B$. Thus,
near $\partial M$ we split the differential forms into forms which are fibrewise harmonic and 
the perpendicular bundle. In the former, $\Delta_{\phi}$
acts as a scattering operator, in the latter $\Delta_{\phi}$ is fully elliptic, in a sense made precise below. 
\medskip

To be precise, let us write $V:= [0,\eps)\times B$. 
Then the $\phi-$cotangent bundle (with trivial fibre) reduces over $V$ to the
\emph{scattering} cotangent bundle, cf. \eqref{splitting2}
\begin{equation}
\label{eqn:phiT* decomp2} 
\scT^*V = \textup{span} \left\{ \frac{dx}{x^2}\right\} \oplus  x^{-1} T^*B.
\end{equation}
Consequently we find from \eqref{splitting2} 
\begin{equation}
 \label{eqn:phiT* decomp}
{}^\phi T^*\U = \phi^* \left(\scT^*V\right) \oplus \calV^*\,.
\end{equation}
If $p\in \overline{\Umath}$ and $F$ is the fibre through $p$, then pull-back under the inclusion of $F$ into $\overline{\Umath}$  is a map $(\phiT^*\U)_p\to T_p^*F$, which restricts to an isomorphism $\calV^*_p\to T^*_pF$, and this is an isometry for $g_0^*$ and $g_F^*$ since $\phi$ is a Riemannian submersion. Hence we have the isometry
\begin{equation} \label{eqn:V T*F identif}
\calV^* \cong T^*F.
\end{equation}
The decomposition \eqref{eqn:phiT* decomp} induces a decomposition
for the exterior algebras
\begin{equation}
 \label{eqn:Lambda decomp}
 \Lambda\, \phiT^*\U = 
 \Lambda \, \phi^*\left(\scT^*V\right) \otimes \Lambda \calV^*.
\end{equation}
Together with \eqref{eqn:phiT* decomp2}, this allows us to write
$$ \Cinf(\Umathbar,\Lambda \phiT^*\U) = \Cinf(V, \Lambda\scT^*V  \otimes \Cinf(F,\Lambda T^*F)),$$
where $\Cinf(F,\Lambda T^*F)$ is considered as a bundle over $B$, the fibre over $y\in B$ being $\Cinf(F_y,\Lambda T^*F_y)$. 
We decompose this bundle as 
\begin{equation}
 \label{eqn:Cinf F decomp}
\Cinf(F,\Lambda T^*F) = \widetilde{\mathscr{H}}\oplus\widetilde{\Cmath},
\end{equation}
where for each $y\in B$ the space $\widetilde{\mathscr{H}}_y$ is the kernel of the Hodge Laplacian $\Delta_{F_y}$ on the fibre $(F_y,g_F(y))$, and $\widetilde{\Cmath}_y$ is its orthogonal complement with respect to the $L^2$-scalar product. Note that $\dim\widetilde{\mathscr{H}}_y=\rank H^*(F_y)$ is finite and independent of $y \in B$ by the Hodge theorem.
It is a classical result \cite[Corollary 9.11]{heatdirac} that $\widetilde{\mathscr{H}}$ is a smooth vector bundle over $B$.
For each $(x,y)\in V$ let 
$$
\mathscr{H}_{(x,y)} =  \Lambda(\scT^* V)_{(x,y)}\otimes\widetilde{\mathscr{H}}_y, 
\quad \Cmath_{(x,y)}= \Lambda(\scT^* V)_{(x,y)}\otimes\widetilde{\Cmath}_y.
$$
We define a smooth section of $\mathscr{H}$ to be a smooth differential form on $\overline{\Umath}$ whose restriction to the fibre over $(x,y)\in V$ lies in $\mathscr{H}_{(x,y)}$ for each $(x,y)$, and similarly for $\Cmath$.  Summarizing, we get 
\begin{equation}
 \label{eqn:Cinf decomp}
 \Cinf(\Umathbar,\Lambda \phiT^*\U) = \Cinf(V, \mathscr{H}) \oplus \Cinf(V,\Cmath)\,.
\end{equation}
Corresponding to this decomposition we have projections
\begin{equation}\label{projections-def}
\begin{split}
&\Pi :  \Cinf(\Umathbar,\Lambda \phiT^*\U) \to \Cinf(V, \mathscr{H}), \\
&\Pi^{\perp} := \text{Id} - \Pi :  \Cinf(\Umathbar,\Lambda \phiT^*\U) \to \Cinf(V,\Cmath),
\end{split} 
\end{equation}
These maps are defined fibrewise from the projections of  $\Cinf(F_y,\Lambda T^*F_y)$ to $\widetilde{\mathscr{H}}_y$ and to $\widetilde{\Cmath}_y$, respectively, for each $y \in B$.

\subsection{Split structure of the Hodge Laplacian} \label{subsec:split hodge}

We now turn to the structure of the Hodge Laplacian with respect to the decomposition above.
Recall from \eqref{hodge de rham} that the Hodge Dirac operator $D_\phi = d + d^{*}$ over $\U$
(still with respect to $g_0$, setting $h\equiv 0$ for the moment)
takes with respect to the decomposition \eqref{eqn:Lambda decomp} the following form
\begin{equation}\label{hodge de rham2}
D_\phi = x^{2}D_{x} + x \mathbb{A} + D_{F} + xD_{B} 
-x^{2}\mathscr{R}.
\end{equation}
We consider now as in \cite{gravi ins}
the decomposition of $D_\phi$ with respect to the 
decomposition \eqref{eqn:Cinf decomp}, which leads to the following 
matrix representation 
\begin{equation}\label{hodge de rham matrix}
\begin{split}
D_\phi &=
\begin{pmatrix}
(D_\phi)_{00} & (D_\phi)_{01}\\
(D_\phi)_{10} & (D_\phi)_{11}
\end{pmatrix} \\ &:= 
\begin{pmatrix}
\Pi D_\phi \Pi & \Pi D_\phi \Pi^{\bot}\\
\Pi^{\bot} D_\phi \Pi & \Pi^{\bot} D_\phi \Pi^{\bot}
\end{pmatrix}
: 
\begin{matrix}                            
\Gamma(V,\mathscr{H}) \\  \oplus \\ \Gamma(V,\Cmath)                                                                                                                            
\end{matrix}
 \longrightarrow
\begin{matrix}                            
\Gamma(V,\mathscr{H}) \\  \oplus \\ \Gamma(V,\Cmath).                                                                                                                          
\end{matrix}
\end{split}
\end{equation}
The individual terms in the matrix \eqref{hodge de rham matrix} are given as follows. 
We define the operator $\eth = \Pi(d_{B} - \mathds{I})\Pi$,
which in \cite[Proposition 15]{gravi ins} is shown to act as a differential, i.e. $\eth^{2} =0$ and $(\eth^{*})^{2} = 0.$ 
Then we find from \eqref{hodge de rham2}
\begin{align}
&(D_\phi)_{00} = 
x^2 D_{x} + x(\eth + \eth^{*}) + x \mathbb{A} - x^{2} \Pi \mathscr{R}\Pi,\\
& (D_\phi)_{11}  = 
x^{2}D_{x} + x \Pi^{\bot}D_{B}\Pi^{\bot}  + x \mathbb{A} - x^{2}\Pi^{\bot} \mathscr{R} \Pi^{\bot} + \Pi^{\bot} D_F \Pi^{\bot},\\\label{D_{01}}
&(D_\phi)_{01} = 
\Pi\left(xD_{B} - x^{2} \mathscr{R}\right) \Pi^{\bot},\\ \label{D_{10}}
&(D_\phi)_{10}  = 
\Pi^{\bot}\left(xD_{B} - x^{2} \mathscr{R}\right) \Pi.
\end{align}
Note that in \eqref{D_{01}} and \eqref{D_{10}} we used that 
$x^{2}D_{x}$ and $\mathbb{A}$ commute with $\Pi$.
We can express the matrix in \eqref{hodge de rham matrix} as follows
\begin{equation}\label{hodge de rham matrix 2}
D_\phi =
\begin{pmatrix}
(D_\phi)_{00} & (D_\phi)_{01}\\
(D_\phi)_{10} & (D_\phi)_{11}
\end{pmatrix}
= \begin{pmatrix}
xA_{00} & xA_{01}\\
xA_{10} & A_{11}
\end{pmatrix},
\end{equation}
where the individual entries are given by
\begin{equation} \label{off diag}
\begin{split}
&A_{00} = xD_{x} + (\eth + \eth^{*}) + \mathbb{A} - x\Pi \mathscr{R}\Pi,\,
\quad A_{01} = \Pi\left( D_{B} - x\mathscr{R} \right)\Pi^{\bot},\\ 
&A_{10} = \Pi^{\bot}\left( D_{B} - x\mathscr{R}\right)\Pi, \,
\qquad \qquad \qquad \  A_{11} = (D_\phi)_{11}. 
\end{split}\end{equation}
We point out that $A_{00}$ acts as an elliptic differential b-operator on sections of $\mathscr{H}$.
Similarly, $A_{11}$ acts as a $\phi$-operator and $A_{01}, A_{10}$ 
as b-operators. However, below it is useful to think of the latter as $\phi$-operators.
\medskip

\noindent For the Hodge Laplacian $\Delta_{\phi} = (D_\phi)^{2}$
(with respect to $g_0$ in \eqref{g-phi-def} with $h\equiv 0$) one computes from \eqref{hodge de rham matrix 2}
that
\begin{equation}\label{hodge laplace matrix }
\Delta_{\phi}  = 
\begin{pmatrix}
(\Delta_{\phi})_{00} & (\Delta_{\phi})_{01}\\
(\Delta_{\phi})_{10} & (\Delta_{\phi})_{11}
\end{pmatrix},
\end{equation}
where the individual entries are given in terms of \eqref{off diag} by
\begin{align*}
(\Delta_{\phi})_{00} &= (xA_{00})^{2} + (xA_{01})(xA_{10}),\\
(\Delta_{\phi})_{01} &= (xA_{00})(xA_{01})+ xA_{01}A_{11},\\
(\Delta_{\phi})_{10} &= (xA_{10})(xA_{00})+ A_{11}(xA_{10}),\\
(\Delta_{\phi})_{11} &= (xA_{10})(xA_{01}) + A_{11}^2.
\end{align*}

\subsubsection{Unitary transformation of $\Delta_\phi$ to an operator in $L^2(M; \dvol_{\bb})$}

The Hodge Laplacian $\Delta_{\phi}$ is identified with its unique self-adjoint extension in 
$L^2(M; \dvol_\phi)$, where $\dvol_\phi$ is the volume form induced by $g_\phi$.
It is convenient to transform $\Delta_{\phi}$ 
to a self-adjoint operator in $L^2(M; \dvol_{\bb})$, 
where
\begin{equation}\label{dvb}
\dvol_{\bb} = x^{b+1} \dvol_\phi.
\end{equation}
We can pass between $L^2(M; \dvol_{\bb})$ and $L^2(M; \dvol_\phi)$ by an isometry
\begin{equation}\label{T}
W: L^2(M; \dvol_\phi) \to L^2(M; \dvol_{\bb}), \quad \w  \mapsto x^{-\frac{b+1}{2}} \w .
\end{equation}
We use that isometry to define the operator 
\begin{equation*}
\square_\phi := W \circ \Delta_{\phi} \circ W^{-1} \equiv x^{-\frac{b+1}{2}} \Delta_{\phi} x^{\frac{b+1}{2}},
\end{equation*} 
which is self-adjoint in $L^2(M; \dvol_{\bb})$ instead of $L^2(M; \dvol_\phi)$.
Below, we will deal only with $\square_\phi$, which is unitarily equivalent to the
Hodge Laplacian. \medskip

Recall that we impose Assumption \ref{assum4} so that $[D_B, \Pi] = 0$. 
This condition implies that $\Pi^{\bot}\left( D_{B}\right)\Pi = 0$ and $\Pi\left( D_{B}\right)\Pi^{\bot} = 0$, hence $xA_{01}=-x^2 \Pi
\mathscr{R}\Pi^\perp$ and $xA_{10}=-x^2\Pi^\perp\mathscr{R}\Pi$ are $x^2$ times zero-order operators.
This implies that $\square_{\phi} = x^{-\frac{b+1}2} \Delta_{\phi}x^{\frac{b+1}2}$ can be written
\begin{equation}
\label{eqn:square_phi-prelim}
\square_\phi = 
\begin{pmatrix}
xP_{00}x  & xP_{01}x \\
xP_{10}x & P_{11}
\end{pmatrix}.
\end{equation}
where $P_{00}\in\Diff^2_{\bb}(M)$, $P_{11}\in \Diff^2_\phi(M)$
and $P_{01}, P_{10}\in \Diff^2_\phi(M)$ (in fact, $P_{01}, P_{10}\in \Diff^1_\phi(M)$ in the special case where the metric perturbation 
$h\equiv 0$),
with each $P_{ij}$ sandwiched by appropriate $\Pi$ and $\Pi^\perp$ factors.
We call \eqref{eqn:square_phi-prelim} the \textit{split structure} of $\square_\phi$.
\medskip

All of the computations above are done for $h\equiv 0$ a priori. Now, for 
a non-trivial higher order term $h$, the Assumption \ref{assum1} 
with a stronger decay of $|h|_{g_0} = O(x^3)$, guarantees that
\eqref{eqn:square_phi-prelim} still holds with $O(x^3)$ contributions in each 
component. Since e.g. $xP_{00}x + O(x^3) = x(P_{00}+O(x))x$, 
we have only higher order contributions to all $P_{ij}$.

\subsubsection{Normal operator of $\square_\phi$ under the splitting}
Here we in particular explain in what sense $P_{11}\in \Diff^2_\phi(M)$ is fully
elliptic. The normal operator of $\square_\phi$ has the following form 
under the splitting above (cf. \eqref{eqn:normal op Delta}) 
\begin{equation}
 \label{eqn:square phi normal op}
 N(\square_\phi)_y = 
 \begin{pmatrix}
N(xP_{00}x)_y & 0 \\
 0 & N(P_{11})_y
\end{pmatrix}
=
\begin{pmatrix}
 \Delta_{\R^{b+1}} \otimes \textup{Id}_{\widetilde{\mathscr{H}}_y} & 0 \\
 0 & \Delta_{\R^{b+1}} \oplus (\Delta_{F_y})_{|\widetilde{\Cmath}_y}
\end{pmatrix}
\end{equation}
Note that the lower right corner is invertible as a map on sections of $\Cmath$, with inverse a convolution operator in the $(T,Y)$-variables on $\R^{b+1}$, rapidly decaying as $|(T,Y)|\to\infty$ (the class of operators with this property is called the \textit{suspended calculus}). This can be shown by using a spectral decomposition of $\Delta_F$, or directly as in \cite[Proposition 2.3]{grieser-para}.
In this sense $P_{11}$ is fully elliptic on sections of $\Cmath$. \medskip

The upper left corner of \eqref{eqn:square phi normal op} behaves differently:
If $b\geq2$ then the standard fundamental solution defines a bounded operator
(acting by convolution)
\begin{equation}
 \label{eqn:inverse Delta}
  \Delta_{\R^{b+1}}^{-1} = c|(T,Y)|^{1-b}*  : x^{\gamma+2} 
  L^2(\R^{b+1},\dvol_{\bb}) \to x^{\gamma} H^2_{\bb}(\R^{b+1},\dvol_{\bb}), 
\end{equation}
if $\gamma\in(0,b-1)$, where $x>0$ is smooth and equals $\frac1r$ for large values of the
radial function $r$ on $\R^{b+1}$. Here, $c$ is a dimensional constant. For $b<2$ there are no values $\gamma$ for 
which $\Delta_{\R^{b+1}}$ is invertible between these spaces. See for example 
\cite[Proposition A.2]{AlbRocShe:RHKTDFC}.
In any case, the convolution only decays polynomially at infinity.

\subsubsection*{Structure of $P_{00}$}
On $V = [0,\eps)\times B$, using \eqref{eqn:phiT* decomp2} we may write any $\calH$-valued form as a sum
$u = u_1 + \frac{dx}{x^2}\wedge u_2$ where $u_1$, $u_2$ are sections of $\Lambda^\ell (x^{-1} \phi^*T^*B) \otimes 
\mathscr{H}$ and $\Lambda^{\ell-1} (x^{-1} \phi^*T^*B) \otimes 
 \mathscr{H}$, respectively. 
Computations verbatim to \cite[(15)]{guillarmou2014low} yield
\begin{align}\label{indicial-operator-2}
P_{00} 
= -(x\partial_x)^{2} + L + x^2 \, \Pi \, \mathscr{R} \, \Pi^\perp \, \mathscr{R} \, \Pi.
\end{align}
where with respect to this splitting $u \leftrightarrow (u_1,u_2)^T$ the action of $L$ 
is given by 
\begin{equation}\label{LL}
L = \left( \begin{array}{cc}
\Pi D^2_{B} \Pi + \Bigl(\frac{b-1}2-N_B\Bigr)^{2} &2(d_B - \mathbb{I}) \\
2(d_B - \mathbb{I})^* & \Pi D^2_B \Pi + \Bigl(\frac{b+1}2-N_B\Bigr)^{2}
\end{array} \right)
\end{equation}
(do not confuse this matrix representation with 
the matrix formula under the splitting into fibrewise harmonic and perpendicular forms),
where $N_B$ denotes the number operator on $\Lambda^*(x^{-1} \phi^*T^*B)$, 
multiplying elements in $\Lambda^\ell(x^{-1} \phi^*T^*B)$ by $\ell$. The operator
$P_{00}$ is a differential b-operator. Its indicial family $I_\lambda(P_{00})$, given in \eqref{Indicial family},
and the set of indicial roots $\specb(P_{00})$ are defined in \S \ref{subsection-I}. 
\medskip

Note that $xP_{00}x$ is, up to $O(x^4)$ and modulo the unitary transformation \eqref{T}, the Hodge Laplace operator on $(V, \frac{dx^2}{x^4}+ \frac{\phi^*g_B}{x^2})$ twisted by the vector bundle $\calH$, equipped with the flat connection $\eth = \Pi(d_{B} - \mathds{I})\Pi$.

\subsection{Parametrix construction for the split Hodge Laplacian} \label{review grieser-para}

We now review the parametrix construction for the operator $\square_\phi$, 
which is the main result of \cite{grieser-para}. We begin with the definition 
of the split $\phi$-calculus. \medskip

A section $u$ of $\Lambda {}^\phi T^*M$ can be decomposed  over the collar neighborhood $\Umathbar$ of the
boundary $\partial M$ 
into fibrewise harmonic and perpendicular forms as in \eqref{eqn:Cinf decomp}. With 
respect to that decomposition, we can write $u$ over $\Umathbar$ as $\Pi u+\Piperp u$ or as a two-vector 
\begin{equation}
 \label{eqn: u decomp}
u \restriction \Umathbar = \begin{pmatrix}
 \Pi u \\ \Pi^\perp u
\end{pmatrix}.
\end{equation}
The different parts of the parametrix of $\square_\phi$ with respect to this decomposition 
have different index sets in their asymptotics on $M^2_\phi$. We introduce notation to describe 
this behavior. First we consider sections over $M$.

\begin{defn} \label{def:split phg space}
 For an index set $I$, we define $\calA^I_{\mathscr{H}}(M)$ to be the space of 
 $u\in\calA^I_{\textup{phg}}(M, \Lambda {}^\phi T^*M)$, whose 
 decomposition \eqref{eqn: u decomp} has index sets 
 $
\begin{pmatrix}
 I \\ I+2
\end{pmatrix}
 $.
 In other words, the leading terms to two orders in the 
 asymptotics of $u$ at $\partial M$ have values in $\mathscr{H}$. 
\end{defn} 
\
We extend this concept to sections 
on the $\phi$-double space
$M^2_\phi$. Note that for a section on $\Mbar^2$ we can distinguish its 
$\mathscr{H}$- and $\Cmath$-parts in both factors near the corner $\dM\times\dM$ 
and thus represent it by a $2\times2$ matrix. However, near $\lf=\dM\times \Mbar$ 
and $\rf=\Mbar\times\dM$ we can do this only in the first (resp. second) factor, so we get a $2\times1$ resp. $1\times2$ vector. The boundary hypersurfaces of $M^2_\phi$ lying over $\dM\times\dM$ are $\bbf$ and $\phif$. Thus we 
define the split $\phi$-calculus, identifying operators with lifts of their Schwartz kernels, as follows.

\begin{defn}[Split $\phi$-calculus] \label{def:split phi calc}
Let $\calE$ be an index family for $M^2_\phi$ and consider
$
K\in \calA^\calE_\phi(M,\Lambda_\phi M).
$ 
We write $\Pi K$ meaning that $\Pi$ acts on the first component in $M^2$, 
and $K\Pi$ meaning that $\Pi$ acts on the second component -- the notation suggested by 
interpreting $K$ as an operator. Then $$K \in \calA^\calE_{\phi,\mathscr{H}}(M)$$ if the following holds:

\begin{enumerate}
 \item 
 at $\bbf$, when $K$ is written with respect to the $\mathscr{H}$-$\Cmath$ decomposition as a $2\times2$ matrix $
\begin{pmatrix}
 \Pi K\Pi & \Pi K \Pi^\perp \\ \Pi^\perp K \Pi & \Pi^\perp K \Pi^\perp
\end{pmatrix}
$, it has index sets 
$\begin{pmatrix}
 \calE_\bbf & \calE_\bbf+2 \\ \calE_\bbf+2 & \calE_\bbf+4
\end{pmatrix}
$. 
\item
 at $\lf$, when $K$ is written as a vector $
\begin{pmatrix}
 \Pi K \\ \Pi^\perp K
\end{pmatrix}
$, it has index sets
$
\begin{pmatrix}
 \calE_\lf \\ \calE_\lf+2
\end{pmatrix}
$.
\item
at $\rf$, when $K$ is written  as a vector $(K\Pi,K\Pi^\perp)$, it has index sets 
$
\begin{pmatrix}
 \calE_\rf & \calE_\rf+2
\end{pmatrix}
$.
\item 
at $\phif$, when $K$ is written as a $2\times 2$ matrix, it has index sets
$$\begin{pmatrix}
 \calE_\phif & \calE_\phif+2 \\ \calE_\phif+2 & \calE_\phif
\end{pmatrix}
$$
\end{enumerate}
The split $\phi$-calculus is then defined in view of Definition \ref{phi-calculus-basic} by
$$ \Psi^{m,\calE}_{\phi, \mathscr{H}}(M) = \Psi^{m}_\phi(M, \Lambda_\phi M) + \calA^\calE_{\phi,\mathscr{H}}(M).$$
We call the matrices and vectors of index sets in (1)-(4) the \textit{split index family} associated with $\calE$.
\end{defn}

In short, the definition says that terms with a $\Piperp$ factor on the left have an extra $x^2$ factor from the left (affecting all faces except $\rf$), and terms with a $\Piperp$ factor on the right have an extra $x^2$ on the right (affecting all faces except $\lf$) -- except for the $\Piperp K\Piperp$ term at $\phif$, which is no better than the $\Pi K\Pi$ term.
 Essentially, this latter fact says that $K$ is diagonal to two leading orders at $\phif$. \medskip
 
Note that the proof of our Composition Theorem \ref{compositionth} also yields a composition theorem for $\Psi^{m,\calE}_{\phi,\Hmath}(M)$, by restriction to $\zf$. However, we do not need it, so we do not state it explicitly here.
\medskip

The following result is similar to \cite[Theorem 12]{grieser-para}, which is the main result therein. 
Our statement here is slightly stronger 
and we write out the proof in detail, since it will be adapted 
to the resolvent construction below.

\begin{theorem}\label{thm:split phi-parametrix}
For each $\alpha \notin \specb(P_{00})$, $\square_\phi$ has a parametrix $Q_\alpha$ such that
$$ 
\square_\phi Q_\alpha = \textup{Id} - R_\alpha,\ Q_\alpha\square_\phi = \textup{Id} - R'_\alpha, 
$$
where $Q_\alpha\in \Psi^{-2,\calE}_{\phi, \mathscr{H}}(M)$ and remainders
$R_\alpha\in x^\infty\,\Psi^{-\infty,\calE}_{\phi, \mathscr{H}}(M)$ and $R'_\alpha\in\Psi^{-\infty,\calE}_{\phi, \mathscr{H}}(M)\,x^\infty$.
The index family $\calE$ is given in terms of $\calE(\alpha)$ $($that is determined by $\specb(P_{00})$
and satisfies \eqref{eqn:index sets}$)$ by
$$
\calE_{\lf}=\calE(\alpha)_\lf-1, \ \calE_\rf=\calE(\alpha)_\rf-1, \ \calE_\bbf\geq-2, \ \calE_{\phif}\geq 0.
$$ 
Index sets for $Q_\alpha$ with respect to the $\mathscr{H}$-$\Cmath$ decomposition
are illustrated schematically in Figure \ref{fig-schematic-index} ($\ell$ shall run through elements of 
$\calE(\alpha)_\lf$ and $r$ through elements of  $\calE(\alpha)_\rf$ $)$
\end{theorem}

\begin{figure}[h]
\centering
\begin{tikzpicture}[scale =1]

\begin{scope}[scale = 1.5]
\draw[-] (0,2)--(0,1);
\draw[-] (0,1)--(0.4,0.6);
\draw[-] (0.6,0.4)--(1,0);
\draw[-] (1,0)--(2,0);
\draw (0.4,0.6).. controls (0.5,0.7) and (0.7,0.5) ..  (0.6,0.4);

\node at (0.3,1.8) {\tiny{$\ell-1$}};
\node at (0,0.7) {\tiny{$-2$}};
\node at (0.7,0.7) {\tiny{$0$}};
\node at (1.6,0.2) {\tiny{$r-1$}};

\node at (1.5,1.5) {\tiny{$(\mathscr{H},\mathscr{H})$}};

\end{scope}

\begin{scope}[shift={(5,0)}, scale = 1.5]

\draw[-] (0,2)--(0,1);
\draw[-] (0,1)--(0.4,0.6);
\draw[-] (0.6,0.4)--(1,0);
\draw[-] (1,0)--(2,0);
\draw (0.4,0.6).. controls (0.5,0.7) and (0.7,0.5) ..  (0.6,0.4);

\node at (0.3,1.8) {\tiny{$\ell-1$}};
\node at (0.1,0.7) {\tiny{$0$}};
\node at (0.7,0.7) {\tiny{$2$}};
\node at (1.6,0.2) {\tiny{$r+1$}};

\node at (1.5,1.5) {\tiny{$(\mathscr{H},\Cmath)$}};
\end{scope}

\begin{scope}[shift={(0,-4)}, scale = 1.5]

\draw[-] (0,2)--(0,1);
\draw[-] (0,1)--(0.4,0.6);
\draw[-] (0.6,0.4)--(1,0);
\draw[-] (1,0)--(2,0);
\draw (0.4,0.6).. controls (0.5,0.7) and (0.7,0.5) ..  (0.6,0.4);

\node at (0.3,1.8) {\tiny{$\ell+1$}};
\node at (0.1,0.7) {\tiny{$0$}};
\node at (0.7,0.7) {\tiny{$2$}};
\node at (1.6,0.2) {\tiny{$r-1$}};

\node at (1.5,1.5) {\tiny{$(\Cmath,\mathscr{H})$}};
\end{scope}
\begin{scope}[shift={(5,-4)}, scale = 1.5]

\draw[-] (0,2)--(0,1);
\draw[-] (0,1)--(0.4,0.6);
\draw[-] (0.6,0.4)--(1,0);
\draw[-] (1,0)--(2,0);
\draw (0.4,0.6).. controls (0.5,0.7) and (0.7,0.5) ..  (0.6,0.4);

\node at (0.3,1.8) {\tiny{$\ell+1$}};
\node at (0.1,0.7) {\tiny{$2$}};
\node at (0.7,0.7) {\tiny{$0$}};
\node at (1.6,0.2) {\tiny{$r+1$}};

\node at (1.5,1.5) {\tiny{$(\Cmath,\Cmath)$}};
\end{scope}

\end{tikzpicture}
 \caption{Schematic structure of index sets of $Q_\alpha$
 }
  \label{fig-schematic-index}
\end{figure}

As preparation for the proof we need two considerations for dealing with the splitting of $\square_\phi$ in 
\eqref{eqn:square_phi-prelim}. Both arise from the need to compose a parametrix of $P_{00}$ with $\phi$-operators coming from the other entries of the matrix.  Recall that $P_{00}$ is a b-operator in the base $V=[0,\varepsilon) \times B$, so its 
Schwartz kernel, and the Schwartz kernel of its parametrix $Q_{00}$, lift to distributions on the b-double space 
$V^2_{\bb}$, valued in $\End(\mathscr{H})$ (and half-densities). On the other hand, the kernels of $\phi$-operators are distributions on the $\phi$-double space $\U^2_\phi$. Thus, in order to analyze the compositions, we first lift $Q_{00}$ to $V^2_\phi$ and\footnote{The notation $V^2_\phi$ indicates the double space constructed in Section \ref{subsec:phi psido}, with $M$ replaced by $V$ and with the trivial fibration (whose fibres are points).} then to $\U^2_\phi$. To do this we need the following facts.

 \subsubsection*{Fact 1: Lifting from b-double to $\phi$-double space.} 
 
Consider the lift of a kernel from $V^2_{\bb}$ to $V^2_\phi$ under the blow-down 
map $\beta_{\phi - \bb}$ in Figure \ref{figure 2}.  An elementary calculation 
(compare \cite[Proposition 1]{grieser-para}) implies that for any index 
family $\calF$ on $V^2_{\bb}$ and any $m\in\R$ (we use the notation introduced in 
Definitions \ref{b-calculus-basic} and \ref{phi-calculus-basic} and omit the vector bundle $\Hmath$ from the notation)
\begin{equation}
 \label{eqn:lifting}
\beta_{\phi - \bb}^*: \Psi_{\bb}^{-m}(V) \to \rho_\phif^m\Psi_\phi^{-m}(V) +\calA^{\calF_m}_\phi(V),
 \quad \beta_{\phi - \bb}^*: \calA^{\calF}_{\bb}(V) \to \calA^{\calF'}_\phi(V),
\end{equation}
where $\calF_{m,\bbf}=0$, $\calF_{m,\phif}=m \,\overline\cup\,(b+1)$, $\calF_{m,\lf}=\calF_{m,\rf}=\varnothing$ 
and $\calF'$ has the same index sets as $\calF$ at $\lf$, $\rf$ and $\bbf$, and in addition 
$\calF'_\phif = \calF_\bbf+(b+1)$ (the shift arises from the density factor, see \eqref{eqn:densities M2phi}).
We use this to lift the b-parametrix $Q_{00}$, where $m=2$. Note that, since $b\geq2$ by Assumption \ref{assum5}, the leading term in $\calF'_\phif$ is $2$, without logarithm. If $b=1$, there would be an additional logarithm. If $b=0$, the leading term would be $1$. 

 \subsubsection*{Fact 2: Extended calculus.} 
The lift $\beta_{\phi - \bb}^* Q_{00}$ is a distribution on $V^2_\phi$, valued in $\End(\mathscr{H})$
(and half-densities). Since $\mathscr{H}\subset\Cinf(F,\Lambda T^*F)$ and $\Umathbar$ is an $F$-bundle over $V$, this distribution can also be interpreted as a distribution on $\U^2_\phi$. It is conormal with respect to the interior submanifold $\{x=x', y=y'\}$, which is the diagonal in $V^2_\phi$, but is the fibre diagonal (thus larger than the diagonal) in $\U^2_\phi$. Thus $\beta_{\phi - \bb}^* Q_{00}$ does not define a pseudo-differential operator in $\Psi^{-2, *}_\phi(\U)$.\footnote{Note that it is the fibre diagonal in the $\phi$-double space, unlike the fibre diagonal in the b-double space which was used in the definition of $M^2_\phi$.} 
\medskip
 
Since in our argument sums of such operators and operators in $\Psi^{-2,*}_\phi(\U)$ will be considered, we  
define for any pseudo-differential calculus on $M$ or $\U$, the corresponding \textit{extended calculus} as 
the space of operators whose kernels are sums of two terms, one conormal with respect to the diagonal and one conormal with respect to the fibre diagonal, the latter term being supported near the corner $(\partial M)^2$ of $\Mbar^2$.
We denote the extended calculus by $\Psibar^{-2,*}_\phi(\U)$. \medskip

Note that the fibre diagonal hits the boundary of  $M^2_\phi$ only in the interior of  the $\phi$-face $\phif$, 
just like the diagonal, because $\U^2$ is a fibre bundle with fibre $F^2$. In particular, the extension 
does not affect asymptotic behavior at other faces. In \cite[Proposition 2]{grieser-para} it is shown that the 
extended calculus enjoys the same composition properties as the standard calculus. 
\medskip

The extended calculus is only used in the intermediate steps, the final para\-metrix lies in the standard (non-extended) calculus.

\begin{proof}[Proof of Theorem \ref{thm:split phi-parametrix}]
We follow the parametrix construction given in \cite{grieser-para}. We explain the main steps since our notation is somewhat different and since we will generalize the construction to the $k$-dependent case. We also simplify the construction slightly and give more details of the second step. We construct the right parametrix, it then turns out to be a left parametrix also, by standard arguments. The construction proceeds in four steps:
\begin{enumerate}
 \item[Step 1:] We first construct a parametrix $Q_1$, with $\square_\phi Q_1 = \textup{Id} - R_1$ 
 where the remainder $R_1$ vanishes at the boundary faces $\bbf$ and $\phif$ suitably.
 \item[Step 2:] We improve $Q_1$ to a parametrix $Q_2$ whose remainder also vanishes at $\lf$.
 \item[Step 3:] Using a parametrix in the small $\phi$-calculus, obtained by inverting the principal symbol, 
 we remove the interior singularity of the remainder.
 \item[Step 4:] Iteration gives a remainder as in the theorem.
\end{enumerate} \ \medskip

\noindent \textbf{Step 1:}
We first work near the boundary, i.e. in $\Umath$.
We use the fact that the diagonal terms of $\square_\phi$ in \eqref{eqn:square_phi-prelim} are elliptic resp. fully elliptic in the 
b-calculus and (extended) $\phi$-calculus sense, and that the product of the off-diagonal terms is higher order compared to the product of the diagonal terms in terms of $x$. 
 Abstractly, an approximate right inverse for a block matrix $P=
\begin{pmatrix}
 A & B\\ C & D
\end{pmatrix}
$ with this property is given by\footnote{This formula arises from taking leading terms 
in the Schur complement formula for the inverse of a block matrix.}
 \begin{equation}\label{Q1}
 Q_1= \begin{pmatrix}
 \widehat{A} & - \widehat{A} B' \\ -\widehat{D} C' & \widehat{D}
\end{pmatrix},
\end{equation}
where $\widehat{A}$, $\widehat{D}$ are right parametrices for $A,D$, say
 $$A\widehat{A}=\textup{Id} - R,\ D\widehat{D}=\textup{Id}-S\quad\text{and}\quad B' = B\widehat{D},\ C'=C\widehat{A}.$$
 A short calculation gives $PQ_1=\textup{Id} - R_1 $ where
$$R_1=R_1'+R_1'', \quad R_1' = 
\begin{pmatrix}
 R  & -RB' \\
 -SC' & S 
\end{pmatrix}
,\ R_1'' = 
\begin{pmatrix}
 B'C' & 0\\
 0 & C'B'
\end{pmatrix}.
$$
We apply this with $A=xP_{00}x$, $B=xP_{01}x$, $C=xP_{10}x$, $D=P_{11}$. 
Since $\alpha\not\in \specb(P_{00})$ there is a b-calculus parametrix $Q_{00}$, 
obtained from a small calculus parametrix (near the diagonal) and inversion of the 
indicial operator $I(P_{00})$
\begin{equation}
 \label{eqn:Q00 R00}
 P_{00}Q_{00} = \textup{Id} - R_{00},\quad Q_{00}\in\Psi_{\bb}^{-2,\calE(\alpha)}(V,\Hmath),\ 
 R_{00}\in \rho_\bbf \calA_{\bb}^{\calE(\alpha)}(V,\Hmath).
\end{equation}
Note that Theorem \ref{thm: b-parametrix} actually yields a better parametrix
with remainder in $\rho^\infty_\bbf\rho^\infty_\lf \calA_{\bb}^{\calE(\alpha)}(V,\Hmath)$.
However, an extension of that result to the resolvent is not straightforward,
and in fact the rather crude parametrix $Q_{00}$ with remainder $R_{00}$
is fully sufficient for our purposes. \medskip

\noindent Also, there is a $\phi$-calculus parametrix for $P_{11}$, i.e.\ 
\begin{equation}
 \label{eqn:Q11 R11}
P_{11}Q_{11}=\textup{Id} - R_{11},\quad 
Q_{11}\in\Psibar^{-2}_\phi(\U),\ R_{11}\in \calA^{\varnothing}_{\phi} (\U).
\end{equation}
(We leave out the bundle $\Lambda {}^\phi T^*M$ from notation in this proof.)
We refer the reader to \cite[Proposition 2]{grieser-para} for details why this works in the extended calculus,
and emphasize that we removed higher order terms at $\phif$ right away. \medskip
 
Then $\widehat{A}=x^{-1}Q_{00}x^{-1}$ lifts to an element of $\Psi_\phi^{-2,\calE}(V,\calH)\subset\Psibar_\phi^{-2,\calE}(\U)$ 
by \eqref{eqn:lifting}, with $\calE$ as in the statement. With $\widehat{D}=Q_{11}$ we get
\begin{align*}
&B'=(xP_{01}x) Q_{11}\in x^2\, \Psibar^{\, 0}_\phi(\U)=\Psibar^{\, 0}_\phi(\U) x^2,  \\
&C'=(xP_{10}x) \widehat{A} \in x^2\, \Psibar_\phi^{\, 0,\calE}(\U).
\end{align*} 
The extra $x^2$ factors in $B',C'$  give $Q_1\in \Psibar_{\phi, \mathscr{H}}^{-2,\calE}(\Umath)$.

We now analyze the remainder terms $R_1'$ and $R_1''$. 
We shall use the notation for index families
\begin{equation}
 \label{eqn:index set tuples}
\begin{split}
 &\calF = (a,b,l+\lambda,r+\rho) \ \textup{means:}\\
 &\calF_\bbf\geq a,\ \calF_\phif\geq b,\ \calF_\lf=\calE(\alpha)_\lf+\lambda,\ \calF_\rf=\calE(\alpha)_\rf+\rho.
\end{split}
\end{equation}
First, $R=xR_{00}x^{-1}$ lifts by \eqref{eqn:lifting}
to be in $\calA_{\phi}^{(1,4,l+1,r-1)}(\U)$. Then by the composition result 
\cite[Theorem 9]{grieser-para} we find $RB' \in \calA_{\phi}^{(3,6,l+1,r+1)}(\U)$. 
Similarly, we conclude with $S=R_{11}$ that $SC'\in \calA_{\phi}^{(\varnothing,\varnothing,\varnothing,r-1)}(\U)$.
The diagonal terms in $R_1''$ are 
\begin{align*}
&B'C'\in  x^4\Psibar_\phi^{0,\calE}(\U)
=\rho_\phif^4 \Psibar_\phi^{0,(2,0,l+3,r-1)}(\U), \\
&C'B'\in x^2\Psibar_\phi^{0,\calE} (\U)x^2
=\rho_\phif^4\Psibar_\phi^{0,(2,0,l+1,r+1)}(\U).
\end{align*}
In summary we get $R_1 \in \rho_\phif^4\Psibar_\phi^{0,\calR_1}(\Umath)$, where
\begin{equation}
 \label{eqn:phi-param R1}
\calR_1 = 
\begin{pmatrix}
(1,0,l+1,r-1) & (3,2,l+1,r+1) \\
(\varnothing,\varnothing,\varnothing,r-1) & (2,0,l+1,r+1)
\end{pmatrix}.
\end{equation}
Note that all terms in $R_1$ vanish at $\bbf$ and $\phif$. If we had simply inverted the 
diagonal terms then we would have got a remainder whose $\Piperp R\Pi$ component 
does not vanish at $\bbf$, which would not be good enough for the 
iteration argument in the resolvent construction below. 
\medskip

At this point the Schwartz kernels of $Q_1$ and $R_1$ are defined over $\Umath\times\Umath$ only. Using a cutoff function we modify $Q_1$ to an element of $\Psibar_{\phi, \mathscr{H}}^{-2,\calE}(M)$, without changing it near $\bbf\cup\phif$, so that $R_1=\textup{Id}-PQ_1$ is in $\rho_\phif^4\Psibar_\phi^{0,\calR_1}(M)$ with $\calR_1$ as above.
\medskip

\noindent \textbf{Step 2:}
We want to refine our parametrix $Q_1$ from Step 1 so that the remainder vanishes to 
infinite order at $\lf$. We accomplish this by determining $Q_1'$ supported near $\lf$ so that $\square_\phi Q_1'$ agrees with $R_1$  at $\lf$ to infinite order. Then 
\begin{equation}
\label{eqn:Q1' eqn} 
\square_\phi Q_1'  = R_1 - R_2,
\end{equation}
where $R_2$ has the same index sets as $R_1$ except for $\varnothing$ at $\lf$, and also has order zero.
The construction is essentially the same as that in \cite[Lemma 5.44]{melrose1993atiyah}, 
but we need to be careful with the correct exponents in the $\mathscr{H}-\Cmath$ splitting.
Finding $Q_1'$ amounts to solving
\begin{equation}
 \label{eqn:remove boundary terms}
\begin{pmatrix}
xP_{00}x  & xP_{01}x \\
xP_{10}x & P_{11}
\end{pmatrix}
\begin{pmatrix}
 q_0\\ q_1
\end{pmatrix}
\equiv 
\begin{pmatrix}
 r_0 \\
 r_1
\end{pmatrix}
\end{equation}
as an equation in the $(x,y,z)$ variables, with $(x',y',z')$ as parameters, 
to infinite order as $x\to0$, with control of the $x'\to0$ behavior which corresponds to 
the asymptotics at the intersection of $\lf$ with $\bbf$. Here
\begin{itemize}
\item $\begin{pmatrix} r_0 \\ r_1 \end{pmatrix}$ runs through the two columns of $R_1$, \medskip

\item $\begin{pmatrix} q_0 \\ q_1 \end{pmatrix}$ runs through the corresponding two columns of $Q_1'$.
\end{itemize}
By \eqref{eqn:phi-param R1} the exponents at $\lf$ occuring in $\begin{pmatrix}
 r_0 \\
 r_1
\end{pmatrix}
$ are $\begin{pmatrix}
 l+1 \\
 l+1
\end{pmatrix}
$
 for $l\in\pi\calE(\alpha)_\lf$, for either column of $R_1$ (in fact, slightly better in the first column). Then \eqref{eqn:remove boundary terms} 
 can be solved to leading order for $\begin{pmatrix}
 q_0 \\
 q_1
\end{pmatrix}
$ having leading exponents
$\begin{pmatrix}
 l-1 \\
 l+1
\end{pmatrix}
$ 
at $\lf$
by first choosing\footnote{If $l$ is an indicial root of $P_{00}$ then $q_0$ gets extra logarithmic terms.}  
$q_0$ satisfying (always to leading order) $(xP_{00}x)q_0=r_0$ (using b-ellipticity of $P_{0}$). 
Note that the term $(xP_{01}x)q_1$ will be of higher order. Then we choose $q_1$ satisfying 
$P_{11} q_1 = r_1 - (xP_{10}x)q_0$ (using full ellipticity of $P_{11}$).
Doing this iteratively removes all orders step by step. \medskip

Uniformity at $\lf\cap\bbf$ (as $x'\to 0$) is shown as in \cite[Lemma 5.44]{melrose1993atiyah}, we only need to check the exponents at $\bbf$. For the left column of $R_1$, they are $(1,\varnothing)^T$ by  \eqref{eqn:phi-param R1}, which implies they are $(-1,1)^T$ for the left column of $Q_1'$ by the explanation above. For the right column of $R_1$ they are $(3,2)^T$, which implies similarly $(1,2)^T$ for the right column of $Q_1'$. In summary, $Q'_1$ has index sets (only listing $\bbf,\lf$) 
$$\begin{pmatrix}
 (-1,l-1) & (1,l-1) \\
 (1,l+1) & (2,l+1)
\end{pmatrix},$$
so we get 
that $Q'_1\in\Psi^{-\infty,\calE}_{\phi, \mathscr{H}}(M)$ after possibly adjusting $\calE$ by allowing more log terms at $\lf$. 
More precisely, comparing with Figure \ref{fig-schematic-index} we see that 
$Q'_1$ is strictly higher order at $\bbf$ than $Q_1$, except in the $(\Cmath,\Cmath)$ (lower right) component 
(this information is not used in this paper). \medskip

\noindent \textbf{Step 3:}
We now have $\square_\phi Q_2 = \textup{Id} - R_2$ with $Q_2=Q_1+Q_1'$ and $R_2$ from \eqref{eqn:Q1' eqn}.
Next we remove the interior conormal singularity of the remainder. First, note that $( R_2\Pi, R_2\Piperp)$
has index sets (notation as in \eqref{eqn:index set tuples})
\begin{equation}
 \label{eqn:R2 index sets} 
\Bigl( (1,4,\varnothing,r-1) , (2,4,\varnothing,r+1) \Bigr).
\end{equation}
Since $\square_\phi$ is $\phi$-elliptic,  we may find a 'small' parametrix $Q_\sigma$ 
arising from inverting the $\phi$-symbol, so $\square_\phi Q_\sigma = \Id -R_\sigma$, 
with $Q_\sigma \in \Psi^{-m}_\phi(M)$ and $R_\sigma \in \Psi^{-\infty}_\phi(M)$.
Then we obtain
$$
\square_\phi Q_3=\textup{Id} - R_3, \ \textup{with} \ Q_3:=Q_2 + Q_\sigma R_2, \ R_3:= R_\sigma R_2
$$
Since $Q_\sigma R_2$ has the same index sets as $R_2$, see \eqref{eqn:R2 index sets}, it is 
also $\Psi^{-2,\calE}_{\phi, \mathscr{H}}(M)$ like $Q_2$ and does not contribute to the leading terms 
of $Q_3$ at $\bbf$ and $\phif$. Also, $R_3$ has the same index sets as $R_2$ and in addition is smoothing.
\medskip

A priori, our construction only yields that $Q_3$ lies in the extended $\phi$-calculus. 
The following standard regularity argument shows that $Q_3$ actually lies in the $\phi$-calculus: since $Q_\sigma$ is also a small left parametrix, $Q_\sigma\square_\phi=\textup{Id} - R_{\sigma}'$ with $R_{\sigma}'\in\Psi^{-\infty}_\phi$, we have 
$$ Q_\sigma - Q_\sigma R_3 = Q_\sigma \square_\phi Q_3 = Q_3 - R_{\sigma}' Q_3$$
and since both remainders are smoothing, $Q_3$ has the same singularity as $Q_\sigma$.
\medskip

\noindent \textbf{Step 4:}
The remainder term $R_3$ is smoothing and vanishes at $\lf,\bbf,\phif$, so that the Neumann series $\textup{Id}+R_3+R_3^2+\dots$ can be summed asymptotically by the composition theorem for the full $\phi$-calculus. Denote the sum by $\textup{Id}+S$. Multiplying this from the right gives $Q_\alpha =Q_3(\textup{Id}+S)$ and $R_\alpha=R_3(\textup{Id}+S)$ as required. \medskip

\noindent Finally, we note that $Q_\alpha$ has the same leading terms as $Q_1$, to two orders at $\bbf$ and $\phif$.
This concludes the proof of the theorem.
\end{proof}

\begin{Rem}
 Note that, as in \cite{grieser-para}, we first constructed the boundary parame\-trix and then combined it with the interior parametrix arising from inverting the $\phi$-symbol. This is the opposite order from what is done, e.g., by 
 Vaillant \cite{vaillant} in the same context. Our approach has the advantage of giving control of the $\mathscr{H}-\Cmath$ decomposition of the parametrix both on the domain and range side (the opposite order would give only control on the range side of the right parametrix). This gives more precise information on the parametrix and is needed, for example, for the proof of the Fredholm property stated below.
\end{Rem}
\begin{Rem}\label{Dpi-condition}
The fact that the off-diagonal terms of $\square_\phi$ in \eqref{eqn:square_phi-prelim} are higher order in terms of $x\to0$ is the reason that we impose the assumption $[\Pi,D_B]=0$. Without this assumption the off-diagonal terms would only be in $\Diff^2_\phi(M)\cap x\Diff^2_{\bb}(M)$ (rather than $x^2\Diff^2_\phi(M)$), and it is not clear what the result would be in this case.
\end{Rem}

\subsection{Mapping properties and absence of resonances}
By standard arguments the existence of a parametrix allows to deduce mapping, in particular Fredholm, results for $\square_\phi$, as well as asymptotic information on elements in its kernel. In our context we obtain different regularity for the $\mathscr{H}$ and the $\Cmath$ components. \medskip

\noindent The \emph{split Sobolev space} is the Sobolev space analogue of Definition \ref{def:split phg space} 
\begin{equation}
 \label{eqn:def Hsplit}
\Hsplit^2(M) = x^{-2}\,H^2_{\bb,0}(V;\mathscr{H}) + H^2_\phi(M; \Lambda_\phi M)
\end{equation}
where $H^2_{\bb,0}(V,\mathscr{H})$ is the space of $H^2_{\bb}$ sections of 
$\mathscr{H}$ compactly supported in $V=[0,\varepsilon) \times B$.\footnote{Note that $H^2_{\bb}\subset H^2_\phi\subset x^{-2}H^2_{\bb}$.}
Recall that we always use the measure $\dvol_{\bb}$.
 This space has a natural Hilbert space topology, see
\cite[\S 6.1, Definition 7]{grieser-para}. Note that
\begin{equation}\label{H-split-inclusions}
\begin{split}
&H^2_{\mathscr{H}}(M) \not\subseteq
L^2(M; \Lambda_\phi M), \\
&x^2H^2_{\mathscr{H}}(M) \subseteq
L^2(M; \Lambda_\phi M).
\end{split}
\end{equation}

There are also higher regularity versions of \eqref{eqn:def Hsplit}, defined as $x^{-2}\,H^{2+k}_{\bb,0}(V;\mathscr{H}) + H^{2;k}_{\phi;\bb}(M; \Lambda_\phi M)$ where in the second summand up to two $\phi$-derivatives in addition to up to $k$ b-derivatives are required to lie in $L^2$.\footnote{It would not be natural in our context to define higher split spaces by replacing $2$ by $k$ everywhere in \eqref{eqn:def Hsplit}, since these would not form a scale of spaces.} 
We only formulate the following result for the case $k=0$ since this is all that we need here.

\begin{Cor} \label{cor:grihun fredholm inverse}
The operator $\square_\phi$ is bounded
\begin{equation}
\label{eqn:box cphi bounded}
\square_{\phi}: x^{\alpha+1} \Hsplit^2(M) \to x^{\alpha+1} L^2 (M; \Lambda_\phi M)
\end{equation}
for all $\alpha\in\R$ and Fredholm if $\alpha\not\in \specb(P_{00})$.
Its Fredholm inverse $G_{\phi,\alpha}$ lies in $\Psi^{-2,\calE}_{\phi, \mathscr{H}}(M)$ 
with the same index family $\calE$ as in Theorem \ref{thm:split phi-parametrix}.
The leading asymptotic terms of the (lift of the) Schwartz kernel of $G_{\phi,\alpha}$ are as follows:
\medskip

\begin{enumerate}
\item
at $\bbf$: it is 
$(xx')^{-1}$ times the inverse of $I(P_{00})$, acting on $x^\alpha L^2$, pulled back from the 
b-face of $M^2_{\bb}$ to the $\textup{bf}$-face of $M^2_\phi$. In particular, only the 
$\mathscr{H}\mathscr{H}$ component of the Fredholm inverse $G_{\phi,\alpha}$ is non-zero at $\textup{bf}$.
\medskip

\item
at $\phif$: it is the inverse of the normal operator of $\square_{\phi}$, where the $\mathscr{H}$ part of the inverse is given by \eqref{eqn:inverse Delta}.
\end{enumerate}
\end{Cor}

\noindent The Fredholm inverse was defined after Theorem \ref{thm:b Fredholm regularity}.
We now continue with the proof of Corollary \ref{cor:grihun fredholm inverse}.

\begin{proof}
Boundedness follows easily from \eqref{eqn:square_phi-prelim} and the Fredholm property follows from boundedness and compactness of the parametrix and remainder on the appropriate spaces, see \cite[Theorem 13]{grieser-para}. The shift in weight arises since $Q_{00}$ in \eqref{eqn:Q00 R00} is bounded $x^\alpha L^2\to x^\alpha H^2_{\bb}$, so $x^{-1}Q_{00}x^{-1}: x^{\alpha+1}L^2\to x^{\alpha-1}H^2_{\bb}$. The statement on the Fredholm inverse then follows by standard arguments as in 
\cite[Propositions 5.42 and 5.64]{melrose1993atiyah}. \medskip

The leading terms at $\bbf$ and $\phif$ only arise from the first parametrix $Q_1$ constructed in Step 1 of the proof of Theorem \ref{thm:split phi-parametrix}. The leading (i.e. order $-2$) contribution at $\bbf$ occurs only in the $\mathscr{H}\mathscr{H}$-component and is $x^{-1}Q_{00}x^{-1}$, which implies the claim for $\bbf$. 
The leading (i.e. order $0$) term at $\phif$ is the direct sum of the pull-back of $x^{-1}Q_{00}x^{-1}$ to $\phif$ and of the inverse of the normal operator of $P_{11}$. Since $b\geq2$ by Assumption \ref{assum5}, the singularity of $Q_{00}$ at the diagonal is of the type $\left(|s-1|^2+|y-y'|^2\right)^{-\frac{b-1}2}$ where $s=x/x'$ (recall projective 
coordinates \eqref{sx}). Thus $x^{-1}Q_{00}x^{-1}$ is given
by the half-density (for the same constant $c$ as in \eqref{eqn:inverse Delta})
$$
c \cdot (xx')^{-1}\left(|s-1|^2+|y-y'|^2\right)^{-\frac{b-1}2} \sqrt{\frac{dx}x ds dydy'}.
$$
Pulling back this half-density to $M^2_\phi$ gives in projective coordinates \eqref{SU}
$$
c \cdot \left(|T|^2+|Y|^2\right)^{-\frac{b-1}2} \sqrt{dTdY\frac{dx}{x^2}\frac{dy'}{x^b}},
$$ 
which shows that the standard fundamental solution on $\R^{b+1}$ 
indeed appears at $\phif$ as in \eqref{eqn:inverse Delta}.
\end{proof}

\noindent We now obtain information on the kernel of $\square_\phi$.
\begin{Cor}
 \label{cor:kernel}
Under the Assumption \ref{assum3} the following holds:
for any solution $u\in x^{-1}L^2(M;\Lambda_\phi M)$ to $\square_\phi u=0$, we have that
 $u\in \calA_{\phi,\mathscr{H}}^{\calE(1)-1}(M)$. In particular, $u\in L^2(M;\Lambda_\phi M)$ and thus
 there are no resonances, i.e.
$$
\ker_{x^{-1}L^2(M;\Lambda_\phi M)} \square_\phi = \ker_{L^2(M;\Lambda_\phi M)} \square_\phi.
$$
Equivalently, we have for the unitarily equivalent Hodge Laplacian $\Delta_{\phi}$
 $$
\ker_{x^{-1}L^2(M;;\Lambda_\phi M; \dvol_\phi)} \Delta_{\phi} = \ker_{L^2(M; \Lambda_\phi M;\dvol_\phi)} \Delta_{\phi}.
$$
\end{Cor}
\begin{proof}
 Taking the dual spaces of \eqref{eqn:box cphi bounded} with respect to $L^2$ we get
 $\square_\phi: x^{\alpha-1}L^2 \to x^{\alpha-1} \Hsplit^{-2}$
 (compare \cite[Theorem 13]{grieser-para}) if $\alpha\not\in \specb(P_{00})$. We take $\alpha=0$ and use the left parametrix $Q_0$ from Theorem \ref{thm:split phi-parametrix}. Applying $Q_0$ to $\square_\phi u=0$ we get $u=R'_0u$. Since
$R'_0$ has index sets 
$$
\begin{pmatrix}
\calE(0)_\lf-1 \\ \calE(0)_\lf+1 
\end{pmatrix},
$$
at $\lf$ and $\varnothing$ at all other faces, and since $\calE(0)_\lf=\calE_\lf(1)$ by Assumption \ref{assum3},
we conclude $u\in \calA_{\phi,\mathscr{H}}^{\calE(1)-1}(M)$. The claim on no resonances follows.
\end{proof}

\section{Review of the resolvent construction on scattering manifolds}
 \label{section 3}

The operator $(\Delta_{\phi} + k^2)$ and its unitary transformation
$(\square_{\phi} + k^2)$ are fully elliptic $\phi$-differential operators and invertible for $k>0$, so the Schwartz kernels of their inverses are
polyhomogeneous distributions on $M^2_\phi \times (0,\infty)_k$, where
the $\phi$-double space $M_\phi^2$ is defined in \eqref{phi-double}, with a conormal singularity at $\Diag_\phi\times(0,\infty)$ where $\Diag_\phi$ is the lifted diagonal in $M^2_\phi$. However, that description
is not uniform up to $k=0$. In case of $\dim F = 0$ (scattering manifolds), the behavior of  the resolvent $(\square_{\phi} + k^2)^{-1}$
as $k\to 0$ was analyzed by Guillarmou-Hassell \cite{guillarmou2008resolvent, guillarmou2009resolvent}, 
as well as Guillarmou-Sher \cite{guillarmou2014low}, who define 
a blowup $M^2_{k, \textup{sc}}$ of $\Mbar^2 \times \R^+$, $\R^+=[0,\infty)$,
on which the resolvent is polyhomogeneous and conormal.
\medskip
 
In this section we review this construction and generalize it slightly to obtain a space $M^2_{k,\ssc,\phi}$, which in case of point fibres reduces to $M^2_{k,\ssc}$ and in the general case serves as intermediate step in our construction of the resolvent space $M^2_{k,\phi}$ for $\square_\phi$, which is carried out in Section \ref{sec:proof main thm}.
In Section \ref{section 4.4} we review the results of Guillarmou, Hassell and Sher in the case of point fibres.
\medskip

In this section $M$ is a manifold with fibred boundary, and local coordinates near the boundary are as in Definition \ref{Phi-vector-def}.

\subsection{Blowup of the codimension $3$ corner}

We consider $\Mbar^2 \times \R^+$ with copies of local 
coordinates $(x,y,z)$ and $(x',y',z')$ on the two factors $\Mbar$ near $\partial M$.
We use the following notation for the corners of $\Mbar^2\times \R^+$. 
For any $i_1,i_2 \in \{0,1\}$ we define
$
C_{i_1i_2} := \{x_j = 0 : \ \textup{for all $j$ with } \ i_j = 1\} \subset \Mbar^2
$
and
$$ C_{i_1i_2}^\bullet = C_{i_1i_2}\times \{0\}\,,\quad C_{i_1i_2}^+ = C_{i_1i_2} \times \R^+\,.$$
For example, the highest codimension corner of $\Mbar^2\times \R^+$ is given by
 $$C_{11}^\bullet = \{ x = x'  = k =0\}\,.$$
The blowup of the corner $C_{11}^\bullet$ and blow-down map are denoted by 
$$\beta_1:[\Mbar^2 \times \R^+, C_{11}^\bullet]\to \Mbar^2 \times \R^+$$
This leads to a new boundary hypersurface that we call $\text{bf}_{0}$, as illustrated in Figure 
\ref{figure 4}, where all other variables are omitted. We may introduce local projective coordinates near
each corner of $\text{bf}_{0}$. 
   
\begin{figure}[h]
\centering
\begin{tikzpicture}[scale =0.5]
\draw[-] (0,2)--(0,5);
\draw[-] (-2,-0.5)--(-5,-2);
\draw[-] (2,-0.5)--(5,-2);
\node at (-5.3,-2.3) {$x'$};
\node at (-3,2) {$\text{lb}$};
\node at (3,2) {$\text{rb}$};
\node at (0,-2) {$\text{zf}$};
           \node at (5.3,-2.3) {$x$};
           \node at (0,5.6) {$k$};
           \node at (0,0) {$\text{bf}_{0}$};         
           
 \draw (0,2).. controls (-1,1.5) and (-1.6,0.5) .. (-2,-0.5);           
 \draw (0,2).. controls (1,1.5) and (1.6,0.5) .. (2,-0.5);
 \draw (-2,-0.5).. controls (-1.7,-1) and (1.7,-1) .. (2,-0.5);           
  
\begin{scope}[shift={(13,0)}, scale = 1]
       \draw[->](0,1)--(0,5);
       \draw[->](0,1)--(-4,-1.5);
       \draw[->](0,1)--(4,-1.5);
\node at (-4.5,-1.9) {$x'$};
           \node at (4.5,-1.9) {$x$};
           \node at (0,5.6) {$k$};
  
\end{scope}
\draw[->] (5,1.5)-- node[above] {$\beta_{1}$} (9,1.5);
\end{tikzpicture}
\caption{ Blowup of $C_{11}^\bullet$} \label{figure 4}
\end{figure}

\subsubsection*{Near top corner}
Away from zf we may introduce, 
\begin{equation} \label{top corner}
\xi = \frac{x}{k}, \ \xi' = \frac{x'}{k}, \ y, \ z,  \ y', \ z', \ k.
\end{equation}
Here, $\xi'$ is a local boundary defining function of the right boundary face rb 
(we write $\rho_{\rf} = \xi'$), $\xi$ of the
left boundary face lb (we write $\rho_{\lf} = \xi$), and $k$ of the new boundary face $\text{bf}_{0}$
(we write $\rho_{\bbf_{0}} = k$). 

\subsubsection*{Near right corner}
In a similar way, away from lb we may introduce,
\begin{equation}\label{right corner}
s' = \frac{x'}{x}, \  \kappa = \frac{k}{x}, \ x, \ y, \ z, \ y', \ z'.
\end{equation}
Here, local boundary defining functions are given by
$ \rho_{\bbf_{0}} = x , \rho_{\rf} = s', \rho_{\zf} = \kappa$.
Projective coordinates near the left corner are obtained by interchanging the
roles of $x$ and $x'$, and replacing rb by lb. 

\subsection{Blowup of the codimension $2$ corners}

The next step is to blow up the codimension $2$ corners
that are given by 
\begin{align*}
C_{01}^\bullet = \Mbar \times \partial M \times \{0\}, \quad
C_{10}^\bullet = \partial M \times \Mbar \times \{0\}, \quad
C_{11}^+ = \partial M \times \partial M \times \R^+.
\end{align*}
More precisely we blow up their lifts to  $[\Mbar^2 \times \R^+, C_{11}^\bullet]$,
which we still denote by  $C_{01}^\bullet, C_{10}^\bullet, C_{11}^+$.
This defines 
\begin{equation}\label{M2kb}
M^{2}_{k,\bb} := \left[ \left[\Mbar^2 \times \R^+, C_{11}^\bullet\right], C_{01}^\bullet, C_{10}^\bullet, C_{11}^+\right]
\end{equation}
with blow-down map $\beta_2$ and new front faces  $\lfz$, $\rfz$ and $\bbf$.
This blowup is illustrated in Figure \ref{fig:boat1}.
We keep the notation $\bbf_0$ for the lift of the face $\bbf_0$. \medskip

\begin{figure}[h]
\centering
\begin{tikzpicture}[scale =0.5]
      \draw[-](-1,2)--(-1,5);
       \draw[-](1,2)--(1,5);
       
       \draw[-](-3,-0.5)--(-6,-2);
       \draw[-](-2.5,-2)--(-6,-3.8);
       
           \draw[-](3,-0.5)--(6,-2);
       \draw[-](2.5,-2)--(6,-3.8);

\draw (-1,2).. controls (-1.2,1.7) and (0.8,1.7) .. (1,2);
\draw (-1,2).. controls (-2,1.5) and (-2.5,1) .. (-3,-0.5);
\draw (1,2).. controls (2,1.5) and (2.5,1) .. (3,-0.5);
    
\draw (-3,-0.5).. controls (-2.5,-1) and (-2.5,-2.2) .. (-2.5,-2);
\draw (3,-0.5).. controls (2.5,-1) and (2.5,-2.2) .. (2.5,-2);
\draw (-2.5,-2).. controls (-2,-2.8) and (2,-2.8) .. (2.5,-2);

 \node at (0,0) {$\text{bf}_{0}$};
 \node at (0,-3.5) {$\text{zf}$};
 \node at (0,4) {$\text{bf}$};
 \node at (5,0) {$\text{rb}$};
 \node at (-5,0) {$\text{lb}$};
 \node at (5,-2.5) {$\text{rb}_{0}$};
 \node at (-5,-2.5) {$\text{lb}_{0}$};
                                       
\begin{scope}[shift={(15,0)}, scale = 1]
\draw[-] (0,2)--(0,5);
\draw[-] (-2,-0.5)--(-5,-2);
\draw[-] (2,-0.5)--(5,-2);
\node at (-5.2,-2.6) {$x'$};
           \node at (5.1,-2.6) {$x$};
           \node at (0,5.5) {$k$};
           \node at (0,0) {$\text{bf}_{0}$};

 \draw (0,2).. controls (-1,1.5) and (-1.6,0.5) .. (-2,-0.5);           
 \draw (0,2).. controls (1,1.5) and (1.6,0.5) .. (2,-0.5);
 \draw (-2,-0.5).. controls (-1.7,-1) and (1.7,-1) .. (2,-0.5);           
           
\end{scope}
\draw[->] (5,1.5)-- node[above] {$\beta_{2}$} (10,1.5);
\end{tikzpicture}
\caption{ Blow up of $C_{01}^\bullet, C_{10}^\bullet, C_{11}^+.$}
  \label{fig:boat1}
\end{figure}

\noindent The associated (full) blowdown map to $\Mbar^2 \times \R^+$
is given by 
\begin{equation*}
\beta_\bb = \beta_{2}\circ \beta_{1}: M^{2}_{k,\bb} 
\longrightarrow  \Mbar^2 \times \R^+.
\end{equation*}
We now describe the blowup in terms of projective coordinates
that are valid near some of the intersections of the various 
boundary hypersurfaces in $M^{2}_{k,\bb}$.

\subsubsection*{Projective coordinates near $\bbf_{0} \cap \bbf$}
We use the projective coordinates \eqref{top corner}. The new projective
coordinates near the left corner on the top are now given by 
\begin{equation}\label{ bf cap bf0 co}
\zeta = \frac{\xi}{\xi'} = \frac{x}{x'}, \ \xi' = \frac{x'}{k}, \ y, \ z, \ y', \ z', \ k.
\end{equation}
where $\rho_{\lf} = \zeta, \rho_{\bbf} = \xi', \rho_{\bbf_{0}} = k$. Interchanging the 
roles of $x$ and $x'$ gives a set of coordinates near the left top corner, i.e. near $bf_{0} \cap bf$,
valid away from rf.

\subsubsection*{Projective coordinates near $bf_{0} \cap rb_{0}$}
We use the projective coordinates \eqref{right corner}. The new projective
coordinates in the lower right corner, where $\textup{bf}_{0}, \rf_{0}$ and $\textup{zf}$ meet, are now given by 
\begin{equation}
s' = \frac{x'}{x}, \  \tau = \frac{\kappa}{s'} = \frac{k}{x'}, \ x, \ y, \ z, \ y', \ z'.
\end{equation}
where $\rho_{\zf} = \tau, \rho_{rb_{0}} = s', \rho_{\bbf_{0}} = x$. Projective 
coordinates near the left lower corner, where $\textup{bf}_{0}, \lf_{0}$ and $\textup{zf}$ meet, 
are obtained by interchanging the roles of $x$ and $x'$. Projective coordinates near the other 
corners are obtained similarly. 

\subsection{Blowup of the fibre diagonal} \label{subsec:bup fibre diag}

The final step is to blow up the intersection of the lifted (interior) fibre diagonal, which is locally given by $x=x',y=y'$, with the face $\bbf$:
\begin{equation}
\label{eqn:def diag k sc phi}
\textup{diag}_{k,sc,\phi} := \beta_\bb^*(\diag_{\phi,\textup{int}}\times\R^+) \cap \bbf = \{\zeta=1, \ y = y', \ \xi'=0\}
\end{equation}
in projective coordinates
\eqref{ bf cap bf0 co}.
The new front face is denoted $\ssc$ and the resulting space $$M^{2}_{k , \textup{sc},\phi}:= [M^{2}_{k,\bb} ; \textup{diag}_{k,sc,\phi}]$$ with blow-down map $\beta_3$ 
is illustrated in Figure \ref{fig:whole}. \medskip

\begin{figure}[h]
\centering
\begin{tikzpicture}[scale =0.5]

\draw[-] (-2,2.3)--(-2,5);
\draw[-] (-1,2)--(-1,5);
\draw[-] (1,2)--(1,5);
\draw[-] (2,2.3)--(2,5);
\draw[-] (-3,1)--(-6.5,-0.5);
\draw[-] (3,1)--(6.5,-0.5);
\draw[-] (-2.5,-0.5)--(-6.5,-2.3);
\draw[-] (2.5,-0.5)--(6.5,-2.3);
\draw (-1,2).. controls (-0.7,1.7) and (0.7,1.7) .. (1,2);
\draw (-1,2).. controls (-1.1,1.9) and (-1.8,2.1) .. (-2,2.3);
\draw (1,2).. controls (1.1,1.9) and (1.8,2.1) .. (2,2.3);
\draw (-2.5,-0.5).. controls (-2,-1.5) and (2,-1.5) .. (2.5,-0.5);
\draw (-3,1).. controls (-2.2,0.7) and (-2.2,-0.2) .. (-2.5,-0.5);
\draw (3,1).. controls (2.2,0.7) and (2.2,-0.2) .. (2.5,-0.5);
\draw (-2,2.3).. controls (-2.4,2) and (-2.5,1.7) .. (-3,1);
\draw (2,2.3).. controls (2.4,2) and (2.5,1.7) .. (3,1);

\node at (-4,-0.3) {$\text{lb}_{0}$};
\node at (4,-0.3) {$\text{rb}_{0}$};
\node at (0,4) {$\text{sc}$};
\node at (1.5,3) {$\tiny{\text{bf}}$};
\node at (3,4) {$\text{rb}$};
\node at (-3,4) {$\text{lb}$};
\node at (0,-2) {$\text{zf}$};
\node at (0,0.5) {$\text{bf}_{0}$};

\begin{scope}[shift={(15,1.5)}, scale = 0.9]
 \draw[-](-1,2)--(-1,5);
       \draw[-](1,2)--(1,5);
       
       \draw[-](-3,-0.5)--(-6,-2.1);
       \draw[-](-2.5,-2)--(-6,-3.9);
       
           \draw[-](3,-0.5)--(6,-2.1);
       \draw[-](2.5,-2)--(6,-3.9);

\draw (-1,2).. controls (-1.2,1.7) and (0.8,1.7) .. (1,2);
\draw (-1,2).. controls (-2,1.5) and (-2.5,1) .. (-3,-0.5);
\draw (1,2).. controls (2,1.5) and (2.5,1) .. (3,-0.5);
    
\draw (-3,-0.5).. controls (-2.5,-1) and (-2.5,-2.2) .. (-2.5,-2);
\draw (3,-0.5).. controls (2.5,-1) and (2.5,-2.2) .. (2.5,-2);
\draw (-2.5,-2).. controls (-2,-2.8) and (2,-2.8) .. (2.5,-2);

 \node at (0,0) {$\text{bf}_{0}$};
 \node at (0,-4) {$\text{zf}$};
 \node at (0,4) {$\text{bf}$};
 \node at (4,1) {$\text{rb}$};
 \node at (-4,1) {$\text{lb}$};
 \node at (5,-2.5) {$\text{rb}_{0}$};
 \node at (-5,-2.6) {$\text{lb}_{0}$};
           
\end{scope}
\draw[->] (5,1.5)-- node[above] {$\beta_{3}$} (10,1.5);
\end{tikzpicture}
\caption{Final bowup $M^{2}_{k , \textup{sc},\phi}:= [M^{2}_{k,\bb}, \textup{diag}_{k,sc,\phi}]$.}
  \label{fig:whole}
\end{figure}

\noindent The associated (full) blowdown map to $\Mbar^2 \times \R^+$
is given by 
\begin{equation*}
\beta_{k , \textup{sc},\phi} = \beta_{3} \circ \beta_{2} \circ \beta_{1}: M^{2}_{k , \textup{sc},\phi}
\longrightarrow  \Mbar^2 \times \R^+.
\end{equation*}
In the same pattern as before, 
we may introduce projective coordinates near the intersection $\textup{bf}_{0} \cap \textup{sc}$.
Using the coordinate system \eqref{ bf cap bf0 co}, we define new projective 
coordinates as follows
\begin{equation}\label{bf-sc-coord}
X := \frac{\zeta -1}{\xi'} = kT, \
 U := \frac{y-y'}{\xi'} = kY, \
  \xi'= \frac{x'}{k}, \ y', \ z, \ z', \ k,
\end{equation}
with $(T,Y)$ as in \eqref{SU}.
Here, the boundary defining functions are as follows: 
$\rho_{\textup{bf}_{0}} = k, \rho_{\textup{sc}} = \xi'$ and the 
boundary face $\textup{bf}$ lies in the limit $|(X,U)| \to \infty$.
We illustrate these coordinates in Figure \ref{fig-coord}. \medskip

The lifted diagonal in $M^2_{k,\ssc,\phi}$ is, by definition, the set
$$ \Diag_{k,\ssc,\phi} := \beta_{k,\ssc,\phi}^* (\Diag_\Mbar\times\R^+)\,.$$
  
\begin{figure}[h]
\centering
\begin{tikzpicture}[scale =0.5]

\draw[-] (-2,2.3)--(-2,5);
\draw[-] (1,2)--(1,5);
\draw[-] (-1,2)--(-1,5);
\draw[thick, red, ->] (0,1.7)--(0,4);
\draw[thick, magenta, ->] (0,1.7).. controls (-0.3,1.5) and (-0.3,0.5) .. (-0.3,0.3);

\draw[-] (2,2.3)--(2,5);
\draw[-] (-3,1)--(-6.5,-0.5);
\draw[-] (3,1)--(6.5,-0.5);
\draw[-] (-2.5,-0.5)--(-6.5,-2.3);
\draw[-] (2.5,-0.5)--(6.5,-2.3);
\draw[thick, blue, <->] (-1,2).. controls (-0.7,1.7) and (0.7,1.7) .. (1,2);
\draw (-1,2).. controls (-1.1,1.9) and (-1.8,2.1) .. (-2,2.3);
\draw (1,2).. controls (1.1,1.9) and (1.8,2.1) .. (2,2.3);
\draw (-2.5,-0.5).. controls (-2,-1.5) and (2,-1.5) .. (2.5,-0.5);
\draw (-3,1).. controls (-2.2,0.7) and (-2.2,-0.2) .. (-2.5,-0.5);
\draw (3,1).. controls (2.2,0.7) and (2.2,-0.2) .. (2.5,-0.5);
\draw (-2,2.3).. controls (-2.4,2) and (-2.5,1.7) .. (-3,1);
\draw (2,2.3).. controls (2.4,2) and (2.5,1.7) .. (3,1);

\node at (0,4.5) {\tiny{$\text{sc}$}};
\node at (0,-2) {\tiny{$\text{zf}$}};
\node at (0.7,-0.3) {\tiny{$\text{bf}_{0}$}};
\node[red] at (-0.4,3) {\tiny{$k$}};
\node[blue] at (0.5,1.2) {\tiny{$X$}};
\node[magenta] at (-0.6,1) {\tiny{$\xi'$}};
\node at (1.6,4) {\tiny{$\text{bf}$}};

\end{tikzpicture}
\caption{Illustration of projective coordinates in $M^{2}_{k , \textup{sc},\phi}$.}
\label{fig-coord}
\end{figure}

\begin{Rem} 
For each fixed $k_0> 0$ the level set $\{k=k_0\}$ in the blowup
manifold $M^{2}_{k , \textup{sc},\phi}$ is simply the $\phi$-space
$M_\phi^2$, introduced in \eqref{phi-double}. The face $\text{zf}$ is diffeomorphic to
the b-space $M^{2}_\text{b}$, and in fact in case of trivial
fibres $F$, the Hodge Laplacian $\Delta_{\phi}$ can be reduced to a
b-operator zf, see \eqref{P-square-relation} below. That observation has been crucial in 
the resolvent constructions by Guillarmou, Hassell and Sher.
\end{Rem}

If the fibres of $\phi$ are points then we generally write $\ssc$ instead of $\phi$, e.g.   $\Delta_{\textup{sc}}$, $\scT M$, $M^2_{k,\ssc}$,
$\beta_{k , \textup{sc}}$ instead of $\Delta_\phi$, $\phiT M$, $M^2_{k,\ssc,\phi}$, $\beta_{k , \textup{sc},\phi}$, respectively. This
notation is used e.g. in Definition \ref{def:ksc space} below. \medskip

Note that, if $M$ is a general $\phi$-manifold with trivialization $\Umathbar\cong[0,\eps)\times\dM$ 
near the boundary then $\phi$ defines a fibration $\Umathbar\to V=[0,\eps)\times B$, and this induces a fibration
\begin{equation}
 \label{eqn:F2 fibration}
 \Umath^2_{k,\ssc,\phi} \to V^2_{k,\ssc}
\end{equation}
with fibres $F^2$. Therefore, a distribution on $V^2_{k,\ssc}$ valued in $\End(\Hmath)$ for a 
subbundle $\Hmath\subset \Cinf(F,E)$ and some bundle $E\to\Umathbar$ can instead be considered as a distribution on $\Umath^2_{k,\ssc,\phi}$ valued in $E$.

\subsection{Asymptotics of the resolvent on
scattering manifolds} \label{section 4.4}

We close the section with a short review of the main 
result by Guillarmou-Hassell \cite{guillarmou2008resolvent}, as well as 
Guillarmou-Sher \cite{guillarmou2014low} on the resolvent 
kernel of the Hodge Laplacian in the special case where fibres are points, usually referred to as 
\emph{scattering} manifolds. 
\medskip

Note that, in the case of point fibres, $\Delta_{\textup{sc}}$ reduces to  the top left corner of $\Delta_{\phi}$ in 
\eqref{eqn:square_phi-prelim}. Conjugating $\Delta_{\textup{sc}}$ by $W$ as in \eqref{T} we obtain $\square_{\textup{sc}}$, the top left corner of $\square_\phi$.
Thus, if we define (identify $x$ with the corresponding multiplication operator)
\begin{equation}\label{P-square-relation}
P:= x^{-1} \circ \square_{\textup{sc}} \circ x^{-1}.
\end{equation}
then $P$ is precisely the operator $P_{00}$ in  \eqref{eqn:square_phi-prelim}.
This is an elliptic b-differential operator on $\Mbar$. By \cite[(15)]{guillarmou2014low}, parallel to the formulae in 
\eqref{indicial-operator-2} and \eqref{LL}, we have up to higher order terms coming from the 
higher order terms in the metric
\begin{align}\label{P00-sc}
P = -(x\partial_x)^{2} + L_{\textup{sc}},
\end{align}
where the action of $L_{\textup{sc}}$ on  $\Lambda^* (x^{-1} T^*B) 
\oplus x^{-2} dx \, \wedge \, \Lambda^{*-1} (x^{-1} T^*B)$ is given by
\begin{equation}\label{P00-tangential-sc}
L_{\textup{sc}} = \left( \begin{array}{cc}
\Delta_B + \Bigl(\frac{b-1}2-N_B\Bigr)^{2} &2d_B \\
2d_B^* & \Delta_B + \Bigl(\frac{b+1}2-N_B\Bigr)^{2}
\end{array} \right),
\end{equation}
where $\Delta_B$ is the Hodge Laplacian of $B$, and $d_B$ denotes here the exterior 
differential on $B$. $N_B$ denotes as before the number operator on $\Lambda^*(x^{-1} T^*B)$, 
multiplying elements in $\Lambda^\ell(x^{-1} T^*B)$ by $\ell$. The indicial family $I_\lambda(P)$
and the set of indicial roots $\specb(P)$ are defined as in \S \ref{subsection-I}. 
\medskip

Before we can state the main theorem of Guillarmou and Sher \cite{guillarmou2014low}, 
let us introduce the $(k,\textup{sc})$-calculus of pseudo-differential operators with Schwartz kernels 
lifting to $M^{2}_{k , \textup{sc}}$. As usual we identify the operators with the lifts of their Schwartz kernels.

\begin{defn}\label{def:ksc space}
Let $\Mbar$ be a compact manifold with boundary and $E\to\Mbar$ a vector bundle.
We define small and full $(k,\textup{sc})$-calculi as follows. Consider the 
following density bundle (compare to \eqref{eqn:densities M2phi})
\begin{equation}
\W _{\bb\phi}(M^{2}_{k , \textup{sc}}) := \rho_{\textup{sc}}^{-(b+1)}
\W _{\bb}(M^{2}_{k , \textup{sc}}) = \rho_{\textup{sc}}^{-2(b+1)}
\beta_{k , \textup{sc}}^* \W _{\bb}(\Mbar^2 \times \R^+).
\end{equation}
The corresponding half-density bundle is denoted by\footnote{This is precisely the
half-density bundle $\widetilde{\W }^{\frac{1}{2}}(M^{2}_{k , \textup{sc}})$ introduced in 
\cite[\S 2.2.2]{guillarmou2008resolvent}.} $\W ^{1/2}_{\bb\phi}(M^{2}_{k , \textup{sc}})$.
\begin{itemize}
\item[1.] 
The small $(k,\textup{sc})$-calculus, denoted $\Psi_{k,\textup{sc}}^{m}(M;E)$ for $m\in\R$, is the space of distributions on $M^2_{k,\ssc}$, valued in $\W ^{1/2}_{\bb\phi}(M^{2}_{k , \textup{sc}}) \otimes \End(E)$,
which are conormal of order $m-\frac14$ with 
respect to the lifted diagonal and which vanish to infinite order at all boundary hypersurfaces except  $\bbfz$, $\zf$ and $\ssc$.  \medskip

\item[2.] Consider $(a_{\bbf_0},a_{\zf},a_{sc}) \in \mathbb{R}^{3}$ and an index family $\mathcal{E}$ for $M^{2}_{k , \textup{sc}}$ such that
$\mathcal{E}_{\bbf},\mathcal{E}_{\lf}, \mathcal{E}_{\rf} = \varnothing$.
The full $(k,\textup{sc})$-calculus is then defined by
\begin{equation}
\label{GH-calculus}
\begin{split}
\Psi_{k,\textup{sc}}^{m,(a_{\bbf_0},a_{\zf},a_{sc}),\calE}(M;E) := 
\rho_{\bbfz}^{a_\bbfz}\rho_\zf^{a_\zf}\rho_\ssc^{a_\ssc}\,\Psi_{k,\textup{sc}}^{m}(M;E) + \calA_{k, \textup{sc}}^\calE(M;E),
\end{split}
\end{equation}
where we simplified notation by setting 
\begin{equation}
\begin{split}
\calA_{k, \textup{sc}}^\calE(M;E):= \calA_{\text{phg}}^\calE \Bigl(M^{2}_{k , \textup{sc}}, 
\W ^{1/2}_{\bb\phi} (M^{2}_{k , \textup{sc}})
\otimes \End(E)\Bigr).
\end{split}
\end{equation}

\end{itemize}
\end{defn}

See the remark after Definition \ref{def:conormal distr} concerning the order shift by $\frac14$.

\begin{Th} \label{initial ACM}\cite[Theorem 18]{guillarmou2014low}
Let $(M,g)$  be a scattering (asymptotically conic) manifold, satisfying 
Assumption \ref{assum1}, \ref{assum3} and \ref{assum5}, where
\begin{enumerate}
\item Assumption \ref{assum3} can be replaced with \eqref{noresonance-consequence} and
$0 \notin \specb(P)$. 
\item Assumption \ref{assum5} is equivalent to $\dim M \geq 3$.
\end{enumerate}
Then $(\square_{\textup{sc}} + k^{2})^{-1}$
lies in the full $(k,\textup{sc})$-calculus $\Psi_{k,\textup{sc}}^{-2, \, (-2,0,0),\, \mathcal{E}}(M,\Lambda_\ssc M)$,
where
\begin{align*}
&\mathcal{E}_{\zf} \geq -2, 
\quad \mathcal{E}_{\bbf_{0}} \geq -2, \quad \mathcal{E}_{\ssc} \geq 0,  \\
&\mathcal{E}_{\lf_{0}} = \mathcal{E}_{\rf_{0}} >0.
\end{align*}
Moreover, the leading terms of the resolvent at all boundary 
hypersurfaces of $M^2_{k,\ssc}$ are given by solutions of explicit model problems.
\end{Th}

\section{Low energy resolvent for $\phi$-metrics, proof of main theorem} \label{sec:proof main thm}

Recall that our aim is the construction of the inverse for $\Delta_{\phi} +k^2$, which is a self-adjoint operator in $L^2(M;\Lambda_\phi M;\dvol_\phi)$. We equivalently describe the parametrix construction for
$$
x^{-\frac{b+1}2} \circ (\Delta_{\phi} +k^2) \circ x^{\frac{b+1}2} = \square_\phi + k^2,
$$
which is a self-adjoint operator in $L^2(M;\Lambda_\phi M):=L^2(M;\Lambda_\phi M;\dvol_{\bb})$. In the collar neighborhood 
$\U$ of the boundary, $\square_\phi$ acts with respect to the splitting into 
fibre harmonic forms $\mathscr{H}$ and the perpendicular bundle $\Cmath$ by 
a $2\times 2$ matrix (see \eqref{eqn:square_phi-prelim})
\begin{equation*}
\square_\phi = 
\begin{pmatrix}
xP_{00}x  & xP_{01}x \\
xP_{10}x & P_{11}
\end{pmatrix}.
\end{equation*}
As in the parametrix construction for $\square_\phi$ in Section \ref{review grieser-para}, we will start the construction of a resolvent parametrix for $\square_\phi$ by taking parametrices for the diagonal terms $xP_{00}x+k^2$ and $P_{11}+k^2$. These are 
well-behaved (i.e. polyhomogeneous and conormal) on two different spaces:
\begin{itemize}
 \item
Recall from Section \ref{subsec:split hodge} that $xP_{00}x$ is, up to higher order terms, a Hodge Laplacian on $V=B\times[0,\eps)$ for a scattering metric, twisted by the bundle $\Hmath$. Therefore it has a resolvent parametrix which is 
well-behaved on the space $V^2_{k,\ssc}$ and valued in $\End(\Hmath)$. By \eqref{eqn:F2 fibration} and the subsequent explanation it is therefore a well-behaved distribution on $\Umath^2_{k,\ssc,\phi}\subset M^2_{k,\ssc,\phi}$.
\item
For operators like $P_{11}$, i.e. fully elliptic $\phi$-Laplacians, the resolvent is well-behaved on 
$\Umath^{2}_{\phi} \times \R^+\subset M^{2}_{\phi} \times \R^+$.  
\end{itemize}
These two spaces are illustrated along each 
other for comparison in Figure \ref{comparison-blowup}.
In order to construct the resolvent of $\square_\phi$ we therefore need to find a blowup of
$\Mbar^2 \times \R^+$ that blows down to both $M^{2}_{k , \textup{sc},\phi}$ and $M^{2}_{\phi} \times \R^+$.
We now construct such a space, which we call $M^{2}_{k,\phi}$.

\begin{figure}[h] 

\begin{tikzpicture}[scale = 0.5]      
\draw[magenta, thick, -] (-2,2.3)--(-2,5);
\draw[magenta, thick, -] (-1,2)--(-1,5);
\draw[magenta, thick, -] (1,2)--(1,5);
\draw[magenta, thick, -] (2,2.3)--(2,5);
\draw[-] (-3,1)--(-6.5,-0.5);
\draw[-] (3,1)--(6.5,-0.5);
\draw[-] (-2.5,-0.5)--(-5.7,-1.9);
\draw[-] (2.5,-0.5)--(5.7,-1.9);
\draw[magenta, thick] (-1,2).. controls (-0.7,1.7) and (0.7,1.7) .. (1,2);
\draw[magenta, thick] (-1,2).. controls (-1.1,1.9) and (-1.8,2.1) .. (-2,2.3);
\draw[magenta, thick] (1,2).. controls (1.1,1.9) and (1.8,2.1) .. (2,2.3);
\draw (-2.5,-0.5).. controls (-2,-1.5) and (2,-1.5) .. (2.5,-0.5);
\draw (-3,1).. controls (-2.2,0.7) and (-2.2,-0.2) .. (-2.5,-0.5);
\draw (3,1).. controls (2.2,0.7) and (2.2,-0.2) .. (2.5,-0.5);
\draw (-2,2.3).. controls (-2.4,2) and (-2.5,1.7) .. (-3,1);
\draw (2,2.3).. controls (2.4,2) and (2.5,1.7) .. (3,1);

\node at (0,-3) {$M^{2}_{k , \textup{sc},\phi}$};

\begin{scope}[shift={(15,1)}, scale = 2.5]
        \draw[-](1,-1) --(2,-1.5);
     \draw[-](-1,-1)--(-2,-1.5);
     \draw[magenta, thick, -](1,-1) --(1,1.5);
      \draw[magenta, thick, -](-1,-1)--(-1,1.5);
      \draw[magenta, thick, -](-0.5,-1.3)--(-0.5,1.5);
      \draw[magenta, thick, -](0.5,-1.3)--(0.5,1.5);

         \node at (0,-2) {$M^{2}_{\phi} \times \R^+$};
        \draw[magenta, thick](1,-1).. controls (1,-1.3) and (0.5,-1.3)..(0.5,-1.3);
           \draw[magenta, thick] (-0.5,-1.3).. controls (-0.5,-1.5) and (0.5,-1.5)..(0.5,-1.3);
            \draw[magenta, thick] (-1,-1).. controls (-1,-1.3) and (-0.5,-1.3)..(-0.5,-1.3);        
\end{scope}
       
\end{tikzpicture} 
\caption{Comparison of the blowup spaces $M^{2}_{k , \textup{sc},\phi}$ and $M^{2}_{\phi} \times \R^+$.}
\label{comparison-blowup}
\end{figure}

\subsection{Construction of the blowup space $M^{2}_{k,\phi}$}

The space $M^{2}_{k,\phi}$ is constructed from $M^{2}_{k , \textup{sc},\phi}$ 
by one additional blow-up. Let $\textup{diag}_{k,\phi}$ be the intersection of the lifted interior fibre diagonal with $\bbf_{0}$:
\begin{equation}
\label{eqn:def diag k phi}
\textup{diag}_{k,\phi} := \beta_{k,\ssc,\phi}^*(\diag_{\phi,\textup{int}}\times\R^+) \cap \bbf_0 = \{ X=0, U=0, k=0\}
\end{equation}
in projective coordinates 
\eqref{bf-sc-coord}.
Then we define
\begin{equation}
M^{2}_{k,\phi} := [ M^{2}_{k , \textup{sc},\phi} ; \textup{diag}_{k,\phi}]\,,
\end{equation}
with the blowdown map $\beta_{ \phi - \ssc} : M^{2}_{k, \phi} \longrightarrow M^{2}_{k , \textup{sc},\phi},$ and total blowdown map
\begin{equation*}
 \beta_{k, \phi} =  \beta_{k, \textup{sc},\phi} \circ \beta_{ \phi - \ssc} :\, 
M^{2}_{k,\phi} \longrightarrow \Mbar^{2} \times \R^+.
\end{equation*}
The resulting blowup space is illustrated in Figure \ref{fig:9}. \medskip

\begin{figure}[h]
\centering
\begin{tikzpicture}[scale =0.55]

\draw[-] (-2,2.3)--(-2,5);
\draw[-] (-1,2)--(-1,5);
\draw[-] (1,2)--(1,5);
\draw[-] (2,2.3)--(2,5);
\draw[-] (-3,1)--(-6.5,-0.5);
\draw[-] (3,1)--(6.5,-0.5);
\draw[-] (-2.5,-0.5)--(-6.5,-2.3);
\draw[-] (2.5,-0.5)--(6.5,-2.3);

\draw (-1,2)--(-0.5,1.8);
\draw (1,2)--(0.5,1.8);
\draw (-0.5,1.8).. controls (0,2) and (0,2) .. (0.5,1.8);

\draw (-1,2).. controls (-1.1,1.9) and (-1.8,2.1) .. (-2,2.3);
\draw (1,2).. controls (1.1,1.9) and (1.8,2.1) .. (2,2.3);

\draw (-2.5,-0.5).. controls (-2,-1) and (-0.8,-1.2) .. (-0.5,-1.2);
\draw (2.5,-0.5).. controls (2,-1) and (0.8,-1.2) .. (0.5,-1.2);
\draw (-0.5,-1.2).. controls (-0.2,-1.5) and (0.2,-1.5) .. (0.5,-1.2);

\draw (-3,1).. controls (-2.2,0.7) and (-2.2,-0.2) .. (-2.5,-0.5);
\draw (3,1).. controls (2.2,0.7) and (2.2,-0.2) .. (2.5,-0.5);
\draw (-2,2.3).. controls (-2.4,2) and (-2.5,1.7) .. (-3,1);
\draw (2,2.3).. controls (2.4,2) and (2.5,1.7) .. (3,1);

\draw (-0.5,-1.2).. controls (0,-1) and (0,1.6) .. (-0.5,1.8);
\draw (0.5,-1.2).. controls (1,-1) and (1,1.6) .. (0.5,1.8);

\node at (-4,-0.3) {\tiny{$\text{lb}_{0}$}};
\node at (4,-0.3) {\tiny{$\text{rb}_{0}$}};
\node at (-4,2) {\tiny{$\text{lb}$}};
\node at (4,2) {\tiny{$\text{rb}$}};
\node at (0,-2) {\tiny{$\text{zf}$}};
\node at (-1,1) {\tiny{$\text{bf}_{0}$}};
\node at (0.4,0.6) {\tiny{$\phif_0$}};
\node at (0,3) {\tiny{$\ssc$}};
\node at (-1.5,3) {\tiny{$\bbf$}};

\begin{scope}[shift={(13,0)}, scale = 0.85]
      
\draw[-] (-2,2.3)--(-2,5);
\draw[-] (-1,2)--(-1,5);
\draw[-] (1,2)--(1,5);
\draw[-] (2,2.3)--(2,5);
\draw[-] (-3,1)--(-6.5,-0.5);
\draw[-] (3,1)--(6.5,-0.5);
\draw[-] (-2.5,-0.5)--(-6.5,-2.3);
\draw[-] (2.5,-0.5)--(6.5,-2.3);
\draw (-1,2).. controls (-0.7,1.7) and (0.7,1.7) .. (1,2);
\draw (-1,2).. controls (-1.1,1.9) and (-1.8,2.1) .. (-2,2.3);
\draw (1,2).. controls (1.1,1.9) and (1.8,2.1) .. (2,2.3);
\draw (-2.5,-0.5).. controls (-2,-1.5) and (2,-1.5) .. (2.5,-0.5);
\draw (-3,1).. controls (-2.2,0.7) and (-2.2,-0.2) .. (-2.5,-0.5);
\draw (3,1).. controls (2.2,0.7) and (2.2,-0.2) .. (2.5,-0.5);
\draw (-2,2.3).. controls (-2.4,2) and (-2.5,1.7) .. (-3,1);
\draw (2,2.3).. controls (2.4,2) and (2.5,1.7) .. (3,1);
\draw[dashed] (0,-1.2).. controls (1,-1) and (1,1.6) .. (0,1.8);

\node at (-4,-0.3) {\tiny{$\text{lb}_{0}$}};
\node at (4,-0.3) {\tiny{$\text{rb}_{0}$}};
\node at (-3.5,2) {\tiny{$\text{lb}$}};
\node at (4,2) {\tiny{$\text{rb}$}};
\node at (0,-2) {\tiny{\text{zf}}};
\node at (-1,1) {\tiny{$\text{bf}_{0}$}};

\node at (0,3) {\tiny{$\ssc$}};
\node at (-1.5,3) {\tiny{$\bbf$}};

\end{scope}
\draw[->] (5,1.5)-- node[above] {$\beta_{\phi-\ssc}$} (9,1.5);

\end{tikzpicture}
 \caption{Blowup space $M^{2}_{k , \phi}$}
  \label{fig:9}
\end{figure}

\noindent By construction, the blowup space $M^{2}_{k , \phi}$
  blows down to 
$M^{2}_{k , \textup{sc},\phi}$. However, existence of a blowdown map to 
$M^{2}_{\phi} \times \R^+$ is non-trivial and is the subject of the 
following lemma. 

\begin{Lemma} \label{lemma b-map cons}
The identity in the interior extends to a b-map $$
\beta_{k,\phi}' : M^{2}_{k,\phi} \longrightarrow M^{2}_{\phi}\times \R^+\,.$$
More precisely, $\beta_{k,\phi}'$ is the composition of the blow-down maps for the blow-up of $M^2_\phi\times \R^+$ in the faces $f\times\{0\}$, where $f$ runs through the faces $\bbf$, $\phif$, $\lf$, $\rf$ of $M^2_\phi$, with $\bbf$ blown up first.
The face $\phif_0$ of $M^2_{k,\phi}$ is the lift of the front face of the blow-up of $\phif\times\{0\}$.
\end{Lemma}

\noindent For the proof we use the following result on interchanging the
order of blowups.

\begin{Lemma} \label{lem:commuting blow-ups}
Let $Z$ be a manifold with corners and $A,B\subset Z$ be two
p-submanifolds which intersect cleanly\footnote{This means that near any $p\in A\cap B$ one can find coordinates adapted to the corners 
 of $Z$ so that both $A$ and $B$ are coordinate subspaces locally. This implies that the lift of $B$ under the blow-up of $A$ is again a p-submanifold, and similarly with $A,B$ interchanged. See \cite[Prop. 5.7.2]{Mel-diff}.}. We shall blow up both submanifolds in different order and 
write e.g. $[Z;A,B] := [[Z,A],\widetilde{B}]$, where $\widetilde{B}$ is the lift of 
$B$ under the blow-down map $[Z,A]\to Z$. Then there are natural diffeomorphisms as follows.

\begin{itemize}
 \item[(a)] If $A,B$ are transversal or disjoint, or one is contained in the other, then 
 interchanging the order of blowups of $A$ and $B$ yields diffeomorphic manifolds with corners
 $$[Z;A,B]\cong[Z;B,A].$$

 \item[(b)] In general, interchanging the order of blowups of 
 $A$ and $B$ yields diffeomorphic results if additionally the intersection $A\cap B$ is blown up:
$$ 
[Z; A,B,A\cap B] \cong [Z; B, A, A\cap B].
$$
\end{itemize}
In both cases diffeomorphy holds in the sense that the identity on $Z\setminus(A\cup B)$ extends smoothly to a b-map 
between the two  spaces with smooth inverse.
\end{Lemma}

\begin{proof}
For the  proof of statement (a) we refer to \cite[Lemma 2.1]{analytic surgery}.
For the proof of statement (b), that is stated without proof in \cite[Proposition 5]{Mel-real}, we shall illustrate the idea on a specific example
that models the blowup of $\diag_{k,\phi}$ in $M^2_{k,\textup{sc},\phi}$. Consider $Z:=\R_y \times \R^+_x \times \R^+_z$
with coordinates $(x,y,z)$ as indicated in the lower indices.
We set 
$$
A:= \{x=z=0\} \ \textup{and} \ B:= \{x=y=0\}.
$$ 
Their intersection is $A\cap B = \{x=0, y=0, z=0\}$. We shall consider projective
coordinates on the blowups $[Z; A, B, A \cap B]$ and $[Z; B, A, A \cap B]$.
\medskip

\noindent \textbf{\emph{Blowup $[Z; A, B, A \cap B]$}}:
The blowup is obtained in three steps, blowing up $A$ first, 
then the lift of $B$ second, and finally the lift of the intersection $A\cap B$
as the third step. This is illustrated in Figure \ref{figure 9}. \medskip

\begin{figure}[h]

\begin{tikzpicture}[scale =1]
 \draw[->] (-1,0)--(1,0) node at (1,-0.3) {$y$};
 \draw[->] (0,0)--(0,1) node at (-0.3,1) {$z$};
 \draw[->] (0,0)--(-0.5,-0.5) node at (-0.2,-0.6) {$x$};
 
 \begin{scope}[shift={(3,0)}, scale = 1]
 \draw[-] (0,0.2)--(0,1);
 \draw[-] (0.7,0.2)--(-0.7,0.2);
  \draw[-] (0.7,0)--(-0.7,0);
 \end{scope}
 
 \begin{scope}[shift={(6,0)}, scale = 1]

 \draw[-] (0.5,0.2)--(0.5,1);
 \draw[-] (-0.5,0.2)--(-0.5,1);
 \draw[-] (-0.5,0.2)--(-1.1,0.2);
 \draw[-] (0.5,0.2)--(1.1,0.2);
  \draw[-] (1,0)--(-1,0);

 \draw (-0.5,0.2).. controls (-0.3,0.1) and (0.3,0.1) .. (0.5,0.2);
  \draw (0,0.133)--(0,0);
 \end{scope}

 \begin{scope}[shift={(9,0)}, scale = 1]
 
 \draw[-] (0.5,0.2)--(0.5,1);
 \draw[-] (-0.5,0.2)--(-0.5,1);
 \draw[-] (-0.5,0.2)--(-1.1,0.2);
 \draw[-] (0.5,0.2)--(1.1,0.2);
  \draw[-] (1,0)--(0.3,0);
    \draw[-] (-1,0)--(-0.3,0);

 \draw (-0.5,0.2)--(-0.3,0.15);
  \draw (0.5,0.2)--(0.3,0.15);
  
   \draw (-0.3,0.15).. controls (0,0.2) and (0,0.2) .. (0.3,0.15);
    \draw (-0.3,0).. controls (0,-0.05) and (0,-0.05) .. (0.3,0);
    
 \draw node at (-0.3,0.15) {\tiny$\bullet$};
 
  \draw (0.3,0.15)--(0.3,0);
  \draw (-0.3,0.15)--(-0.3,0);

 \end{scope}
 \draw[->] (1,0.5) -- (1.7,0.5);
  \draw[->] (4,0.5) -- (4.7,0.5);
   \draw[->] (7,0.5) -- (7.7,0.5);

\end{tikzpicture}
\caption{$[Z; A, B, A \cap B]$.}
\label{figure 9}
\end{figure}

\noindent \textbf{\emph{Blowup $[Z; B, A, A \cap B]$}}:
The blowup is obtained in three steps, blowing up $B$ first, 
then the lift of $A$ second, and finally the lift of the intersection $A\cap B$
as the third step. This is illustrated in Figure \ref{figure 8}. \medskip

 \begin{figure}[h]
\begin{tikzpicture}[scale =1]
 \draw[->] (-1,0)--(1,0) node at (1,-0.3) {$y$};
 \draw[->] (0,0)--(0,1) node at (-0.3,1) {$z$};
 \draw[->] (0,0)--(-0.5,-0.5) node at (-0.2,-0.6) {$x$};

 \begin{scope}[shift={(3,0)}, scale = 1]
 \draw[-] (-0.5,0)--(-1,0);
 \draw[-] (0.5,0)--(1,0);
 \draw[-] (0.5,0)--(0.5,1);
 \draw[-] (-0.5,0)--(-0.5,1);
 \draw (-0.5,0).. controls (-0.3,-0.2) and (0.3,-0.2) .. (0.5,0);
 \end{scope}
 
 \begin{scope}[shift={(6,0)}, scale = 1]
 \draw[-] (-0.5,0)--(-1,0);
 \draw[-] (0.5,0)--(1,0);
 \draw[-] (0.5,0.2)--(0.5,1);
 \draw[-] (-0.5,0.2)--(-0.5,1);
 \draw[-] (-0.5,0.2)--(-1.1,0.2);
 \draw[-] (0.5,0.2)--(1.1,0.2);
 \draw (-0.5,0.2).. controls (-0.45,0.1) and (-0.45,0.1) .. (-0.5,0);
 \draw (0.5,0.2).. controls (0.45,0.1) and (0.45,0.1) .. (0.5,0);
 \draw (-0.5,0).. controls (-0.3,-0.2) and (0.3,-0.2) .. (0.5,0);
 \end{scope}

 \begin{scope}[shift={(9,0)}, scale = 1]
 \draw[-] (-0.5,0)--(-1,0);
 \draw[-] (0.5,0)--(1,0);
 \draw[-] (0.5,0.2)--(0.5,1);
 \draw[-] (-0.5,0.2)--(-0.5,1);
 \draw[-] (-0.5,0.2)--(-1.1,0.2);
 \draw[-] (0.5,0.2)--(1.1,0.2);
\draw (-0.5,0.2)--(-0.45,0.1);
\draw (0.5,0.2)--(0.45,0.1);

 \draw (-0.5,0).. controls (-0.5,0.1) and (-0.46,0.1) .. (-0.45,0.1);
 \draw (0.5,0).. controls (0.5,0.1) and (0.46,0.1) .. (0.45,0.1);

 \draw (-0.5,0).. controls (-0.3,-0.2) and (0.3,-0.2) .. (0.5,0);
  \draw (-0.45,0.1).. controls (-0.4,0) and (0.3,0) .. (0.45,0.1);
   \draw node at (-0.45,0.1) {\tiny$\bullet$};
 \end{scope}

 \draw[->] (1,0.5) -- (1.7,0.5);
  \draw[->] (4,0.5) -- (4.7,0.5);
   \draw[->] (7,0.5) -- (7.7,0.5);

\end{tikzpicture}
\caption{$[Z; B, A, A \cap B]$.}
\label{figure 8}
\end{figure}

There is an obvious isomorphism between the face lattices\footnote{i.e.\ the ordered sets whose elements are the faces of the space, ordered by inclusion} of the two spaces.
We shall write out explicitly the projective coordinates near the corners indicated by a bullet in Figures \ref{figure 9} and
\ref{figure 8}. Straightforward computations show that projective coordinates in both blowups are the same 
and given by
\begin{equation}\label{AB-projective}
s_{1} =  \frac{zy}{x}, \quad s_{2} =\frac{x}{y}, \quad s_{3} = \frac{x}{z}.
\end{equation}

\noindent These projective coordinates  are illustrated in Figure \ref{figure project}. \medskip

 \begin{figure}[h]

\begin{tikzpicture}[scale =2.7]

 \begin{scope}[scale = 1]
 
   \draw node at (-1.2,0.7) {$[Z; A, B, A \cap B]$};
   
 \draw[-] (-0.5,0.2)--(-0.5,1);
  \draw[-] (-0.5,0.2)--(-1.1,0.2);
  \draw[-] (-1,-0.5)--(0.5,-0.5);
  \draw (-0.1,-0.2)--(-0.1,-0.2);
  \draw (-0.2,0.1)--(-0.2,-0.5);

\draw (-0.2,0.1).. controls (-0.25,0.1) and (-0.48,0.16) .. (-0.5,0.2);
\draw (-0.2,0.1).. controls (-0.15,0.15) and (-0.05,0.15) .. (0,0.1);
 \draw (-0.2,0.1) circle [radius=0.1];
 
 \draw[blue,thick, ->] (-0.3,0).. controls (-0.35,0) and (-0.58,0.09) .. (-0.6,0.1) node at (-0.4,-0.1) {$s_{1}$};
 \draw[red,thick, ->] (-0.1,0.2).. controls (-0.05,0.25) and (0.05,0.22) .. (0.1,0.2) node at (0,0.35) {$s_{2}$};
 \draw[brown,thick,<-] (-0.08,-0.27)--(-0.08,0) node at (0.05,-0.16) {$s_{3}$};
 \end{scope}

 \begin{scope}[shift={(2.5,0)}, scale = 1]
 
  \draw node at (-1.2,0.7) {$[Z; B, A, A \cap B]$};

 \draw[-] (-0.5,0.2)--(-0.5,1);
 \draw[-] (-0.5,0.2)--(-1.1,0.2);
 \draw[-] (-1,-0.5)--(0.5,-0.5);

\draw (-0.3,-0.2)--(0.3,-0.2);
\draw (-0.3,-0.2).. controls (-0.4,-0.3) and (-0.32,-0.52) .. (-0.3,-0.5);
\draw (-0.5,0.2).. controls (-0.3,0) and (-0.3,-0.15) .. (-0.3,-0.2);

  \draw (-0.3,-0.2) circle [radius=0.1];
  
  \draw[brown,thick, ->]  (-0.45,-0.15).. controls (-0.55,-0.25) and (-0.455,-0.44) .. (-0.45,-0.45) node at (-0.59,-0.3) {$s_{3}$};
  \draw[blue,thick, <-]  (-0.35,0.25).. controls (-0.15,0.05) and (-0.15,-0.1) .. (-0.15,-0.15) node at (-0.05,0) {$s_{1}$};
   \draw[red,thick,->] (-0.18,-0.27)--(0.1,-0.27) node at (-0.03,-0.35) {$s_{2}$};

 \end{scope}
 
\end{tikzpicture}
\caption{Coordinates in corners of $[Z; A, B, A \cap B]$ and $[Z; B, A, A \cap B]$.}
\label{figure project}
\end{figure}
Similarly we may check that projective coordinates near any pair of corresponding 
corners of $[Z; A, B, A \cap B]$ and $[Z; B, A, A \cap B]$
coincide, proving the statement in this model case. The general case is studied along the same lines. 
\end{proof}

\noindent We can now prove Lemma \ref{lemma b-map cons}.
\begin{proof} [Proof of Lemma \ref{lemma b-map cons}] We proceed in the notation of 
\S \ref{section 3} and denote for example by $C_{11}^\bullet$ the highest codimension corner in
$\Mbar^2 \times \R^+$, and  $C_{11}^+ = \partial M\times\partial M\times \R^+$; we will use the same notation for the lifts. 
We will apply 
Lemma \ref{lem:commuting blow-ups} with
$$
Z=[\Mbar^2\times\R^+;C_{11}^+] = M^2_{\bb}\times\R^+, \quad 
A=C_{11}^\bullet, \quad B=\textup{diag}_\phi\times\R^+
$$ 
where $\textup{diag}_\phi$ is the fibre diagonal in the b-face of $M^2_{\bb}$, see \eqref{phi-double}.
In local projective coordinates $(s=x/x', x', y,y',z,z',k)$ near the resulting front 
face in $Z$, we have $A = \{x'=k=0\}$ and $B=\{x'=0, s=1, y=y'\}$, which shows that $A,B$ intersect cleanly {
(but not transversally). 
Now $[Z;A,*]=M^2_{k,\bb}$ (where $*$ denotes the left and right edges $C_{10}^\bullet,C_{01}^\bullet$) and the lift of $B$ to this space is $\diag_{k,sc,\phi}$ and disjoint from $*$, so $[Z;A,B,*]= M^2_{k,sc,\phi}$. The lift of $A\cap B$ to this space is $\diag_\phi$ and disjoint from $*$, so we obtain}
\begin{equation}
\label{eqn:M2kphi}
\begin{split}
\Bigl[ [Z;A,B,A\cap B]; C_{10}^\bullet,C_{01}^\bullet\Bigr]  = M^2_{k,\phi}\,.
\end{split}
\end{equation}
On the other hand, $[Z;B]=M^2_\phi\times\R^+$, so we also have
\begin{equation}
 \label{eqn:M2phitimes R}
 \Bigl[ [Z;B,A,A\cap B]; C_{10}^\bullet,C_{01}^\bullet\Bigr] = \Bigl[ [M^2_\phi \times \R^+,A,A\cap B]; C_{10}^\bullet,C_{01}^\bullet\Bigr] \,.
\end{equation}
By Lemma \ref{lem:commuting blow-ups} the spaces in \eqref{eqn:M2kphi}, \eqref{eqn:M2phitimes R} coincide so the claim follows since the lifts of $A$, $A\cap B$, $C_{10}^\bullet$, $C_{01}^\bullet$ are $f\times\{0\}$ with $f=\bbf, \phif, \lf,\rf$. 
\end{proof}

We can apply the construction of the $(k,\phi)$ double space to the boundary neighborhood $\Umathbar\cong [0,\eps)\times\dM$ instead of $M$. We can also apply it to the space $V=[0,\eps)\times B$, with the trivial fibration that has point fibres. Then, similar to \eqref{eqn:F2 fibration}, the fibration $\Umathbar\to V=[0,\eps)\times B$ induces a fibration
\begin{equation}
 \label{eqn:F2 fibration phi}
 \Umath^2_{k,\phi} \to V^2_{k,\phi}
\end{equation}
with fibres $F^2$. Therefore, a distribution on $V^2_{k,\phi}$ valued in $\End(\Hmath)$ for a subbundle $\Hmath\subset \Cinf(F,E)$ and some bundle $E\to\Umathbar$ can instead be considered as a distribution on $\Umath^2_{k,\phi}$ valued in $E$.

\subsection{Definition of the $(k,\phi)$-calculus}\label{subsec:def psi k phi}
The lifted diagonal $\textup{Diag}_{k,\phi}$, i.e. the closure of the preimage 
of $\{(p,p,k)\mid p\in M, k>0\}\subset \Mbar^2\times\R^+$ under the map 
$\beta_{k,\phi}:M^2_{k,\phi}\to \Mbar^2\times\R^+$, 
is a p-submanifold of $M^2_{k,\phi}$ and hits its boundary only in the faces $\ssc$, $\phif_0$ and $\zf$. 
The Schwartz kernel of the operator $\square_\phi+k^2$ lifts to $M^2_{k,\phi}$ to be conormal to 
$\textup{Diag}_{k,\phi}$\,, uniformly to the boundary with a non-vanishing delta type singularity, when written as a section of the half-density bundle
\begin{equation}
\W ^{\frac{1}{2}}_{\bb\phi}(M^{2}_{k , \phi}) := \rho_{\textup{sc}}^{\frac{-(b+1)}{2}}  
\rho_{\phif_0}^{\frac{-(b+1)}{2}} \W ^\frac12_{\bb}(M^{2}_{k , \phi}).
\end{equation}
This is because the same is true for $\square_\phi$ on $M^2_\phi$, hence on $M^2_\phi\times\R^+$, 
with respect to its diagonal $\textup{diag}_\phi\times\R^+$, and this diagonal hits $\bbf\times\{0\}$ transversally, 
so blowing up that face (which results in $\phif_0$ away from $\bbf_0$) does not affect conormality. 
Equivalently, in local coordinates we blow up $\{x'=k=0\}$, and this does not affect conormality with 
respect to $\{(T,Y)=0,z=z'\}$. We also use here that the half-density bundles are compatible in the 
sense that 
$$
\W ^\frac12_{\bb\phi}(M^2_{k,\phi}) =(\beta_\phi')^*\Bigl(\W ^\frac12_{\bb\phi}(M^2_\phi)
\otimes\W ^\frac12_{\bb}(\R^+)\Bigr).
$$
This motivates the following definition. As usual we define operators 
(in our case, families of operators on $M$ depending on the parameter $k>0$) 
by their Schwartz kernels and identify kernels on ${M}^2\times(0,\infty)$ with 
those lifted to the interior of $M^2_{k,\phi}$.

\begin{defn}\label{def:psi k phi}
Let $\Mbar$ be a compact manifold with fibred boundary and $E\to\Mbar$ a vector bundle.
We define small and full $(k,\phi)$-calculi as follows. 

\begin{itemize}
\item[1.]  The small $(k,\phi)$-calculus, denoted by $\Psi_{k,\phi}^m(M;E)$
for $m\in\R$, is the set of distributions on $M^2_{k,\phi}$, valued in 
$\W ^\frac12_{\bb\phi}(M^{2}_{k , \phi})  \otimes \End(E)$, 
which are conormal of order\footnote{See the remark after Definition \ref{def:conormal distr} concerning the order shift by $\frac14$.} 
$m-\frac14$ with respect to the lifted diagonal
and which vanish to infinite order at all boundary hypersurfaces except  $\phif_0$, $\zf$ and $\ssc$.

\item[2.] Consider $(a_{\phif_0},a_{\zf},a_{\ssc}) \in \mathbb{R}^{3}$,
and an index family $\mathcal{E}$ for $M^{2}_{k , \phi}$ satisfying
$ \calE_{\bbf}=\calE_{\lf}=\calE_{\rf}=\varnothing$. The full $(k,\phi)$-calculus is  defined by
\begin{equation}
\Psi_{k,\phi}^{m,(a_{\phif_0}, a_\zf,a_\ssc),\calE}(M;E) = 
\rho_{\phif_0}^{a_{\phif_0}}\rho_\zf^{a_\zf}\rho_\ssc^{a_\ssc}\,\Psi_{k,\phi}^m(M;E) + \calA^{\calE}_{k,\phi}(M;E),
 \end{equation}
where we simplified notation by setting 
\begin{equation*}
\begin{split}
 \calA^{\calE}_{k,\phi}(M;E):= \calA_{\text{phg}}^\calE \Bigl(M^{2}_{k , \phi}, 
\W ^{1/2}_{\bb\phi} (M^{2}_{k , \phi})
\otimes \End(E)\Bigr).
\end{split}
\end{equation*}
If the triple $(a_{\phif_0}, a_\zf,a_\ssc) = (0,0,0)$, we simply write $\Psi_{k,\phi}^{m,\calE}(M;E)$.
\end{itemize}
\end{defn}

Note that, as in Definition \ref{def:ksc space}, the numbers $a_{\phif_0}, a_\zf,a_\ssc$ refer to the behavior of the 
conormal singularity at the boundary faces, while the index family $\calE$ describes the boundary behavior 
of the smooth part. These two behaviors are allowed to, and in general will, be different.
If $E=\Lambda_\phi M$ then 
we also need a split version of the $(k,\phi)$-calculus, analogous to Definition \ref{def:split phi calc}, 
because of the different behavior of $\mathscr{H}$- and $\Cmath$-valued sections. 

\begin{defn}[Split $(k,\phi)$-calculus]\label{def:split k phi calculus}
 Let $\calE$ be an index family for $M^2_{k,\phi}$ as above.
 Let $\calA^\calE_{k,\phi,\mathscr{H}}(M)$ be the space of sections $K\in \calA^{\calE}_{k,\phi}(M;\Lambda_\phi M)$ which satisfy the conditions in Definition \ref{def:split phi calc}, with $\bbf$, $\lf$, $\rf$ and $\phif$ replaced by $\bbf_0$, $\lf_0$, $\rf_0$ and 
 $\phif_0$, respectively. The split $(k,\phi)$-calculus is defined as 
 $$ 
 \Psi^{m,\calE}_{k,\phi,\mathscr{H}}(M) = \Psi^{m}_{k,\phi}(M;\Lambda_\phi M) + 
 \calA^\calE_{k,\phi,\mathscr{H}}(M).
 $$
The space $\Psi_{k,\phi,\calH}^{m,(a_{\phif_0}, a_\zf,a_\ssc),\calE}(M)$ is defined in an analogous way.

\end{defn}

\subsection{Initial parametrix construction on $M^{2}_{k,\phi}$}
\label{subsec:intl param}

\begin{theorem}\label{initial resl phi}
There exists a resolvent parametrix $G(k)\in \Psi^{-2,\calE}_{k,\phi,\mathscr{H}}(M)$,
such that 
$$
(\square_\phi+k^2) G(k) = \textup{Id} - R(k),
$$ 
with remainder $R(k)\in \Psi^{-\infty,(1,1,1),\calR}_{k,\phi,\mathscr{H}}(M)$
where the index set $\calE$ satisfies
\begin{equation}
 \label{eqn:ind set k phi para}
 \begin{split}
&\calE_{\lfz}=\calE(-1)_\lf-1,\ \calE_\rfz=\calE(-1)_\rf-1,
\\ &\calE_\bbfz\geq-2,\ \calE_{\phif_0}\geq0,\ \calE_\ssc=0,\ \calE_\zf=-2,
\end{split}
\end{equation}
with $\calE(-1)$ determined by $\spec_{\bb}(P_{00})$ and satisfying \eqref{eqn:index sets} 
with $P=P_{00}$, $\alpha=-1$, and an index set $\calR$ positive at all faces. 
Moreover, $\calE_f$ and $\calR_f$ are empty at $f=\bbf,\lf,\rf$. Note that by Assumption 
\ref{assum3}, $\calE(-1)_\lf= \calE(1)_\lf > 1$ and similarly at $\rf$. Hence
\begin{equation}
 \label{eqn: lf rf positive}
\calE_{\lfz}>0,\quad \calE_\rfz>0.
\end{equation}

\end{theorem}

The proof of this result occupies this subsection.
Constructing $G(k)$ requires solving model problems at the various boundary 
hypersurfaces of $M^2_{k,\phi}$, leading to the construction of the leading terms at these faces.
We first find the leading terms at $\zf$ and then at the faces $\ssc$, $\bbf_0$, $\phif_0$ lying over  $\partial M\times\partial M$.  Along the way we check that the constructions match near all intersections, and also with a small parametrix at the diagonal.
\medskip

The construction at zf is global on $M\times M$ and uses Corollary \ref{cor:grihun fredholm inverse}. The construction at the faces $\ssc$, $\bbf_0$, $\phif_0$ closely follows the route taken in Step 1 in the proof of Theorem \ref{thm:split phi-parametrix}, with $\square_\phi$ replaced by $\square_\phi+k^2$ and $M^2_\phi$ by $M^2_{k,\phi}$.
In the intermediate steps of the construction we need the extended calculus as explained after Theorem \ref{thm:split phi-parametrix}. The definition given there carries over to the $(k,\phi)$ because $\Umath^2_{k,\phi}\to V^2_{k,\phi}$ is a bundle with fibres $F^2$, see \eqref{eqn:F2 fibration phi}. The lifting property \eqref{eqn:lifting} also carries over and now reads (leaving out bundles) 
\begin{equation}
 \label{eqn:lifting k phi}
 \beta_{\phi-\ssc}^*: \Psi_{k,\ssc}^{-m}(V) \to \rho_{\phif_0}^m\Psi_{k,\phi}^{-m}(V)+\calA_{k,\phi}^{\calF_m}(V),
 \quad \calA_{k,\ssc}^{\calF}(V) \to \calA_{k,\phi}^{\calF'}(V)
\end{equation}
for any index set  $\calF$ for $M^2_{k,\ssc}$,
where $\calF_{m,\bbf_0}=0$, $\calF_{m,{\phif_0}}=m \,\overline\cup\,(b+1)$, and 
$\calF_{m,f}=\varnothing$ at all other faces $f$, and where $\calF'$ has the same index sets as 
$\calF$ at all faces that already exist on $V^2_{k,\ssc}$, and in addition $\calF'_{\phif_0} = \calF_{\bbf_0}+(b+1)$. Note that, since the boundary fibration for $V$ has point fibres, we have $V^2_{k,\ssc}= V^2_{k,\ssc,\phi}$.
\medskip

\noindent We also need to pull back via $\beta_{k,\phi}':M^2_{k,\phi}\to M^2_\phi \times \R^+$:
\begin{equation}
 \label{eqn:lifting k phi2}
 (\beta_{k,\phi}')^*: C^\infty(\R^{+}_{k},\Psi^{-m}_\phi(M)) \to \Psi^{-m}_{k,\phi}(M)\,.
\end{equation}
This follows from the fact that $\phif_0$ arises as front face of the blow-up of the corner 
$\phif\times\{0\}\subset M^2_\phi\times\R^{+}_{k}$ (see Figure \ref{figure 8}) and the fact that this face is transversal to the diagonal in $M^2_\phi\times\R^+$.
\medskip

In the following construction we denote the leading term of order $m$ at a face $f$ by $G_m(f)$. That is, the resolvent behaves like $\rho_f^m G_m(f) + o(\rho_f^m)$ near the interior of $f$, and similarly if there are several leading terms. Here $\rho_f$ is a defining function for the interior of $f$. We use $k$ as interior defining function for all faces 'at $k=0$' i.e. $\zf$, $\lf_0$, $\bbf_0$, $\phif_0$. 
At $\zf$ we need to construct two leading terms, $G_{-2}(\zf)$ and $G_0(\zf)$.

\subsubsection{Leading terms at $\zf$} \ \medskip
We shall define (a fibred cusp operator)
$$
\square_{\mathrm{c}\phi} := x^{-1}\square_\phi x^{-1}.
$$ From Corollary \ref{cor:grihun fredholm inverse} it follows that the following operators are Fredholm:
\begin{equation}
\label{eqn:square phi cphi maps}
\begin{split}
&\square_\phi: \ H^2_{\mathscr{H}}(M) 
\to L^2(M;\Lambda_\phi M), \quad \textup{if} \ -1 \notin \specb(P_{00}) \\
&\square_{\mathrm{c}\phi}: \ x^2H^2_{\mathscr{H}}(M) 
\to L^2(M;\Lambda_\phi M), \quad \textup{if} \ 0 \notin \specb(P_{00}). 
\end{split}
\end{equation}
We abbreviate $L^2=L^2(M;\Lambda_\phi M)$ and $H^2_{\mathscr{H}} = H^2_{\mathscr{H}}(M)$.
Assume $0\not\in \specb(P_{00})$\footnote{This is of course a consequence of Assumption \ref{assum3}.}. 
Recall that by \eqref{H-split-inclusions} only the
second map in \eqref{eqn:square phi cphi maps} is an operator in $L^2$. Let $G_{\mathrm{c}\phi}$ be the Fredholm inverse of $\square_{\mathrm{c}\phi}$ in $L^2$. Because $\square_{\mathrm{c}\phi}$ is self-adjoint in $L^2$ we have
\begin{equation}\label{cphi-Fredholm}
\square_{\mathrm{c}\phi}G_{\mathrm{c}\phi} = \textup{Id} - \Pi_{\mathrm{c}\phi},
\end{equation}
where $\Pi_{\mathrm{c}\phi}$ is the orthogonal projection in $L^2$ onto the $L^2$-kernel of $\square_{\mathrm{c}\phi}$, which has finite dimension.
Conjugating by $x$ and inserting $x\cdot x^{-1}$ between $\square_{\mathrm{c}\phi}$ and $G_{\mathrm{c}\phi} $ we obtain
\begin{equation}\label{GP1}
\square_{\phi} x^{-1} G_{\mathrm{c}\phi} x^{-1} = \textup{Id} - x \Pi_{\mathrm{c}\phi} x^{-1}.
\end{equation}
This is an identity in $x L^2$, since the operators on the left hand side in \eqref{GP1} map 
$$
xL^2 \xrightarrow{x^{-1} G_{\mathrm{c}\phi} x^{-1} } x \Hsplit^2 \xrightarrow{\square_\phi} xL^2.
$$ 
We now show that \eqref{GP1} is also an identity in $L^2$ (without weight $x$!). 
First, we consider the right hand side of \eqref{GP1}. 
\begin{Lemma} \label{lem:ker piphi}
If  Assumption \ref{assum3} is satisfied, then $x\Pi_{\mathrm{c}\phi}x^{-1}$ extends to a projection%
\footnote{Note that the projection $x\Pi_{\mathrm{c}\phi}x^{-1}$ is not orthogonal.} in $L^2$, and the orthogonal 
projection $\Pi_\phi$ of $L^2$ onto the $L^2$-kernel of $\square_{\phi}$ satisfies
$$
\ker \Pi_\phi = \ker \left( x \Pi_{\mathrm{c}\phi} x^{-1} \right).
$$
\end{Lemma}
\begin{proof}
 We have by Assumption \ref{assum3}
\begin{equation}
 \label{eqn l2 xl2}
 \ker_{L^2} \square_{\mathrm{c}\phi}  = \ker_{L^2} \square_\phi x^{-1} = x\ker_{x^{-1}L^2} \square_\phi = x\ker_{L^2}\square_\phi
\end{equation}
where the last equality is by Corollary \ref{cor:kernel} that follows from Assumption \ref{assum3}.
In particular, $\ker_{L^2} \square_{\mathrm{c}\phi}\subset xL^2$. Writing the Schwartz kernel of $\Pi_{\mathrm{c}\phi}$ as $\sum_{j=1}^N \phi_j\otimes\overline\phi_j$ for an orthonormal basis $(\phi_j)$ of $\ker_{L^2} \square_{\mathrm{c}\phi}$, we see that the Schwartz kernel of $x\Pi_{\mathrm{c}\phi}x^{-1}$ equals $\sum_{j=1}^N (x\phi_j)\otimes(x^{-1}\overline\phi_j)$, and from $x^{-1}\phi_j\in L^2$ we conclude that this extends to an operator on $L^2$, and that (for $u \in L^2$)
$$ u\in \ker \left( x \Pi_{\mathrm{c}\phi} x^{-1} \right) \iff u \perp x^{-1}\ker_{L^2} \square_{\mathrm{c}\phi}.$$
Now \eqref{eqn l2 xl2} gives $ x^{-1}\ker_{L^2} \square_{\mathrm{c}\phi}=\ker_{L^2}\square_\phi$, and the claim follows from $\ker_{L^2}\square_\phi = \Ran\Pi_\phi=(\ker\Pi_\phi)^\perp$ by the following sequence of identities
$$
\ker \Pi_\phi = \left( \ker_{L^2} \square_\phi \right)^\perp
= \left( x^{-1} \ker_{L^2} \square_{\mathrm{c}\phi} \right)^\perp = \ker \left(x \Pi_{\mathrm{c}\phi} x^{-1}\right).
$$
\end{proof}
Next, we consider the left hand side of \eqref{GP1}. 
Since $x^{-1}G_{\mathrm{c}\phi}x^{-1}$ is a parametrix of $\square_\phi$ in $xL^2$ with finite rank remainder, the argument in Corollary \ref{cor:grihun fredholm inverse}, with $\alpha=0$, shows that its index set at $\rf$ is $\calE(0)_\rf - 1$. Now  $x^{-1}G_{\mathrm{c}\phi}x^{-1}$ extends to $L^2$ iff its index set at $\rf$ is positive, and by \eqref{eqn:index sets} this is equivalent to
\begin{equation}
\label{eqn:-1,0 condition} 
 [-1,0]\cap \specb(P_{00})=\varnothing.
\end{equation}
Thus by Assumption \ref{assum3}, \eqref{GP1} is indeed an identity on $L^2$.
In view of \eqref{GP1} as an identity on $L^2$ and Lemma \ref{lem:ker piphi} we can obtain a formula for a Fredholm inverse of 
$\square_\phi$, using the following  simple functional analytic result, relating projections and orthogonal projections. 

\begin{Lemma}\label{funkana}
Let $H_1$ and $H_2$ be Hilbert spaces, $P:H_1\to H_2$ and $G:H_2 \to H_1$ be operators
such that $P\circ G=\textup{Id} - \Pi$ for a continuous projection $\Pi$ in $H_2$. Then the operator $\textup{Id} - \Pi + \Pi^*$ in $H_2$ is invertible, and the 
orthogonal 
projection $\Pi_o$ in $H_2$ with $\ker \Pi = \ker \Pi_o$ is given by
\begin{equation}\label{Po}
\Pi_o = \Pi^* \circ (\textup{Id} - \Pi + \Pi^*)^{-1}.
\end{equation}
Moreover, setting $G_o:= G \circ (\textup{Id} - \Pi + \Pi^*)^{-1}$, we have
\begin{equation}\label{Go}
P\circ G_o=\textup{Id} - \Pi_o.
\end{equation}
\end{Lemma}

\begin{proof}
Since $\Pi$ is a projection, $\ker \Pi = \textup{Ran} \, (\textup{Id} - \Pi)$.
The property $\ker \Pi = \ker \Pi_o$ implies that $\Pi_o$ also vanishes on 
$\textup{Ran} \, (\textup{Id} - \Pi)$, so 
$$\Pi_o \circ (\textup{Id} - \Pi) = 0. 
$$
Next, $\Pi_o$ is the identity on the orthogonal complement $(\ker \Pi_o)^\perp = (\ker \Pi)^\perp = 
\textup{Ran} \, \Pi^*$. Thus $\Pi_o \circ \Pi^* = \Pi^*$. Adding this to the property
$\Pi_o \circ (\textup{Id} - \Pi) = 0$ above, we conclude
$$
\Pi_o \circ (\textup{Id} - \Pi + \Pi^*) = \Pi^*.
$$
The operator $\Pi - \Pi^*$ is skew-adjoint, thus has purely imaginary spectrum. 
Consequently, $\textup{Id} - \Pi + \Pi^*$ is invertible and \eqref{Po} follows.
Finally, \eqref{Go} follows by a straightforward computation with $S= \textup{Id} - \Pi + \Pi^*$:
$$
P\circ G_o=P \circ G \circ S^{-1} = 
\left( \textup{Id} - \Pi\right) \circ S^{-1}
= \left( S - \Pi^*\right) \circ S^{-1} = \textup{Id} - \Pi^* \circ S^{-1}.
$$
\end{proof}

\noindent We apply Lemma \ref{funkana} by setting (in the notation therein)
$$
P=\square_\phi, \ G=x^{-1}G_{\mathrm{c}\phi}x^{-1}, \ \Pi=x\Pi_{\mathrm{c}\phi}x^{-1},
H_1=\Hsplit^2, H_2=L^2.
$$ 
We then obtain a Fredholm inverse $G_o$, which we denote $G_\phi$, to $\square_\phi$ by setting (note $\Pi_{\mathrm{c}\phi}^*=\Pi_{\mathrm{c}\phi}$)
\begin{equation}\label{Gphi}
G_\phi:= x^{-1} G_{\mathrm{c}\phi} x^{-1} \circ (\textup{Id} - x \Pi_{\mathrm{c}\phi} x^{-1} + x^{-1} \Pi_{\mathrm{c}\phi} x)^{-1}.
\end{equation}

\begin{Rem}
Formulas like \eqref{Po} have appeared in the literature before, e.g. in \cite[Lemma 3.5]{BLZ09}.
Our functional analytic approach is different from the 
approach by \cite{guillarmou2014low} and \cite{guillarmou2008resolvent}, where 
the Fredholm para\-metrix is obtained by algebraic computation with bases of 
$\ker_{L^2} \square_{\phi}$.
\end{Rem}

\noindent We can now define the two leading terms of the 
resolvent parametrix at zf as
\begin{equation}\label{Gkphi}
G_{-2}(\zf) = \Pi_\phi\,,\quad G_0(\zf) = G_\phi\,.
\end{equation}
Note that $\square_{\phi} G_{\phi} = \textup{Id} -  \Pi_{\phi}$ and $\square_\phi \Pi_\phi=0$ imply that this defines indeed a parametrix near $\zf$:
\begin{equation}
\label{eqn:zf param eqn} 
 (\square_\phi+k^2)(k^{-2} \Pi_\phi + G_\phi) = \textup{Id} + O(k^2)\,.
\end{equation}

\subsubsection{Leading terms at $\ssc$, $\bbf_0$, $\phif_0$} \ \medskip
In this step we will construct a parametrix
\begin{equation}
 \label{eqn:Q_1 k phi para}
Q_1\in \Psibar_{k,\phi,\mathscr{H}}^{-2,\calE'} (M),
\end{equation}
by defining it near the boundary faces $\ssc$, $\bbf_0$, $\phif_0$ of $M^2_{k,\phi}$ and extending it to $M^2_{k,\phi}$ using a cutoff function, 
with remainder term vanishing at the boundary, see \eqref{eqn:R1 k phi para} below for the precise
statement. Here $\calE'$ coincides with the index set $\calE$ in 
\eqref{eqn:ind set k phi para} except at $\zf$, where $\calE'_\zf=0$.
We follow Step 1 in the proof of Theorem \ref{thm:split phi-parametrix}, 
replacing $\square_\phi$ by $\square_\phi+k^2$ and taking the weight $\alpha=-1$. 
That is, with $\square_\phi$ written as in \eqref{eqn:square_phi-prelim} and in the 
notation of the proof of that theorem we now set 
$$
A=xP_{00}x+k^2, \ B=xP_{01}x, \ C=xP_{10}x, \ D=P_{11}+k^2.
$$
\underline{\textit{Diagonal terms in \eqref{Q1}}:} \medskip

\noindent Comparing \eqref{P00-sc}, \eqref{P00-tangential-sc} with \eqref{indicial-operator-2}, \eqref{LL}
shows in view of \eqref{P-square-relation} that we can apply the 
resolvent construction of Guillarmou and Hassell to find a parametrix $Q_{00,k}$ for $A=xP_{00}x+k^2$,
but only near the boundary, i.e. on $V^2_{k,\ssc}$. That is, we use the solutions of the model 
problems at the diagonal, at $\ssc$ and $\bbf_0$ (but not at $\zf$) to obtain
\begin{equation}
\begin{split}
&(xP_{00}x+k^2)Q_{00,k} = \textup{Id}-R_{00,k},\\
&Q_{00,k}\in\Psi_{k,\ssc}^{-2,(-2,0,0),\calE_0}(V,\Hmath), \
R_{00,k}\in \rho_{\bbf_0} \rho_\ssc\calA_{k,\ssc}^{\calE_0}(V,\Hmath),
\end{split}
\end{equation}
where $\calE_0$ is the index set $\calE$ but without the $\phif_0$ part.
By \eqref{eqn:lifting k phi} $Q_{00,k}$ lifts to an element of $\Psi^{-2,\calE}_{k,\phi}(V,\Hmath)\subset\Psibar^{-2,\calE}_{k,\phi}(\U)$ (we leave out the bundle $\Lambda_\phi M$ from notation in this proof).
In the notation of \eqref{Q1} we set $\widehat{A}= Q_{00,k}$ and $R = R_{00,k}$.\medskip

Next we construct a parametrix for $D=P_{11}+k^2$.
As discussed after \eqref{eqn:square phi normal op}, the normal operator of $P_{11}=\Piperp\square_\phi\Piperp$ is invertible, and by  \cite[Proposition 2.3]{grieser-para} its inverse lies in the extended suspended calculus of $M^2_\phi$. The argument given there shows that the same holds for $P_{11}+k^2$, with smooth dependence on $k\geq0$. So the arguments of loc.\ cit.\ apply to give  $Q_{11,k} \in\Cinf(\R_{k}^+,\Psibar^{-2}_\phi(\U))$, $R_{11,k}\in\Cinf(\R_{k}^+,\calA^{\varnothing}_\phi(\U))$, satisfying $(P_{11}+k^2)Q_{11,k}=\textup{Id} - R_{11,k}$ as an identity in $\Cinf(V\times(0,\infty),\Cmath)$.
Pulling these operators back under the map $\beta'_{\phi} : \U^{2}_{k,\phi} \longrightarrow \U^{2}_{\phi}\times \R^+$ we obtain by the extended analogue of \eqref{eqn:lifting k phi2}
$$ Q_{11,k}\in \Psibar^{-2}_{k,\phi}(\U),\ R_{11,k} \in \calA^{\calF}_{k,\phi}(\U)$$ 
where $\calF_\zf=0$ and $\calF_f=\varnothing$ at all other faces $f$. \smallskip

\noindent \underline{\textit{Off-diagonal terms in \eqref{Q1}}:} \medskip

\noindent With $\widehat{A}=Q_{00,k}$ and $\widehat{D}=Q_{11,k}$ we get
$B'=xP_{01}x Q_{11,k}\in x^2\Psibar^{0}_{k,\phi}(\U)=\Psibar^{0}_{k,\phi}(\U) x^2$ and $C'=xP_{10}x Q_{00,k} \in x^2\Psibar_{k,\phi}^{0,\calE}(\U)$. The extra $x^2$ factors in $B',C'$  give $Q_1\in \Psibar_{k,\phi,\mathscr{H}}^{-2,\calE}(\U)$.
\medskip

\noindent \underline{\textit{Remainder}:}
\medskip

\noindent The analysis of the remainders is analogous to the proof of Theorem \ref{thm:split phi-parametrix}. For index families we use a notation analogous to \eqref{eqn:index set tuples}:
\begin{equation}
 \label{eqn:index set tuples k phi}
\begin{split}
&\calF = (a,b,c,l+\lambda,r+\rho) :\iff \\
&\calF_{\bbf_0}\geq a,\ \calF_{\phif_0}\geq b,\ \calF_\ssc\geq c,\calF_\lfz=\calE(0)_\lf+\lambda,\ \calF_\rfz=\calE(0)_\rf+\rho
\end{split}
\end{equation}
and $\calF_\zf=0$ and all other index sets empty.
By \eqref{eqn:lifting k phi} $R_{00,k}$ lifts to $\calA_{k,\phi}^{(1,4,1,l+1,r-1)}(\U)$.
Then $R_{00,k}B'\in\calA_{k,\phi}^{(3,6,3,l+1,r+1)}(\U)$ and $R_{11,k}C' \in \calA_{k,\phi}^{(\varnothing,\varnothing,\varnothing,\varnothing,r-1)}(\U)$. Also, 
\begin{align*}
&B'C'\in x^4\Psibar_{k,\phi}^{0,\calE}(\U)=
\rho_{\phif_0}^4\rho_\ssc^4 \Psibar_{k,\phi}^{0,(2,0,0,l+3,r-1)}(\U), \\
&C'B'\in x^2\Psibar_{k,\phi}^{0,\calE}(\U)x^2 = \rho_{\phif_0}^4
\rho_\ssc^4\Psibar_{k,\phi}^{0,(2,0,0,l+1,r+1)}(\U).
\end{align*}
In summary we get 
$(\square_\phi+k^2)Q_1=\textup{Id} - R_1$ where 
$R_1\in\rho_{\phif_0}^4\rho_\ssc^4\Psibar_{k,\phi}^{0,\calR_1}(\U)$ with
\begin{equation}
\label{eqn:R1 k phi para}
 \calR_1 = 
\begin{pmatrix}
 (1,0,0,l+1,r-1) & (3,2,-1,l+1,r+1) \\
 (\varnothing,\varnothing,\varnothing,\varnothing,r-1) & (2,0,0,l+1,r+1)
\end{pmatrix}
\end{equation}
This has positive index sets at $\bbf_0,\ssc,\phif_0$.
\medskip

\noindent \underline{\emph{Matching at the intersections of boundary faces}:}
\medskip

\noindent The terms at $\ssc,\bbfz, \phif_0$ match between each other and the faces $\bbf, \lf, \rf$ by construction. 
We now show that the terms at $\bbfz$ and $\phif_0$ also match with the leading terms at $\zf$. \medskip

First, 
the coefficient of the $k^{-2}$ term is the orthogonal projection $\Pi_\phi$ to  $K=\ker_{L^2}\square_\phi$, whose integral kernel is $\sum_{i=1}^N \psi_i \otimes\overline\psi_i$ (times b-half densities) for an orthonormal basis $(\psi_i)$ of $K$. By Corollary \ref{cor:kernel}, $\psi_i\in \calA^F_{\mathscr{H}}(M)$ for an index set $F>0$, and this implies
$\Pi_\phi\in\calA^\calF_{\phi,\mathscr{H}}(M)$ with $\calF$ positive at all faces (and even with an index set $>b+1$ at $\phif$ by \eqref{eqn:densities M2phi}). Since the leading orders at $\bbf_0$ and $\phif_0$ in $\calE$ are $-2$, $0$ respectively, it follows that $k^{-2}\Pi_\phi=(\kappa')^{-2}(x')^{-2}\Pi_\phi$ (recall from \eqref{right corner} that $\kappa'= k/x'$) is lower order than the leading terms of $G_1(k)$ at these faces, in each of the components of the $\mathscr{H}-\Cmath$ decomposition.
\medskip

Next, the $k^0$ coefficient at $\zf$ is $G_\phi$ in \eqref{Gphi}. The fact that $G_{\mathrm{c}\phi}$ is a Fredholm inverse of $\square_{\mathrm{c}\phi}$ implies as in the proof of Corollary \ref{cor:grihun fredholm inverse} that $G_{\mathrm{c}\phi}$ is pseudodifferential, and more precisely that $x^{-1}G_{\mathrm{c}\phi}x^{-1}\in \Psi^{-2,\calE_0}_{\phi, \mathscr{H}}(M)$ with $\calE_0$ the index set for $\zf$ induced by $\calE$ in \eqref{eqn:ind set k phi para}, and has the same leading terms at the intersection with $\bbf_0$ and $\phif_0$ as  $Q_1(k)$ restricted to $k=0$, in each component of the $\mathscr{H}-\Cmath$ decomposition.
Also, the terms $S=x \Pi_{\mathrm{c}\phi} x^{-1} + x^{-1} \Pi_{\mathrm{c}\phi} x$ vanish  at the boundary as in the argument above for $\Pi_\phi$, so $(\textup{Id}+S)^{-1}=\textup{Id}+S'$ where $S'$ has the same vanishing orders by standard arguments, so the factor $(\textup{Id}+S)^{-1}$ does not change the leading term.

\subsubsection{Singularity at the diagonal} 
The fact that  the distribution kernel of $\square_\phi+k^2$ has a delta type singularity on $\textup{diag}_{k,\phi}$, uniformly and non-vanishing at the boundary, means that its $(k,\phi)$-principal symbol, which is an endomorphism of $N^*\textup{diag}_{k,\phi}$, is uniformly invertible. By inverting this symbol and applying the standard parametrix construction one obtains a 'small' parametrix $Q_\triangle$ satisfying
$$ (\square_\phi+k^2)Q_\triangle = \textup{Id} - R_\triangle,\quad 
Q_{\triangle}\in \Psi_{k,\phi}^{-2}(M),\ 
R_{\triangle}\in \Psi_{k,\phi}^{-\infty}(M).
$$

The same argument as in the proof of Theorem \ref{thm:split phi-parametrix} shows two things: that the parametrix constructed so far, which was only in the extended $\phi$-calculus near the boundary, is actually in the $\phi$-calculus itself, and that $Q_\Delta$ can be adjusted to match with the parametrices at the faces hit by the diagonal.

\subsubsection{Remainder term of initial parametrix}
We choose our initial parametrix $G(k)$ to be an element of $\Psi^{-2,\calE}_{k,\phi,\mathscr{H}}(M)$ which matches the models at $\zf$, $\ssc$, $\bbf_0$, $\phif_0$ as explained above. We now analyze the term $R(k)$ in $(\square_\phi+k^2)G(k) = \textup{Id} - R(k)$ and in particular track the different terms in the $\mathscr{H}-\Cmath$ decomposition. 
\medskip

First, consider $\lf_0$ near its intersection with $\zf$, away from $\bbf_0$. Here 
$$
G(k) = k^{-2}G_{-2}(\zf) + G_0(\zf) + \Gtilde,
$$ 
where $\Gtilde$ has index sets $\mathcal{L},1$ at $\lfz$, $\zf$, respectively, with $\mathcal{L}=
\begin{pmatrix}
 \calE_\lfz \\ \calE_\lfz+2
\end{pmatrix}$. Since $\kappa$ (recall $\kappa=k/x$ as defined in \eqref{right corner}) defines $\zf$ near $\zf \, \cap \, \lfz$, this means $\Gtilde=\kappa G'$ with $G'$ having index sets $\mathcal{L},1$ at $\lfz$, $\zf$, respectively. Then by \eqref{eqn:zf param eqn}
$$
(\square_\phi+k^2)G(k) = \textup{Id}+k^2G_0(\zf) + (\square_\phi+k^2)(\kappa G').
$$ 
The main term here is $\square_\phi\kappa G' = k \, \square_\phi k^{-1}\kappa G' = \kappa x\, \square_\phi x^{-1} G'$. Now $x\, \square_\phi x^{-1}$ has the same structure \eqref{eqn:square_phi-prelim} as $\square_\phi$, and $x$ defines $\lfz$, so applying it to $G'$ yields index set
\begin{equation}
\label{eqn:tropical product} 
\begin{pmatrix}
 2 & 2 \\ 2 & 0
\end{pmatrix}
\otimes
\begin{pmatrix}
  \calE_\lfz \\ \calE_\lfz+2
\end{pmatrix}
= 
\begin{pmatrix}
  \calE_\lfz +2\\ \calE_\lfz+2
\end{pmatrix}
\end{equation}
where $\otimes$ is the \emph{tropical} matrix product, which is the usual matrix product with $+$ replaced by $\cup$ (respectively replacing $+$ by '$\min$' for lower bounds on index sets) and $\cdot$ by $+$. \medskip

The index set of $\kappa G'$, and hence of $R(k)$, at $\rfz$ is the same as that of $G'$ since $P$ does not act on the local defining function $x'$ of $\rfz$, so it is $
\begin{pmatrix}
 \calE_\rfz & \calE_\rfz+2
\end{pmatrix}
$.
Now \eqref{eqn: lf rf positive} implies that $R(k)\in\calA^{\calR}_{k,\phi,\mathscr{H}}(M)$ near $(\lf_0\cup\rfz)\cap\zf$ with $\calR_\lf,\calR_\rf$ positive. Similar arguments and \eqref{eqn:R1 k phi para}  show that 
$R(k)\in\Psi^{-\infty,(1,1,1),\calR}_{k,\phi,\mathscr{H}}(M)$ with $\calR$ positive at all faces.
This finished the proof of Theorem \ref{initial resl phi}.

\begin{Rem}
 Our Assumption \ref{assum3} implies that the remainder term, after constructing model solutions at $\zf$, $\ssc$, $\bbfz$, $\phif_0$, is positive at all faces, including $\lfz$ and $\rfz$. For this reason, and since we do not require finer information for the Riesz transform, we do not need a construction of leading terms at these side faces as in the works by 
 Guillarmou, Hassell and Sher. 
\end{Rem}

\subsection{Statement and proof of the main result} \label{section 9}

Now we have all tools in place to finish our microlocal construction of the resolvent 
for the Hodge Laplacian $\Delta_{\phi}$ on $\phi$-manifolds 
at low energy. Recall that we work under the rescaling \eqref{T} and thus $\Delta_{\phi}$
is replaced with the unitarily equivalent operator $\square_{\phi}$ acting in $L^2(M;\Lambda_\phi M)$.
We recall the initial parametrix $Q(k)$ for $(\square_{\phi} + k^{2})$, 
constructed in Theorem \ref{initial resl phi}
$$
(\square_{\phi} + k^{2}) Q(k) = \textup{I} - R(k).
$$
In order to invert the right hand side we begin with a lemma that is parallel to 
\cite[Corollary 2.11]{guillarmou2008resolvent}.

\begin{Lemma} \label{hilbert}
For $N > \dim M$, $R(k)^{N}$ is Hilbert-Schmidt for
each $k > 0$, with Hilbert-Schmidt norm
\begin{equation*}
\lVert R(k)^{N} \rVert_{\text{HS}} \longrightarrow 0,
\ \textup{as} \ k \longrightarrow 0.
\end{equation*}
\end{Lemma}
 
\begin{proof}
Since the Schwartz kernel for the error term $R(k)$ is a polyhomogeneous 
conormal distribution when lifted to $M^2_{k,\phi}$, vanishing to positive order at all boundary faces, 
there exists a positive lower bound $\varepsilon > 0$ for all its index sets.
Then the Composition Theorem \ref{compositionth} implies that $R(k)^N$ has index sets that are bounded below by 
$N\varepsilon>0$. Since the order of $R(k)$ as a pseudo-differential operator is $(-1)$, the order of the conormal singularity of $R(k)^N$ is $(-N)<-\dim M$, so its Schwartz kernel is continuous across the lifted diagonal.
Thus $R(k)^N$ is a Hilbert Schmidt operator in $L^2(M;\Lambda_\phi M)$.
 Finally, its Hilbert-Schmidt norm tends to zero as $k\to0$ since its Schwartz kernel vanishes at the $k=0$ faces  zf, $\lf_0$ and $\rf_0$.
\end{proof}

Thus $\textup{I} - R(k)$ is invertible as an operator in $L^2(M;\Lambda_\phi M)$ for $k>0$ sufficiently small.
We can now state and prove our main theorem.

\begin{theorem}[Main Theorem] \label{thm:main thm detail}
The resolvent $(\square_{\phi} + k^{2})^{-1}$ is an element of the split $(k,\phi)$-calculus
$\Psi_{k, \phi,\calH}^{-2,\calE}(M)$, defined in Definition \ref{def:split k phi calculus}, where the individual index sets satisfy
\begin{equation}\label{lower-bounds-index-sets}
\mathcal{E}_{\textup{sc}} \geq 0,\quad
 \mathcal{E}_{\phif_0} \geq 0, \quad 
 \mathcal{E}_{\textup{bf}_0} \geq -2, \quad \mathcal{E}_{\lf_0}, \mathcal{E}_{\rf_0} >0, 
\quad \mathcal{E}_{\textup{zf}} \geq -2.
\end{equation}
The leading terms at $\ssc$, $\phif_0$, $\bbfz$ and $\zf$ are of orders $0,0,-2,-2$, respectively, and are given by the constructions in Section \ref{subsec:intl param}.
\end{theorem}

\begin{proof}
Fix $N>\dim M$. 
By Lemma \ref{hilbert} there is $k_0>0$ so that  $R(k)^N$ has operator norm less than one for $k\in(0,k_0]$. 
Therefore, $\textup{Id} - R(k)^N$ and hence 
$\textup{Id} - R(k)$ is invertible for these $k$, with inverse given by the Neumann series $\sum_{j=0}^\infty R(k)^j$. 
As in the proof of Lemma \ref{hilbert}, all index sets of $R(k)^{j}$ are bounded below by $j\eps$, and since this tends to $\infty$ as $j\to\infty$, it follows by standard arguments that $(\textup{Id} - R(k))^{-1}=\textup{Id}+S(k)$ where $S(k)$ lies in the calculus with the same lower bounds for the index sets as $R(k)$. \medskip

By Proposition \ref{initial resl phi} the initial parametrix $Q(k)$ satisfies the claims of the theorem. Hence the same holds for $(\square_\phi+k^2)^{-1} = Q(k)(\textup{Id}+S(k))$. 
Since $S(k)$ has positive index sets everywhere, the leading terms of $(\square_\phi+k^2)^{-1}$ are the same as those of $Q(k)$.
This proves the statement when $k$ is restricted to $k\leq k_0$. Since 
$(\square_{\phi} + k^{2})$ is fully elliptic for all $k>0$, with smooth dependence on $k$, the statement holds for all $k$.
\end{proof}

\section{Triple space construction and composition theorems} \label{section 6}

The following results hold for any vector bundle $E$. 
We applied these composition results above for the particular case
where $E= \Lambda {}^\phi T^*M$.

\begin{Th}  \label{compositionth-smooth}
Consider operators $A$ and $B$ with integral kernels lifting to 
\begin{equation}
\begin{split}
&\beta_{k,\phi}^*K_{A} \in \mathcal{A}_{\textup{phg}}^{\mathcal{E}}(M^2_{k,\phi}, \W _{\bb\phi}^{1/2} 
\otimes \End(E)) = \mathcal{A}^{\mathcal{E}}_{k,\phi}, \\
&\beta_{k,\phi}^*K_{B} \in \mathcal{A}_{\textup{phg}}^{\mathcal{F}}(M^2_{k,\phi}, \W _{\bb\phi}^{1/2} 
\otimes \End(E)) =  \mathcal{A}^{\mathcal{F}}_{k,\phi}.
\end{split}
\end{equation}
Assume that both lifts vanish to infinite order at $\bbf$, $\lf$ and $\rf$. Then the composition of operators $A \circ B$
is well-defined and has integral kernel lifting to 
\begin{equation}
\label{eqn:AB}
\beta_{k,\phi}^*K_{A \circ B} \in \mathcal{A}_{\textup{phg}}^{\Cmath}(M^2_{k,\phi}, \W _{\bb\phi}^{1/2} 
\otimes \End(E)) =  \mathcal{A}^{\Cmath}_{k,\phi}.
\end{equation}
Furthermore, the analogous statement holds for the split spaces:
\begin{equation}
 \label{eqn:AB split}
\beta_{k,\phi}^*K_{A}\in \calA^\calE_{k,\phi,\Hmath}\,,\ \beta_{k,\phi}^* K_{B} \in\calA^\calF_{k,\phi,\Hmath} \ \Longrightarrow\ 
\beta_{k,\phi}^*K_{A\circ B} \in \calA^{\Cmath}_{k,\phi,\Hmath}.
\end{equation}
\noindent The index family $\Cmath$ is given by 
\begin{equation}\label{composition-index-family}
\begin{split}
&\Cmath_{\phif_0} = \left( \mathcal{E}_{\textup{bf}_0} + \mathcal{F}_{\textup{bf}_0} + (b+1)\right) 
\overline{\cup} \left( \mathcal{E}_{\lf_0} + \mathcal{F}_{\rf_0} + (b+1)\right) 
\overline{\cup} \left( \mathcal{E}_{\phif_0} + \mathcal{F}_{\phif_0} \right), \\
&\Cmath_{\textup{bf}_{0}} = (\mathcal{E}_{\textup{bf}_{0}} + \mathcal{F}_{\textup{bf}_{0}})
\overline{\cup} (\mathcal{E}_{lb_{0}} + \mathcal{F}_{rb_{0}}) \overline{\cup} (\mathcal{E}_{\phif_0} + \mathcal{F}_{\textup{bf}_0})
\overline{\cup} (\mathcal{E}_{\textup{bf}_0} + \mathcal{F}_{\phif_0}), \\
&\Cmath_{\rf_{0}} =
(\mathcal{E}_{\textup{zf}} + \mathcal{F}_{rb_{0}})
\overline{\cup} (\mathcal{E}_{\rf_{0}}+ \mathcal{F}_{\textup{bf}_{0}})
\overline{\cup} (\mathcal{E}_{\rf_{0}}+\mathcal{F}_{\phif_0}),\\
&\Cmath_{\lf_{0}} = (\mathcal{E}_{\lf_0}+
\mathcal{F}_{\textup{zf}}) \overline{\cup}
(\mathcal{E}_{\textup{bf}_{0}}+\mathcal{F}_{\lf_{0}})  \overline{\cup}
(\mathcal{E}_{\phif_0}+ \mathcal{F}_{\lf_{0}}),\\
&\Cmath_{\textup{zf}} = (\mathcal{E}_{\textup{zf}}+ \mathcal{F}_{\textup{zf}}) \overline{\cup} 
(\mathcal{E}_{\rf_{0}}+ \mathcal{F}_{\lf_{0}}), \\
&\Cmath_{\textup{bf}} =
\Cmath_{\lf} = \Cmath_{\rf}
= \varnothing, \\
&\Cmath_{\textup{sc}} = (\mathcal{E}_{\textup{sc}}  + \mathcal{F}_{\textup{sc}}).
\end{split}
\end{equation}
\end{Th}
{Note that no integrability condition on the index sets at the left and right boundary faces is needed since composition, hence integration, is done only for fixed $k>0$, and Schwartz kernels are assumed to vanish to infinite order at the side faces for $k>0$.}

\begin{proof}
The Schwartz kernel $K_{A \circ B}$ may be expressed using projections
\begin{align*}
\pi_{L} &:\Mbar^3\times \R^+ \longrightarrow \Mbar^2\times \R^+, \quad (p,p',p'',k) \mapsto (p',p'',k),\\
\pi_{C} &:\Mbar^3\times \R^+ \longrightarrow \Mbar^2\times \R^+, \quad (p,p',p'',k) \mapsto (p,p'',k),\\
\pi_{R} &:\Mbar^3\times \R^+ \longrightarrow \Mbar^2\times \R^+, \quad (p,p',p'',k) \mapsto (p,p',k).
\end{align*}
With this notation we can write, provided the pushforward is well-defined,
$$
K_{A \circ B} = (\pi_{C})_{*} \Bigl( \pi_{R}^{*}K_{A}\otimes \pi_{L}^{*}K_{B} \Bigr).
$$
To prove the theorem, we need to define a triple space $M^{3}_{k,\phi}$ given by a blowup of $\Mbar^3\times \R^+$, 
such that the projections $\pi_{L}, \pi_{C}, \pi_{R}$ lift to b-fibrations 
$\Pi_{L}, \Pi_{C}, \Pi_{R}$ on the triple space $M^{3}_{k,\phi}$. More precisely, 
writing $\beta_{k,\phi}^2 \equiv \beta_{k,\phi}: M^{2}_{k,\phi} \to \Mbar^2 \times \R^+$ for the blowdown map on $M^{2}_{k,\phi}$,
we are looking for a space $M^{3}_{k,\phi}$ and a smooth map 
$\beta_{k,\phi}^3: M^{3}_{k,\phi} \to \Mbar^3\times \R^+$,
such that the following diagram commutes for $* \in \{ L,C,R\}$.
 \begin{center}
 \begin{tikzcd} M^{3}_{k,\phi}\arrow[r, "\Pi_{*}" ]\arrow["\beta^{3}_{k,\phi}",d]& M^{2}_{k,\phi}
 \arrow[d, "\beta^{2}_{k,\phi}"] \\\Mbar^3\times \R^+
 \arrow[r,  "\pi_{*}"]& \Mbar^2\times \R^+
 \end{tikzcd}
 \end{center}
Once we know that the projections lift to b-fibrations, we deduce that 
\begin{equation}
\label{eqn:composition phi formula}
\beta_{k,\phi}^* K_{A \circ B} = (\Pi_{C})_{*} \Bigl( \Pi_{R}^{*} \left(\beta_{k,\phi}^* K_{A} \right) \cdot
\Pi_{L}^{*} \left( \beta_{k,\phi}^* K_{B}\right) \Bigr),
\end{equation}
is polyhomogeneous on $M^{3}_{k,\phi}$ by the pullback and pushforward theorems of Melrose.
We proceed in four steps.

\begin{itemize}
\item Step 1: Construct the triple space $M^{3}_{k,\phi}$.
\item Step 2: Show that  projections $\pi_{L}, \pi_{R}, \pi_{C}$
lift to b-fibrations $\Pi_{L}, \Pi_{R}, \Pi_{C}$.
\item Step 3: Compute the index sets \eqref{composition-index-family}.
\item Step 4: Prove \eqref{eqn:AB split}.
\end{itemize}

\subsection*{Step 1: Construction of the triple space $M^{3}_{k,\phi}$.}
We constructed the double space $M^2_{k,\phi}$ from $\Mbar^2\times\R^+$ by first doing a b- (or total boundary) blow-up, then blowing up the fibre diagonal in $k>0$, denoted $\diag_{k,sc,\phi}$, and finally the fibre diagonal at $k=0$, denoted $\diag_{k,\phi}$. (Here fibre diagonal always means the boundary of the interior fibre diagonal, lifted to the b-space.)
Similarly, we will construct the triple space $M^3_{k,\phi}$ by first doing a b-blow-up, then blowing up the preimages of the fibre diagonal in $k>0$ under the three projections, and finally blowing up the preimages of the fibre diagonal at $k=0$. In both cases the three preimages intersect, so their intersection is blown up first. \medskip

We now provide the details. 
We first set up a notation for the various corners in $\Mbar^3\times \R^+$, similar to the notation for the double space. We write $(x,x',x'')$
for the defining functions on the three copies of $\Mbar$. As before, $k\geq 0$ is the coordinate on the
$\R^+$-component.
Now for any triple of binary indices $i_1,i_2,i_3 \in \{0,1\}$ we define
$
C_{i_1i_2i_3} := \{x_j = 0 : \ \textup{for all $j$ with } \ i_j = 1\} \subset \Mbar^3
$
and
$$ C_{i_1i_2i_3}^\bullet = C_{i_1i_2i_3}\times \{0\}\,,\quad C_{i_1i_2i_3}^+ = C_{i_1i_2i_3} \times \R^+\,.$$
For example the highest codimension corner in $\Mbar^3\times \R^+$ is given by
$$C_{111}^\bullet = \{ x = x' = x'' = k =0\},$$
We will slightly abuse notation below by denoting the lifts of $C_{i_1i_2i_3}^{\ast}$ as $C_{i_1i_2i_3}^{\ast}$ again for $\ast\in\{\bullet,+\}$ and $i_1,i_2,i_3\in\{0,1\}$.
\medskip

We will also sometimes use this systematic notation for the boundary hypersurfaces of the double spaces $M^2_{k,sc,\phi}$ and $M^2_{k,\phi}$. Thus, for example,  $\bbf=C_{11}^+$ and $\bbf_0=C_{11}^\bullet$. For the fibre diagonals we also write  $\diag_{k,sc,\phi}=C_{\phi\phi}^+$ and $\diag_{k,\phi}=C_{\phi\phi}^\bullet$.

\subsubsection* {Construction of the b-triple space}

The space  $\Mbar^3\times \R^+$ has boundary faces of codimensions 4, 3, 2 and 1. We do the total boundary blow-up, i.e. we first blow up the codimension 4 corner, then the codimension 3 corners, and then the codimension 2 corners. \medskip

The result of blowing up the codimension 4 corner, $C_{111}^\bullet$, is shown schematically in Figure 
 \ref{triple1}, left. Note that, as before, for the figures we pretend that $\Mbar=\R^+$, i.e.\ we leave out the $y$-variables in $\Mbar$, and we only depict the boundary faces. For the four-dimensional space $\R_+^4$ and its blow-ups this boundary is a union of three-dimensional polyhedra (which are the boundary hypersurfaces), just as for the three-dimensional space $\R_+^3$ it is a union of polygons, as in Figure \ref{fig:9}, for example.

 \begin{figure}[h]
\begin{center}
  \includegraphics[page=1,width=0.3\linewidth]{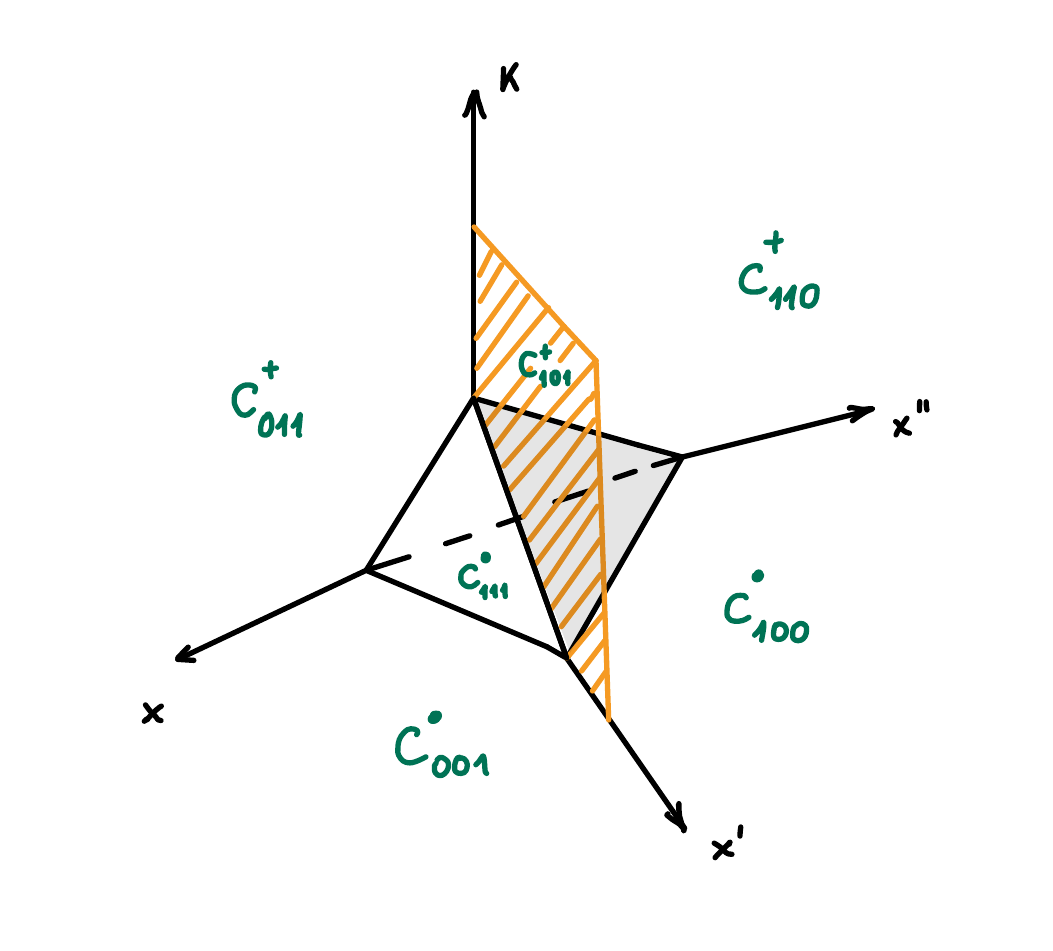}
    \includegraphics[page=2,width=0.3\linewidth]{Pictures-triple-space.pdf}
      \includegraphics[page=3,width=0.3\linewidth]{Pictures-triple-space.pdf}
  \end{center}
  \caption{Construction of $M^3_{k,\bb}$: In the left picture, only the highest codimension corner $[\Mbar^{3} \times \R^+; C_{111}^\bullet]$ is blown up. 
  In the middle picture the codimension three corners are also blown up.  In the last picture the codimension two corner $C_{110}^+$ is also blown-up.
}
  \label{triple1}
\end{figure}

\medskip

After blowing up  $C_{111}^\bullet$ the lifts of the codimension $3$ corners $C_{011}^\bullet, C_{101}^\bullet, C_{110}^\bullet$ and $C_{111}^+$ are pairwise disjoint, so we can blow them up in any order.  The result is illustrated in Figure \ref{triple1}, center.
After this, the lifts of the six codimension $2$ corners are pairwise disjoint, so we can blow them up in any order.
To get an intuitive understanding of the blowups of the codimension $2$ corners, we illustrate schematically the blowup of $C_{110}^+$
in Figure \ref{triple1}, right. This defines the b-triple space
$M^{3}_{k,\bb}$.
with the blowdown map $$\beta_{k,\bb}^{3}: M^{3}_{k,\bb}
\longrightarrow \Mbar^3\times \R^+\,.$$ 
As shown in \cite{guillarmou2008resolvent}, the projections $\pi_{L}, \pi_{C}, \pi_{R}$ lift to b-fibrations $M^{3}_{k,\bb} \to M^{2}_{k,\bb}$ 
(the latter space is constructed in Figure \ref{fig:boat1}), which are denoted by
$\pi_{\textup{b,L}}, \pi_{\textup{b,C}}$, $\pi_{\textup{b,R}}$, respectively. 

\subsubsection*{Blow up of  fibre diagonals $\textup{diag}_{k, \textup{sc},\phi}$ in $\textup{bf}$ faces}

The preimage of the interior $\interior{C_{i_2i_3}^{\ast}}$ under  the projection $\pi_{L}:\Mbar^3\times \R^+ \longrightarrow \Mbar^2\times \R^+$ is $\interior{C_{0i_2i_3}^{\ast}}\cup\interior{C_{1i_2i_3}^\ast}$ for any $i_2i_3$ and $\ast\in\{\bullet,+\}$, so if $i_2i_3\ast\neq00+$ we have for the lifted projection
(by abuse of notation, we denote the lifted corners by the same symbol again)
\begin{equation}
 \label{eqn:preimage pi_bR}
 \pi_{\bb,L}^{-1}(C_{i_2i_3}^\ast) = C_{0\,i_2i_3}^\ast \cup C_{1\.i_2i_3}^\ast\,,
\end{equation}
and the latter two hypersurfaces intersect.
Now consider $\textup{diag}_{k, \textup{sc},\phi} = C_{\phi\phi}^+\subset M^{2}_{k,\bb}$\,, 
the part of the fibre diagonal contained in $C_{11}^+=\bbf$, see \eqref{eqn:def diag k sc phi}. 
By \eqref{eqn:preimage pi_bR} and its analogues at $C$ and $R$, its preimage
under each of the projections $\pi_{\bb,*}$ is given by the union of two 
fibre diagonals:
\begin{equation}\label{diag-preimage1}\begin{split}
\pi_{\bb,L}^{-1}(\textup{diag}_{k,\textup{sc},\phi}) &= C_{0\phi\phi}^+\cup
C_{1\phi\phi}^+,\\
\pi_{\bb,C}^{-1}(\textup{diag}_{k,\textup{sc},\phi}) &= C_{\phi0\phi}^+\cup
C_{\phi1\phi}^+,\\
\pi_{\bb,R}^{-1}(\textup{diag}_{k,\textup{sc},\phi}) &= C_{\phi\phi0}^+ \cup
C_{\phi\phi1}^+\,.
\end{split}\end{equation}
Here, e.g. $C_{\phi\phi 0}^+\subset M^3_{k,\bb}$ denotes the 
intersection of the lift of the fibre diagonal $\{x=x',y=y'\}$
with $(\beta_{k,\bb}^{3})^*(C_{110}^+)$. Similarly, 
$C_{\phi\phi 1}^+\subset M^3_{k,\bb}$ denotes the 
intersection of the lift of the fibre diagonal $\{x=x',y=y'\}$
with $(\beta_{k,\bb}^{3})^*(C_{111}^+)$. This notation is 
illustrated in Figure \ref{triple4}. 

 \begin{figure}[h]
\begin{center}
  \includegraphics[page=4,width=0.8\linewidth]{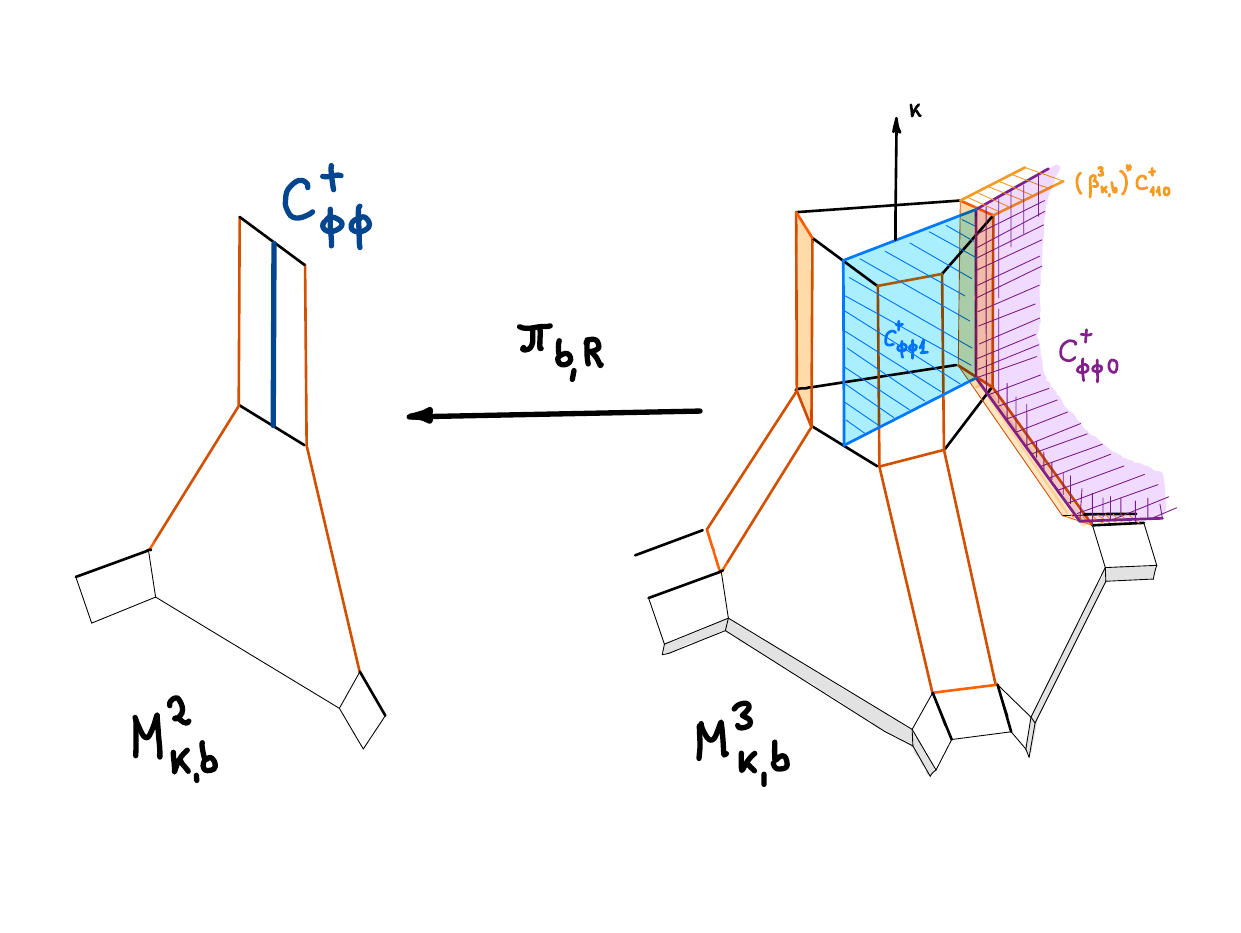}
  \end{center}
  \caption{Lifted diagonals $C_{\phi\phi 0}^+,C_{\phi\phi 1}^+$ in $M^{3}_{k,\bb}$.} \label{triple4}
\end{figure}

Let us check which of the six submanifolds on the right in \eqref{diag-preimage1} intersect non-trivially. First, the first and 
second term in each line intersect. Also, any two of the second terms intersect in the triple fibre diagonal
$$
\calO^+=
 C_{1\phi\phi}^+ \cap C_{\phi1\phi}^+ \cap C_{\phi\phi1}^+.
$$
We refrain from writing $C_{\phi\phi\phi}^+$ instead of $\calO^+$ for better readibility.
\medskip

\noindent We define the triple scattering space
by
\begin{align*}
M^{3}_{k,\ssc} : = [M^{3}_{k,\bb}; \mathcal{O}^+,\ & C_{1\phi\phi}^+, C_{\phi1\phi}^+, C_{\phi\phi1}^+,
\\ & C_{0\phi\phi}^+, C_{\phi0\phi}^+, C_{\phi\phi0}^+\,]\,,
\end{align*} 
with total blow down map $\beta^{3}_{k,\ssc} : M^{3}_{k,\ssc} \longrightarrow \Mbar^3\times \R^+$. 
Note that we blow up the fibre diagonals in $C_{111}^+$ first.\footnote{This is in fact immaterial, the other order would work just as well.}
As shown in \cite[Lemma 6.1]{guillarmou2008resolvent}, the lifted projections $\pi_{\bb,L}, \pi_{\bb,C}, \pi_{\bb,R}$ lift to b-fibrations  $\pi_{\ssc,L}, \pi_{\ssc, C}, \pi_{\ssc, R}: M^3_{k,\ssc}\to M^2_{k,\ssc}$,
respectively. This also follows by the argument in the proof of Theorem \ref{Pi b fib} below.

\subsubsection* {Blow up of fibre diagonals $\phif_0$ in $\textup{bf}_{0}$ faces}

This step is analogous to the previous one, with $k>0$ fibre diagonals replaced by $k=0$ fibre diagonals.
Consider the fibre diagonal $\textup{diag}_{k, \phi} = C_{\phi\phi}^\bullet\subset M^{2}_{k , \textup{sc},\phi}$\,, which is contained in the (lifted) boundary hypersurface $C_{11}^\bullet=\bbf_0$. Again by \eqref{eqn:preimage pi_bR} its preimage under each of the projections $\pi_{\ssc,*}$ is given by the union of two 
fibre diagonals:
\begin{align*}
\pi_{\ssc,L}^{-1}(\textup{diag}_{k, \phi}) &=  C_{0\phi\phi}^\bullet\cup
C_{1\phi\phi}^\bullet,\\
\pi_{\ssc,C}^{-1}(\textup{diag}_{k, \phi}) &= C_{\phi0\phi}^\bullet\cup
C_{\phi1\phi}^\bullet,\\
\pi_{\ssc,R}^{-1}(\textup{diag}_{k, \phi}) &= C_{\phi\phi0}^\bullet \cup
C_{\phi\phi1}^\bullet\,,
\end{align*} 
with notation analogous to before.  The hypersurfaces in the third line are
illustrated (on the level of $M^{3}_{k,\bb}$) in Figure \ref{triple5}. 

 \begin{figure}[h]
\begin{center}
  \includegraphics[page=5,width=0.8\linewidth]{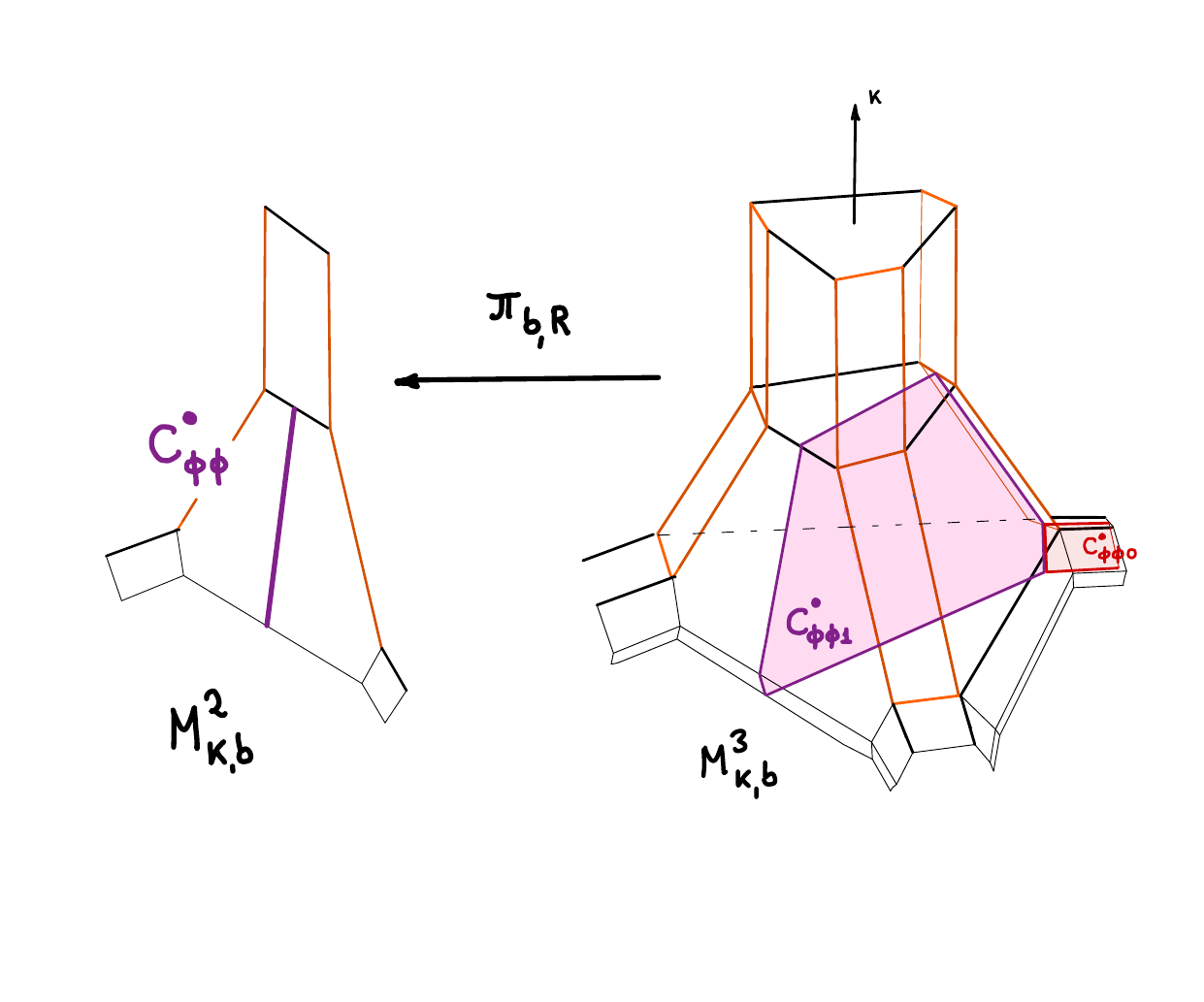}
  \end{center}
  \caption{Lifted diagonals $C_{\phi\phi 0}^\bullet,C_{\phi\phi 1}^\bullet$ in $M^{3}_{k,\bb}$.} \label{triple5}
\end{figure}

Again, intersections are non-trivial only in each line and among second terms on the right, and the latter intersect pairwise in the triple fibre diagonal at $k=0$,
$$
\mathcal{O}^\bullet = C_{1\phi\phi}^\bullet \cap C_{\phi1\phi}^\bullet \cap C_{\phi\phi1}^\bullet.
$$
The blowup of $\mathcal{O}^\bullet$ (and of $\mathcal{O}^+$) and the various 
lifted diagonals are illustrated schematically in Figure \ref{triple6}.

 \begin{figure}[h]
\begin{center}
  \includegraphics[page=6,width=0.5\linewidth]{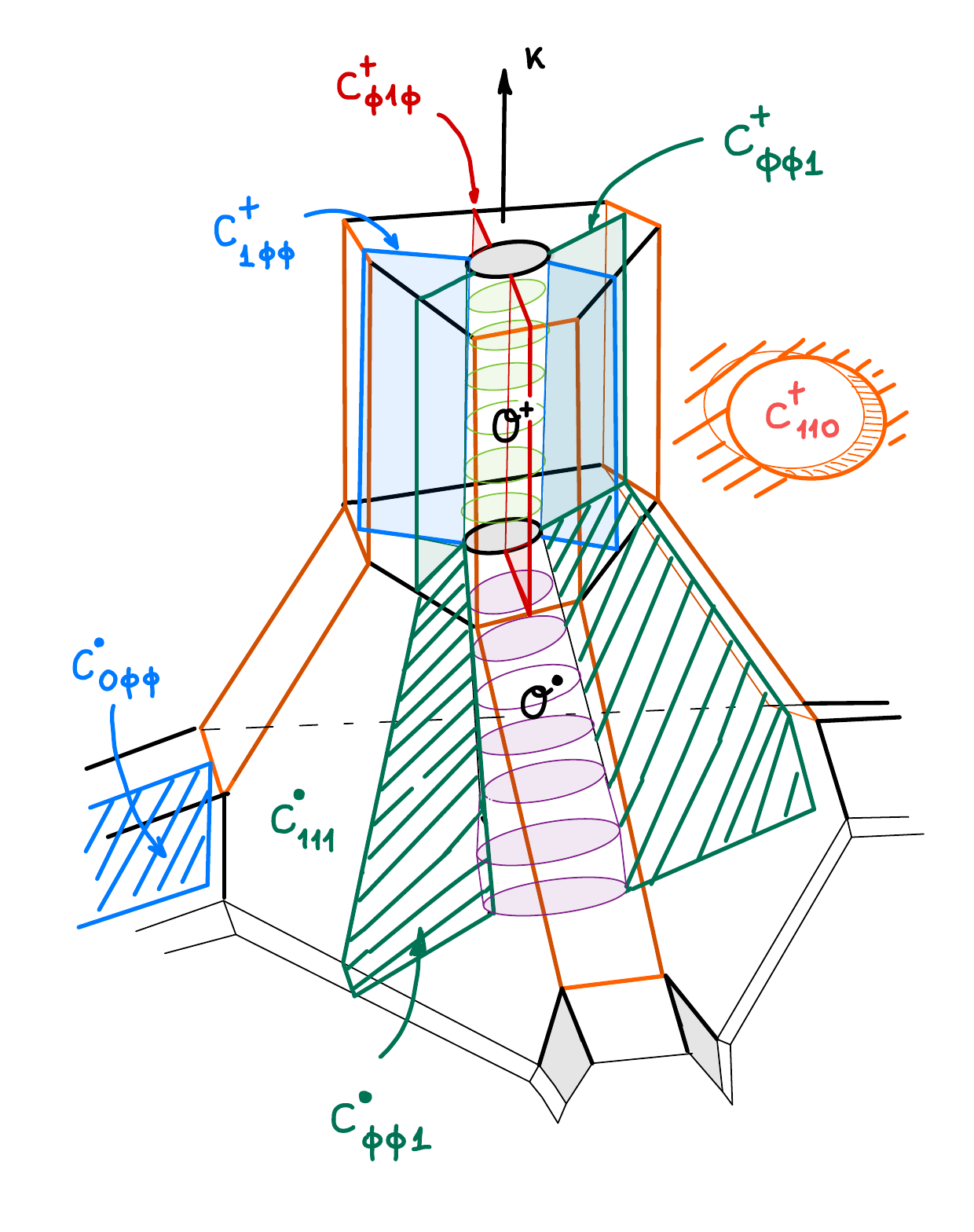}
  \end{center}
  \caption{Various lifted diagonals after blowup of $\mathcal{O}$ and $\mathcal{O}^\bullet$.} \label{triple6}
\end{figure}

We can now define the final (phi) triple space as follows
\begin{equation}\label{phi-triple-space}\begin{split}
M^{3}_{k,\phi} : =  [M^{3}_{k,\ssc}; \mathcal{O}^\bullet,\ & C_{1\phi\phi}^\bullet, C_{\phi1\phi}^\bullet, C_{\phi\phi1}^\bullet,
\\ & C_{0\phi\phi}^\bullet, C_{\phi0\phi}^\bullet, C_{\phi\phi0}^\bullet\,]\,,
\end{split}\end{equation}
with total blow down map, $\beta_{k,\phi}^{3} : M^{3}_{k,\phi} \longrightarrow \Mbar^3\times \R^+$. 

\subsection*{Step 2: lifts $\Pi_{L}, \Pi_{C}, \Pi_{R}$ exist and are b-fibrations}

\begin{Th}\label{Pi b fib}
The projections $\pi_{L}, \pi_{C}, \pi_{R}$ 
lift to b-fibrations
\begin{equation}\label{lift b-fib}
 \Pi_{L},\Pi_{C},\Pi_{R}: M^{3}_{k,\phi}
\longrightarrow M^{2}_{k,\phi}.
\end{equation}

\end{Th}
\begin{proof}
Since the spaces 
$M^{3}_{k,\phi}$ and $M^{2}_{k,\phi}$ are symmetric under permutations of the
 $M$ factors, it suffices to prove the statement only for one of the projections, e.g. for $\Pi_R$.
\medskip

As noted above,  $\pi_{R}$ lifts to a b-fibration
 $\pi_{\ssc, R}: M^{3}_{k,\ssc} \to M^{2}_{k,\ssc}$. 
We now use two standard lemmas about lifting b-fibrations. The first, Lemma 2.5 in \cite{analytic surgery}, states that $\pi_{\ssc,R}$ lifts if we blow up a closed p-submanifold $Z$ in the range $M^2_{k,\ssc}$ and also all the maximal p-submanifolds of the domain $M^3_{k,\ssc}$ which are contained in the preimage $\pi_{\ssc,R}^{-1}(Z)$. We apply this to $Z=\diag_{k,\phi}$, with preimage $C_{\phi\phi0}^\bullet\cup
C_{\phi\phi1}^\bullet$, to conclude that $\pi_{\ssc,R}$ lifts to a b-fibration
\begin{equation}
 \label{eqn:b-fib step1}
 [M^{3}_{k,\ssc}; C_{\phi\phi1}^\bullet,
C_{\phi\phi0}^\bullet] \to [M^{2}_{k,\ssc},\diag_{k,\phi}] = M^2_{k,\phi}\,.
\end{equation}

Thus it remains to blow up the space on the left to obtain $M^3_{k,\phi}$, and to check that the composed map remains a b-fibration.
For this we use Lemma 2.7 in \cite{analytic surgery} repeatedly. This lemma states that if  $f:X\to Y$ is a b-fibration and $S\subset X$ a closed p-submanifold then the composition with the blow-down map, $[X,S]\to X\stackrel{f}\to Y$, is a b-fibration again if $f(S)$ is not contained in a boundary face of $Y$ of codimension two, and if $f$ is b-transversal to $S$. These conditions can also be restated as follows: $f(S)$ is all of $Y$ or a boundary hypersurface of $Y$, and $f$ restricted to $S$ is a b-fibration $S\to f(S)$. This is easily seen to be satisfied for all blow-up centers in the sequel. \medskip

Recall from \eqref{phi-triple-space} that, to obtain $M^3_{k,\phi}$, we need to blow-up $\calO^\bullet$ first in $M^3_{k,\ssc}$. Now the three p-submanifolds $C_{\phi\phi1}^\bullet$, $C_{\phi\phi0}^\bullet$, $\calO^\bullet$ of $M^3_{k,\ssc}$ satisfy the following relations:
$$  C_{\phi\phi1}^\bullet \subset \calO^\bullet\,,\quad C_{\phi\phi0}^\bullet \cap \calO^\bullet = \emptyset\,.$$
The former holds by definition of $\calO^\bullet$, the latter holds since after
blow-up of $C_{111}$ in $M^3$ the triple fibre diagonal (on which $x=x'=x''$ holds) hits the boundary only in the interior of the $C_{111}$ front face, while $C_{\phi\phi0}$ is contained in the $C_{110}$ front face;
 compare Figure 
\ref{triple6}.
Therefore, 
 Lemma \ref{lem:commuting blow-ups}(a) implies
$$  [M^{3}_{k,\ssc}; C_{\phi\phi1}^\bullet,
C_{\phi\phi0}^\bullet,\calO^\bullet] =  [M^{3}_{k,\ssc}; \calO^\bullet, C_{\phi\phi1}^\bullet,
C_{\phi\phi0}^\bullet]\,. $$
Finally, we blow up $C_{\phi1\phi}^\bullet$, $C_{1\phi\phi}^\bullet$, $C_{\phi0\phi}^\bullet$, $C_{0\phi\phi}^\bullet$ to obtain $M^3_{k,\phi}$ (using that the former two are disjoint from $C_{\phi\phi0}^\bullet$), and the proof is complete.
\end{proof}

\subsection*{Step 3: Compute the index sets \eqref{composition-index-family}}
In order to compute the index sets for the composition we first need to determine the exponent matrices (see \eqref{hypb}) of the lifted projections $\Pi_L$, $\Pi_C$, $\Pi_R: M^3_{k,\phi}\to M^2_{k,\phi}$. That is, for each of these maps  and for each boundary hypersurface $H'$ of the target space $M^2_{k,\phi}$ we need to determine which hypersurfaces $H$ of $M^3_{k,\phi}$ are mapped to $H'$. For these the exponent $e(H,H')$ equals one, for all others it is zero. That only the exponents zero and one occur follows from the proof in step 2 and Lemmas 2.5 and 2.7 in \cite{analytic surgery} (our maps are simple in their terminology). \medskip

As an example, we determine the preimage of $H'=\bbf_0 = C_{11}^\bullet$ under the map $\Pi_L$. First, the preimage of $\bbf_0\subset M^2_{k,\bb}$ under the map $\pi_{\bb,L}$ is the union of $C_{011}^\bullet$ and $C_{111}^\bullet$, see \eqref{eqn:preimage pi_bR}.
When passing to the scattering spaces this remains unchanged since only fibre diagonals in $k>0$ are blown up.
Now consider the blow-ups leading from $M^3_{k,\ssc}$ to $M^3_{k,\phi}$. 
A new front face arising from such a blow-up, say of the submanifold $Z$, will be in the preimage of $\bbf_0$ under $\Pi_L$ if and only if $Z$ maps  \textit{onto} $\bbf_0$. This is the case precisely for the fibre diagonals $C_{\phi1\phi}^\bullet$ and $C_{\phi\phi1}^\bullet$. Summarizing, we obtain
$$ (\Pi_L)^{-1}(\bbf_0) = C_{011}^\bullet \cup C_{111}^\bullet \cup C_{\phi1\phi}^\bullet \cup C_{\phi\phi1}^\bullet \,.
$$
By similar arguments we obtain preimages as listed in  Table \ref{table:exp matrix}. \medskip

\begin{table}
\renewcommand{\arraystretch}{1.2}
$$
\begin{array}[h]{l || l | l | l}
\multirow{2}{*}{$H'\subset M^2_{k,\phi}$} & 
\multicolumn{3}{c}{\text{Preimage of $H'$ under the map}} \\
& \multicolumn{1}{c}{\Pi_L} & \multicolumn{1}{c}{\Pi_C} & \multicolumn{1}{c}{\Pi_R}
\\
\hline\hline 
 \zf = C_{00}^\bullet
 &
 C_{*00}^\bullet 
 & C_{0*0}^\bullet
 & C_{00*}^\bullet
\\
\hline
 \bbf_0=C_{11}^\bullet 
 & 
 C_{*11}^\bullet \ C_{\phi1\phi}^\bullet \ C_{\phi\phi1}^\bullet
 &
 C_{1*1}^\bullet \ C_{1\phi\phi}^\bullet \ C_{\phi\phi1}^\bullet
 &
 C_{11*}^\bullet  \ C_{\phi1\phi}^\bullet \ C_{1\phi\phi}^\bullet
\\
 \phif_0 = C_{\phi\phi}^\bullet 
 & 
 C_{*\phi\phi}^\bullet \ \calO^\bullet
 & 
 C_{\phi*\phi}^\bullet \ \calO^\bullet
 & 
 C_{\phi\phi*}^\bullet \ \calO^\bullet
\\
 \lf_0 = C_{10}^\bullet
 & C_{*10}^\bullet  \ C_{\phi\phi0}^\bullet
 & C_{1*0}^\bullet  \ C_{\phi\phi0}^\bullet
 & C_{10*}^\bullet  \ C_{\phi0\phi}^\bullet
\\
 \rf_0 = C_{01}^\bullet
 & C_{*01}^\bullet \ C_{\phi0\phi}^\bullet
 & C_{0*1}^\bullet \ C_{0\phi\phi}^\bullet
 & C_{01*}^\bullet \ C_{0\phi\phi}^\bullet
\\
\hline
 \bbf = C_{11}^+
  & 
 C_{*11}^+ \ C_{\phi1\phi}^+ \ C_{\phi\phi1}^+
 &
 C_{1*1}^+ \ C_{1\phi\phi}^+ \ C_{\phi\phi1}^+
 &
 C_{11*}^+  \ C_{\phi1\phi}^+ \ C_{1\phi\phi}^+
\\
 \ssc = C_{\phi\phi}^+
 & 
 C_{*\phi\phi}^+ \ \calO^+
 & 
 C_{\phi*\phi}^+ \ \calO^+
 & 
 C_{\phi\phi*}^+ \ \calO^+
\\
 \lf = C_{10}^+
  & C_{*10}^+  \ C_{\phi\phi0}^+
 & C_{1*0}^+  \ C_{\phi\phi0}^+
 & C_{10*}^+  \ C_{\phi0\phi}^+
\\
 \rf = C_{01}^+
  & C_{*01}^+ \ C_{\phi0\phi}^+
 & C_{0*1}^+ \ C_{0\phi\phi}^+
 & C_{01*}^+ \ C_{0\phi\phi}^+
\end{array}
$$
\caption{Relations of boundary hypersurfaces of $M^3_{k,\phi}$ and $M^2_{k,\phi}$ under the maps $\Pi_L$, $\Pi_C$ and $\Pi_R$; a $*$ means that both faces with $*=0$ and $*=1$ occur. The bottom four lines are a copy of the middle four lines, with $\bullet$ replaced by $+$.
} 
\label{table:exp matrix}
\end{table}

From Table \ref{table:exp matrix} we can read off the index family $\Cmath$, using \eqref{eqn:composition phi formula} and the index set formulas in the push-forward and the pull-back theorem. We write out the argument for $\Cmath_{\phif_0}$:
The entry for $\phif_0$ in the $\Pi_C$ column shows that the asymptotics of $ \Pi_{R}^{*} \left(\beta_{k,\phi}^* K_{A} \right) \cdot
\Pi_{L}^{*} \left( \beta_{k,\phi}^* K_{B}\right)$ at the faces $C_{\phi0\phi}^\bullet$, $C_{\phi1\phi}^\bullet$ and $\calO^\bullet$ contribute to $\calC_{\phif_0}$. 
To find the asymptotics of $\Pi_{R}^{*} \left(\beta_{k,\phi}^* K_{A} \right)$ at 
$C_{\phi0\phi}^\bullet$ we note that the entry $C_{\phi0\phi}^\bullet$ appears in column $\Pi_R$ only in the line $\lf_0$, so the asymptotics is given by $\calE_{\lf_0}$. Arguing similarly for the other factor and the other faces we find that the product has index sets
\begin{enumerate}
\item  $\mathcal{E}_{\textup{bf}_0} + \mathcal{F}_{\textup{bf}_0}$ at $C_{\phi1\phi}^\bullet$
\item $\mathcal{E}_{\lf_0} + \mathcal{F}_{\rf_0}$ at $C_{\phi0\phi}^\bullet$
\item $\mathcal{E}_{\phif_0} + \mathcal{F}_{\phif_0}$ at $\calO^\bullet$\end{enumerate}
From there we conclude (taking into account the
lifting properties of $\W _{\bb\phi}^{1/2} $ as in \cite[Theorem 9]{grieser-para})
$$
\Cmath_{\phif_0} = \left( \mathcal{E}_{\textup{bf}_0} + \mathcal{F}_{\textup{bf}_0} + (b+1)\right) 
\overline{\cup} \left( \mathcal{E}_{\lf_0} + \mathcal{F}_{\rf_0} + (b+1)\right) 
\overline{\cup} \left( \mathcal{E}_{\phif_0} + \mathcal{F}_{\phif_0} \right) 
$$
Proceeding similarly at other boundary faces we conclude the rest of \eqref{composition-index-family}.

\subsection*{Step 4: Composition of split operators}
We now prove \eqref{eqn:AB split}. 
It is useful to introduce the 'tropical' notation for operations on index sets, $E\oplus F:=E \cup F$ and $E\odot F:= E+F$. Then we have for any $ u\in \calA^E_{\textup{phg}}(M)$ and $v\in \calA^F_{\textup{phg}}(M)$
$$
u+v\in\calA^{E\oplus F}_{\textup{phg}}(M),\ 
u\cdot v \in \calA^{E\odot F}_{\textup{phg}}(M)\,.
$$
This implies a similar rule for $2\times2$ matrices of functions $\uhat=(u_{ij})$ on $M$ and of index sets $\Ewhat=(E_{ij})$ where $i,j\in\{0,1\}$: if we define $\uhat\in\calA^\Ewhat_{\textup{phg}}(M)$ to mean 
$u_{ij}\in\calA^{E_{ij}}_{\textup{phg}}(M)$ for all $i,j$ then for any 
$\uhat\in \calA^\Ewhat_{\textup{phg}}(M)$ and $\vhat\in \calA^\Fwhat_{\textup{phg}}(M)$
$$
\uhat+\vhat\in\calA^{\Ewhat\oplus \Fwhat}_{\textup{phg}}(M),\ 
\uhat\cdot\vhat \in \calA^{\Ewhat\otimes \Fwhat}_{\textup{phg}}(M),
$$
where $\otimes$ denotes the tropical matrix product, which is the usual matrix product with 
$+,\cdot$ replaced by $\oplus,\odot$, respectively. \medskip

\noindent We now turn to index families $\calE$, $\calF$ for $M^2_{k,\phi}$ and integral kernels 
$K_A\in\calA^\calE_{k,\phi}$ and $K_B\in\calA^\calF_{k,\phi}$, omitting $\beta_{k,\phi}^*$ for simplicity. Define the tropical sum 
$\calE\oplus\calF$ face by face and the $\phi$-tropical product as
\begin{equation}
\label{eqn:def odot phi} 
 \calE \odot_\phi \calF := \Cmath
\end{equation}
if $\Cmath$ is defined from $\calE$, $\calF$ as in \eqref{composition-index-family}.
Then linearity and \eqref{eqn:AB} imply
$$ K_A\in \calA^\calE_{k,\phi},\ K_B \in \calA^\calF_{k,\phi} \Longrightarrow
K_{A+B} \in \calA_{k,\phi}^{\calE\oplus\calF},\ K_{A\circ B}\in\calA_{k,\phi}^{\calE\odot_\phi \calF}\,.$$
Again, this implies for $2\times2$ matrices of kernels and index families, with notation analogous to above,
\begin{equation}
 \label{eqn:comp ring homo}
 \Kwhat_A\in \calA_{k,\phi}^\calEwhat,\ \Kwhat_B \in \calA_{k,\phi}^\calFwhat \Longrightarrow
\Kwhat_{A+B} \in \calA_{k,\phi}^{\calEwhat\oplus\calFwhat},\ \Kwhat_{A\circ B}\in\calA_{k,\phi}^{\calEwhat\otimes_\phi \calFwhat}\,.
\end{equation}
Recall the definition of the split $(k,\phi)$-calculus, Definitions \ref{def:split phi calc} and \ref{def:split k phi calculus}.
Given an index family $\calE$ for $M^2_\phi$, let $\calEwhat$ be the associated split index family, considering it as $2\times2$ matrix at each face.
Consider kernels $K_A,K_B$ supported over $\Umath\times\Umath$ (see below for the general case), and let $\Kwhat_A,\Kwhat_B$ be the associated $2\times2$ matrices of kernels as in Definition \ref{def:split phi calc}. 
Now if $K_A\in  \calA^\calE_{k,\phi,\Hmath}$, $K_{B} \in\calA^\calF_{k,\phi,\Hmath}$
then 
$$
\Kwhat_A\in\calA^\calEwhat_{k,\phi}, \quad \Kwhat_B\in\calA^\calFwhat_{k,\phi}.
$$
Thus \eqref{eqn:comp ring homo} implies $\Kwhat_{A\circ B}\in \calA^{\calEwhat\otimes_\phi \calFwhat}_{k,\phi}$.
So \eqref{eqn:AB split} will follow  if we prove the
\begin{equation}
\label{eqn:claim split}
 \text{Claim: }\quad \calEwhat \otimes_\phi \calFwhat \subset (\calE \odot_\phi\calF)\,\widehat{} \,.
 \end{equation}
This is to be understood as inclusion of index sets at each face and in each matrix component. We prove  \eqref{eqn:claim split} at the face $\bbf_0$, the proof at the other faces is analogous. We abbreviate $\bbf_0$, $\phif_0$, $\lf_0$, $\rf_0$ by $\bb$, $\phi$, $\ll$, $\rr$ respectively. For example, $\calEwhat_{01,b}$ denotes the $01$ component of $\calEwhat$ at $\bbf_0$ (which is $\calE_\bbfz+2$). \medskip

\noindent For any $i,j\in\{0,1\}$ we have by definition of $\Cmath$ in 
\eqref{composition-index-family}, with sums over $k\in\{0,1\}$,
\begin{align}
 (\calEwhat \otimes_\phi \calFwhat)_{ij,\bb}
 & = \bigoplus_{k} (\calEwhat_{ik} \odot_\phi \calFwhat_{kj})_{\bb} 
 \notag\\
 & = \bigoplus_k \left[ (\calEwhat_{ik,\bb} \odot \calFwhat_{kj,\bb}) \ocup (\calEwhat_{ik,l} \odot \calFwhat_{kj,r}) \ocup (\calEwhat_{ik,\phi} \odot \calFwhat_{kj,\bb}) \ocup
 (\calEwhat_{ik,\bb} \odot \calFwhat_{kj,\phi})
 \right] 
 \notag\\
 & \subset \left[ \bigoplus_k (\calEwhat_{ik,\bb} \odot \calFwhat_{kj,\bb}) \right] \ocup
 \left[ \bigoplus_k (\calEwhat_{ik,l} \odot \calFwhat_{kj,r}) \right] \ocup \dots
 \label{eqn:E tensor F calc}
\end{align}
where in the last line we used part (a) of the following lemma (recall that $\oplus=\cup$). For an index set $E$ denote
$$ E^b = 
\begin{pmatrix}
 E & E+2\\
 E+2 & E+4
\end{pmatrix}.
$$
\begin{Lemma}
\mbox{} 
\begin{enumerate}
 \item[(a)] If $E_1,\dots,E_N$ and $F_1,\dots,F_N$ are index sets then
 $$ (E_1\ocup\dots\ocup E_N) \cup (F_1 \ocup\dots\ocup F_N)
 \subset (E_1\cup F_1) \ocup \dots \ocup (E_N \cup F_N) $$
 \item[(b)] If $\calE$, $\calF$ are index families for $M^2_{k,\phi}$ then, with notation as introduced above,
\begin{align*}
 \calEwhat_{\bb}\otimes \calFwhat_{\bb} &= (\calE_{\bb}+\calF_{\bb})^b,\quad
 &\calEwhat_l\otimes \calFwhat_r &= (\calE_l+\calF_r)^b, \\
 \calEwhat_\phi\otimes \calFwhat_{\bb} &= (\calE_\phi+\calF_{\bb})^b,\quad 
& \calEwhat_{\bb}\otimes \calFwhat_\phi &= (\calE_{\bb}+\calF_\phi)^b.
\end{align*}
\end{enumerate}
\end{Lemma}
\noindent The identities in (b) are the core of the proof of \eqref{eqn:AB split}.
\begin{proof}
 (a) follows immediately from the definition of the extended union. For (b) we calculate
\begin{align*}
 \calEwhat_{\bb}\otimes \calFwhat_{\bb} 
 &= 
\begin{pmatrix}
 \calE_{\bb} & \calE_{\bb}+2 \\
 \calE_{\bb}+2 & \calE_{\bb}+4 
\end{pmatrix}
\otimes
\begin{pmatrix}
 \calF_{\bb} & \calF_{\bb}+2 \\
 \calF_{\bb}+2 & \calF_{\bb}+4 
\end{pmatrix}
\\ &=
\begin{pmatrix}
 \calE_{\bb}+\calF_{\bb} & \calE_{\bb}+\calF_{\bb}+2 \\
 \calE_{\bb}+\calF_{\bb}+2 & \calE_{\bb}+\calF_{\bb}+4 
\end{pmatrix}
=  (\calE_{\bb}+\calF_{\bb})^b, 
\end{align*}
where, for example, the upper left entry arises as
$(\calE_{\bb} \odot \calF_{\bb}) \oplus 
((\calE_{\bb}+2) \odot (\calF_{\bb}+2))
= 
(\calE_{\bb} + \calF_{\bb}) \cup 
((\calE_{\bb}+2) + (\calF_{\bb}+2))
= \calE_{\bb} + \calF_{\bb}$.
The other calculations are similar.
\end{proof}
We can now finish the proof of the claim \eqref{eqn:claim split}, continuing from \eqref{eqn:E tensor F calc}. By the definition of $\otimes$ and by part (2) of the lemma we have
$ \bigoplus_k (\calEwhat_{ik,\bb} \odot \calFwhat_{kj,\bb}) = (\calEwhat_{\bb}\otimes\calFwhat_{\bb})_{ij} = (\calE_{\bb}+\calF_{\bb})^b_{ij}$.
Rewriting the other terms of \eqref{eqn:E tensor F calc} similarly using the other identities in part (2) of the lemma we get
$$
(\calEwhat \otimes_\phi \calFwhat)_{\bb} \subset
(\calE_{\bb}+\calF_{\bb})^b \ocup (\calE_l+\calF_r)^b \ocup (\calE_\phi+\calF_{\bb})^b \ocup (\calE_{\bb}+\calF_\phi)^b.
$$
Since this last expression equals $(\calE\odot_\phi \calF)\,\widehat{}_{\bb}$,
 we obtain the $\bbf_0$ part of the claim \eqref{eqn:claim split}.
 The proof at the other boundary faces is analogous.\medskip

Finally, the restriction that $K_A$ was supported over $\Umath\times\Umath$ was only made to simplify the notation. One way to remove it is to write $M=\Umath\cup\Umath'$ where $\Umath'=M\setminus[(0,\eps/2)\times\dM]$, then any $K_A$ is a sum of four terms, supported over $\Umath\times\Umath$, $\Umath\times\Umath'$, $\Umath'\times\Umath$
and $\Umath'\times\Umath'$ resprectively. We have dealt with the first term; the others are treated similarly.
\end{proof}

\noindent We can now prove the general composition formula, where the operators have a conormal
singularity along the diagonal. Note that the lifted diagonal in $M^2_{k,\phi}$ intersects only the 
boundary faces $\phif_0$, $\textup{zf}$ and $\textup{sc}$.

\begin{Th}  \label{compositionth}
Consider operators 
\begin{align*} 
&A \in \Psi_{k, \phi}^{m,(a_{\phif_0},a_{\textup{zf}}, a_{\textup{sc}}), \mathcal{E}}(M; E), \\
&B \in \Psi_{k, \phi}^{m',(a'_{\phif_0},a'_{\textup{zf}}, a'_{\textup{sc}}), \mathcal{F}}(M; E).
\end{align*}
Assume that $\mathcal{E}_{\phif_0},\mathcal{E}_{\textup{zf}}, \mathcal{E}_{\textup{sc}}$ contains the 
index sets $a_{\phif_0},a_{\textup{zf}}, a_{\textup{sc}}$, respectively. Similarly, assume that 
$\mathcal{F}_{\phif_0}, \mathcal{F}_{\textup{zf}}, \mathcal{F}_{\textup{sc}}$ contains the 
index sets $a'_{\phif_0},a'_{\textup{zf}}, a'_{\textup{sc}}$, respectively. 
Then the composition of operators $A \circ B$
is well-defined with\footnote{The version of this paper published in Journal de l'Ecole Polytechnique (9) 2002 states the additional condition 'provided $\mathcal{E}_{\rf_0} + \mathcal{F}_{\lf_0} > 0$' here; but it is not needed, see the remark after Theorem \ref{compositionth-smooth}.} 
$$
A \circ B \in \Psi_{k, \phi}^{m+m',(a_{\phif_0} + a'_{\phif_0}, a_{\textup{zf}} + a'_{\textup{zf}}, a_{\textup{sc}}
+ a'_{\textup{sc}}), \, \Cmath}(M; E).
$$
Furthermore, the analogous statement holds for the split calculi:
$$
A \in\Psi^{m,(a),\calE}_{k,\phi,\calH}(M; E),\ B\in \Psi^{m',(a'),\calF}_{k,\phi,\calH}(M; E)
\ \Longrightarrow\
A\circ B \in \Psi^{m+m',(a+a'),\Cmath}_{k,\phi,\calH}(M; E)\,.
$$
The index family $\Cmath$ is given by \eqref{composition-index-family}.
\end{Th}

\begin{proof}
The statement follows from Theorem \ref{compositionth-smooth} 
exactly as in \cite[\S 6]{guillarmou2008resolvent}.
\end{proof}

\end{document}